\documentclass[11pt]{amsart}

\usepackage{amscd, amssymb}
\usepackage{enumerate}
\usepackage[all]{xypic}
\usepackage{multirow}
\usepackage{hhline}
\usepackage{verbatim}
\usepackage{mathdots}
\usepackage{comment}
\usepackage{tikz-cd}
\usepackage[titletoc]{appendix}
\usepackage{graphicx}
\usepackage{mathtools}
\usepackage{extpfeil}
\usepackage{thmtools}
\usepackage{float}

\usepackage{amsthm}
\usepackage{leftidx}
\usepackage{mathabx }
\usepackage{amsfonts}
\usepackage{bbm}

\usetikzlibrary{chains}

\allowdisplaybreaks

\makeatletter

\newcommand{\GL}{\operatorname{GL}}

\newcommand{\SO}{\operatorname{SO}}

\newcommand{\GSp}{\operatorname{GSp}}
\newcommand{\Spin}{\operatorname{Spin}}
\newcommand{\Pin}{\operatorname{Pin}}
\newcommand{\GSpin}{\operatorname{GSpin}}
\newcommand{\GPin}{\operatorname{GPin}}
\newcommand{\GSO}{\operatorname{GSO}}

\newcommand{\UU}{\operatorname{U}}
\newcommand{\OO}{\operatorname{O}}
\newcommand{\tr}{\operatorname{tr}}

\newcommand{\la}{\langle}
\newcommand{\ra}{\rangle}
\newcommand{\lla}{\langle\!\langle}
\newcommand{\rra}{\rangle\!\rangle}

\newcommand{\qand}{\quad\text{and}\quad}

\newcommand{\GPint}{\widetilde{\operatorname{GPin}}}
\newcommand{\GSpint}{\widetilde{\operatorname{GSpin}}}

\newcommand{\OOt}{\widetilde{\operatorname{O}}}
\newcommand{\OOpt}{\widetilde{\operatorname{O}'}}

\newcommand{\SOt}{\widetilde{\operatorname{SO}}}
\newcommand{\GLt}{\widetilde{\operatorname{GL}}}
\newcommand{\Ut}{\widetilde{\operatorname{U}}}
\newcommand{\Gt}{\widetilde{G}}

\newcommand{\Gb}{\overline{G}}

\newcommand{\Fcal}{\mathcal{F}}
\newcommand{\Scal}{\mathcal{S}}
\newcommand{\Ucal}{\mathcal{U}}
\newcommand{\st}{\;:\;}

\newcommand{\CC}{\mathbb{C}}
\newcommand{\C}{\mathbb{C}}

\newcommand{\sign}{\operatorname{sign}}
\newcommand{\one}{\mathbbm{1}}

\theoremstyle{plain}
\newtheorem{theorem}{Theorem}[section]
\newtheorem{cor}[theorem]{Corollary}
\newtheorem{prop}[theorem]{Proposition}
\newtheorem{lemma}[theorem]{Lemma}
\newtheorem{rem}[theorem]{Remark}

\newtheorem{defn}[theorem]{Definition}

\numberwithin{equation}{section}

\DeclareMathOperator{\disc}{disc}
\DeclareMathOperator{\Ind}{Ind}
\DeclareMathOperator{\Span}{Span}

\DeclareMathOperator{\Hom}{Hom}

\DeclareMathOperator{\Aut}{Aut}
\DeclareMathOperator{\Irr}{Irr}

\tikzset{
  symbol/.style={
    draw=none,
    every to/.append style={
      edge node={node [sloped, allow upside down, auto=false]{$#1$}}}
  }
}

\author{Melissa Emory }
\address{Department of Mathematics, Oklahoma State University, 401 MSCS, Stillwater OK 74075, USA}
\email{melissa.emory@okstate.edu}

\author{Shuichiro Takeda}
\address{Department of Mathematics, Graduate School of Science, Osaka University; Toyonaka, Osaka 560-0043, JAPAN }
\email{takedas@math.sci.osaka-u.ac.jp}

\title[Contragredients and Multiplicity One for GSpin groups]{Contragredients and a Multiplicity One Theorem\\ for General Spin Groups}
\makeatother

\begin{document}
\maketitle

\begin{abstract}
Each orthogonal group $\OO(n)$ has a nontrivial $\GL(1)$-extension, which we call $\GPin(n)$. The identity component of $\GPin(n)$ is the more familiar $\GSpin(n)$, the general Spin group. We prove that the restriction to $\GPin(n-1)$ of an irreducible admissible representation of $\GPin(n)$ over a nonarchimedean local field of characteristic zero is multiplicity free and also prove the analogous theorem for $\GSpin(n)$. Our proof uses the method of Aizenbud, Gourevitch, Rallis and Schiffman, who proved the analogous theorem for $\OO(n)$, and of Waldspurger, who proved that for $\SO(n)$.

We also give an explicit description of the contragredient of an irreducible admissible representation of $\GPin(n)$ and $\GSpin(n)$, which is needed to apply their method to our situations.
\end{abstract}

\section{Introduction}

Let $G$ be a reductive group over a nonarchimedean local field $F$ of characteristic zero, and $H\subseteq G$ a reductive subgroup. Let $\pi$ and $\tau$ be irreducible admissible representations of $G$ and $H$, respectively. One of the important questions in representation theory is to know ``how may times" $\tau$ appears as a quotient of $\pi$, when $\pi$ is restricted to $H$, namely to know the dimension
\[
\dim_{\C}\Hom_H(\pi,\, \tau),
\]
where, strictly speaking, $\pi$ in the Hom space is $\pi|_H$. There seem to be two separate questions: whether the dimension is nonzero (nonvanishing question) such as in the Gan-Gross-Prasad conjecture (\cite{GGP}), and whether the dimension is at most one (multiplicity-at-most-one or multiplicity-free question). The latter question is formulated as
\[
\dim_{\C}\Hom_H(\pi,\,\tau)\leq 1,
\]
and we call an assertion of this form a ``multiplicity-at-most-one theorem".

In their celebrated paper \cite{AGRS}, Aizenbud, Gourevitch, Rallis and Schiffmann proved the multiplicity-at-most-one theorem for the pairs
\[
(G, H)=(\GL(n), \GL(n-1)),\ (U(n), U(n-1))\qand (\OO(n), \OO(n-1)),
\]
and later Waldspurger \cite{Wal12} proved the case for
\[
(G, H)=(\SO(n), \SO(n-1)).
\]
Also the Archimedean case was proven by Sun and Zhu in \cite{Sun_Zhu_Archmedean}.

The purpose of this paper is to prove the analogous theorem for two non-classical groups: the general Spin group ($\GSpin$) and what we call the general Pin group ($\GPin$), namely for
\[
(G, H)=(\GSpin(n), \GSpin(n-1)) \qand (\GPin(n), \GPin(n-1)).
\]

Let us first recall some generalities of these groups.  Let $(V,q)$ be a (nondegenerate) quadratic space over a nonarchimedean local field $F$ of characteristic 0 with dimension $n$.  The general Pin group and the general Spin group associated with $(V, q)$, which we denote by $\GPin(V)$ and $\GSpin(V)$, respectively, are reductive groups over $F$ such that we have the following commutative diagram
\[
\begin{tikzcd}
1 \arrow[r] & Z^\circ\arrow[r]&\GPin (V)\arrow[r, "P"]& \OO(V) \arrow[r] &1\\
1 \arrow[r] & Z^\circ\ar[u, symbol={=}] \arrow[r]&\GSpin (V)\ar[u, symbol={\subseteq}]\arrow[r, "P"]& \SO(V)\ar[u, symbol={\subseteq}] \arrow[r] &1\rlap{\, ,}
\end{tikzcd}
\]
where $Z^\circ\simeq\GL_1$ is the connected component of the center of both $\GPin(V)$ and $\GSpin(V)$. Here we call the surjection
\[
P:\GPin(V)\longrightarrow\OO(V)
\]
the canonical projection, which is to be defined in \eqref{eq:canonical_projection}.

Assume $W\subseteq V$ is a nondegenerate subspace of dimension $n-1$. Then there are natural inclusions
\[
\GPin(W)\subseteq\GPin(V)\qand\GSpin(W)\subseteq\GSpin(V),
\]
where the centers of all the groups share the same connected component, namely $Z^\circ$.

The main theorem of the paper is the following multiplicity-at-most-one theorem.
\begin{theorem}\label{thm:A}
Let
\[
(G, H)=(\GPin(V), \GPin(W))\quad\text{or}\quad (\GSpin(V), \GSpin(W)).
\]
For all $\pi\in\Irr(G)$ and $\tau\in\Irr(H)$, we have
\[
\dim_{\mathbb{C}}\Hom_{H}(\pi,\, \tau)\leq 1.
\]
Note that if the central characters of $\pi$ and $\tau$ do not agree on the connected component $Z^\circ$ of the center, then the Hom space is automatically zero.
\end{theorem}

Let us explain the basic idea of our proof, which basically follows the proof in \cite{AGRS}, though we need to make numerous modifications. We consider the space of $\GPin(W)$-invariant distributions on $\GPin(V)$, namely
\[
\Scal'(\GPin(V))^{\GPin(W)}.
\]
We find an involution $\sigma$ on $\GPin(V)$ such that $\sigma(\GPin(W))=\GPin(W)$, so that $\sigma$ acts on $\Scal'(\GPin(V))^{\GPin(W)}$. Since $\sigma^2=1$, this space decomposes as
\[
\Scal'(\GPin(V))^{\GPin(W), +}\oplus \Scal'(\GPin(V))^{\GPin(W), -},
\]
where the first space is the $+1$-eigenspace and the second one the $-1$-eigenspace. By using the argument from \cite{AGRS}, we will show that our main theorem is reduced to the vanishing of the $-1$-eigenspace, namely
\[
\Scal'(\GPin(V))^{\GPin(W), -}=0.
\]
To show this vanishing, we use the Frobenius descent and Bernstein's localization principle in the same way as \cite{AGRS}. The key point is to prove the existence of an involution
\[
\sigma_V:\GPin(V)\longrightarrow\GPin(V)
\]
which preserves the semisimple conjugacy classes; namely for each semsimple $g\in\GPin(V)$ the two elements $g$ and $\sigma_V(g)$ are conjugate in $\GPin(V)$. This is also the crucial fact used in \cite{AGRS} and \cite{Wal12}. But a notable difference is that in \cite{AGRS} they argue inductively on $\dim_FV$ whereas we reduce the above vanishing assertion to the classical group situation of \cite{AGRS} and invoke their results, so we do not argue inductively. We similarly argue for $\GSpin(V)$ by using a slightly different involution. This case is analogous to the special orthogonal case of Waldspurger \cite{Wal12}.

\quad

On the way of proving our main theorem (Theorem \ref{thm:A}), we prove the following theorem on the contragredient, which is also of independent interest.
\begin{theorem}\label{thm:contragredient_introduction}
For $\pi\in\Irr(\GPin(V))$, we have
\[
\pi^\vee\simeq\begin{cases}\omega_{\pi}^{-1}\otimes\pi,&\text{if $n=2k$};\\\sign_{\pi}^{k}\omega_{\pi}^{-1}\otimes\pi,&\text{if $n=2k-1$},\end{cases}
\]
where $\omega_\pi$ is the central character of $\pi$ and $\sign_{\pi}$ is the sign character of $\pi$.

For $\pi\in\Irr(\GSpin(V))$, we have
\[
\pi^\vee\simeq\begin{cases}\omega_{\pi}^{-1}\otimes\pi,&\text{if $n=2k$ with $k$ even, or $n=2k-1$};\\ \omega_{\pi}^{-1}\otimes\pi^\delta,&\text{if $n=2k$ with $k$ odd},\end{cases}
\]
where $\pi^\delta$ is the representation obtained by twisting $\pi$ by any $\delta\in\GPin(V)\smallsetminus\GSpin(V)$.
\end{theorem}

Here the character twist $\omega_{\pi}\otimes\pi$ is via the Clifford norm $N:\GPin(V)\to F^\times$ defined in \eqref{eq:Clifford_norm}, and the sign character $\sign_\pi$ is defined in \eqref{eq:sign_of_pi}.

The theorem is proven in Theorem \ref{thm:contragredient_GPin}, and follows from the existence of the involution $\sigma_V$; namely since $\sigma_V(g)$ and $g$ are conjugate for all semisimple $g\in\GPin(V)$, Harish-Chandra's regularity theorem implies that $\pi^\vee$ is equivalent to the representation $\pi^\sigma$ defined by $\pi^\sigma(\sigma_V(g)^{-1})$. But one can write out $\pi^\sigma(\sigma_V(g)^{-1})$ explicitly as in the theorem. The case for $\GSpin(V)$ is essentially the same.

\quad

Let us give a brief discussion on the question ``why work on $\GSpin$?".  Probably, there are numerous reasons to pay attention to $\GSpin$ or $\GPin$. First, since $\GSpin(V)$ is a $\GL_1$-extension of $\SO(V)$, a representation of $\GSpin(V)$ with the trivial central character factors through $\SO(V)$ and every representation of $\SO(V)$ arises in this way; namely the representation theory of $\SO(V)$ is completely subsumed under that of $\GSpin(V)$. Second, from the point of view of the Langlands program, the dual group of $\GSpin(V)$ is either $\GSO_{2n}(\C)$ or $\GSp_{2n}(\C)$, which is naturally viewed as a subgroup of $\GL_{2n}(\C)$. For this reason, compared to other nonclassical groups, it seems $\GSpin$ is more susceptible of Langlands functoriality or of endoscopic classification as is done in \cite{Asgari-Shahidi} and \cite{Gee-Taibi}. Finally, even for arithmetic applications, as mentioned in \cite{Madapusi}, an orthogonal Shimura variety is of abelian type but is a finite etale quotient of a GSpin Shimura variety. Accordingly, results on orthogonal Shimura varieties can be easily derived from the GSpin counterparts. We believe that these already provide good enough reasons to study the groups $\GSpin$ and $\GPin$.

\quad

The following is the overall structure of the paper. In the next section (Section 2), we review and establish necessary facts about the groups $\GPin(V)$ and $\GSpin(V)$ and define an important involution $\sigma_V$. In Section 3, we prove the important fact that $\sigma_V(g)$ and $g$ are conjugate for semisimple $g\in\GPin(V)$. The key idea is to analyze the structure of the centralizer $\OO(V)_{P(g)}$ of $P(g)$ in the orthogonal group $\OO(V)$. In Section 4, we compute the contragredient $\pi^\vee$ of $\pi\in\Irr(\GPin(V))$. In Section 5, we introduce the group $\GPint(V)$, which is the group we work on when we consider the conjugation action of $\GPin(V)$ on itself. We also review the analogous groups for the classical groups considered in \cite{AGRS}. In Section 6, we reduce our main theorem (Theorem \ref{thm:A}) for $\GPin(V)$ to the vanishing of invariant distributions as discussed above. In Section 7, we reduce the vanishing assertion for $\GPin(V)$  to the classical group cases of \cite{AGRS}. In Section 8, we finish the proof for $\GPin(V)$ by proving the necessary vanishing assertion for the classical group cases. In Section 9, we treat the case for $\GSpin(V)$. We finish the paper with two appendices. In Appendix \ref{Appendix_A}, we prove the known theorem for the structure on the centralizer $\OO(V)_h$ of a semisimple $h\in\OO(V)$. This is well-known for decades but we reproduce the proof here because we have not been able to locate the theorem stated in the precise form we need. In Appendix \ref{Appendix_B}, we give a summary of all the involutions that we use in this paper.

\quad

\begin{center}Notation and Terminology\end{center}

Let us summarize our basic notation and terminology in this paper. We assume that $F$ is a nonarchimedean local field of characteristic zero. For a locally compact totally disconnected (lctd) group $G$ we denote by $\Irr(G)$ the set of (equivalence classes of) irreducible admissible representations of $\pi$. For each $\pi\in\Irr(G)$ we denote by $\pi^\vee$ the contragredient and by $\omega_{\pi}$ the central character of $\pi$.

For a lctd space $X$, we denote by $\Scal'(X)$ the space of distributions on $X$, which is by definition the space of linear functionals on the Schwartz space $\Scal(X)$ on $X$.

By an involution $\sigma$ on an $F$-algebra $A$, we mean an $F$-linear map on $A$ such that $\sigma^2=1$ and $\sigma(ab)=\sigma(b)\sigma(a)$ for all $a, b\in A$, which is sometimes called an anti-involution in the literature. Also by an involution on a group $G$, we mean a map $\sigma$ on $G$ such that $\sigma^2=1$ and $\sigma(gh)=\sigma(h)\sigma(g)$ for all $g, h\in G$.

Suppose a group $G$ acts on a set $X$. We write $G_x$ for the stabilizer of $x\in X$ in $G$. In particular, when $G$ acts on $G$ by conjugation, $G_g$ is the centralizer of $g\in G$.

Unless otherwise stated, by $(V, q)$ or simply by $V$ we mean a nondegenerate quadratic space over our local field $F$, and $\la-,-\ra$ the corresponding symmetric bilinear form, namely
\[
\la v, v'\ra=\frac{1}{2}\big(q(v+v')-q(v)-q(v')\big)
\]
for $v, v'\in V$. Also we set
\[
n=\dim_FV,
\]
and often write
\[
n=\begin{cases}2k\\2k-1\end{cases}
\]
accordingly as $n$ is even or odd. Note that
\[
k=\left[\frac{\dim_FV+1}{2}\right].
\]
We often write $\{e_1,\dots, e_n\}$ for an orthogonal basis of $V$, and assume
\[
W=\Span\{e_1,\dots, e_{n-1}\}, \quad \text{so that}\quad V=W\oplus Fe
\]
by setting $e=e_n$. For each anisotropic $v\in V$, we write
\[
r_v\in\OO(V)
\]
for the reflection in the hyperplane orthogonal to $v$, so in particular $r_v(v)=-v$ and $\det(r_v)=-1$.

If $V$ is a vector space over a field $A\supseteq F$, we write $\GL_A(V)$ for the general linear group over $A$ when to emphasize the field $A$. Also if $V$ is equipped with a Hermitian structure for a quadratic extension $A/A'$, we write $U_A(V)$ for the corresponding unitary group.

\quad

\begin{center}{\bf Acknowledgements}\end{center}

The first named author was partially supported by an AMS-Simons Travel Award, the NSF grant DMS-2002085, and
FY2019 JSPS Postdoctoral Fellowship for Research in Japan (short term). The second named author was partially supported by the Simons Foundations
Collaboration Grant \#584704.

This project was initiated while both of the authors were attending the conference ``On the Langlands Program: Endoscopy and Beyond" from Dec 2018 to Jan 2019 at the Institute for Mathematical Sciences in Singapore, and part of the research was done while they were attending the Oberwolfach workshop  ``New developments in representation theory of $p$-adic groups" in October 2019. We would like to thank their hospitality.

The first named author would like to thank Hiraku Atobe for helpful conversations and his hospitality while hosting her for a week at Hokkaido University, and would like to thank Kyoto University for their hospitality while hosting her fellowship during the summer of 2019, and thank Atsushi Ichino for his interest in this project.

Also part of the paper was completed while the second named author was visiting the National University of Singapore in spring 2020, and he would like to thank their
hospitality, and would like to thank Wee Teck Gan for his interest in this project.

Lastly, the authors would like thank the anonymous referee for his/her helpful comments.

%He would like to thank Martin Weissman for explaining his idea to count the number of $\GL_1$-extensions.

\section{The groups $\GPin(V)$ and $\GSpin(V)$}
We recall the definitions of $\GPin(V)$ and $\GSpin(V)$ and establish some of their properties we need. (In this section the field $F$ does not have to be our nonarchimedean local field but can be any field of characteristic not equal to 2.)

\subsection{Clifford algebra}
Let
\[
T(V):=\bigoplus_{\ell=0}^{\infty}V^{\otimes \ell}=F\oplus V\oplus V^{\otimes 2}\oplus\cdots
\]
be the tensor algebra of $V$. We have the natural inclusion $V\xhookrightarrow{\;} T(V)$. We define the Clifford algebra $C(V)$ by
\[
C(V):=T(V)\slash \langle v \otimes v - q(v) \cdot 1 \st v\in V\rangle,
\]
which is an associative $F$-algebra. The natural inclusion $V\xhookrightarrow{\;} T(V)$ gives the natural inclusion $V\xhookrightarrow{\;} C(V)$. Note that in $C(V)$ we have
\[
v\cdot v=q(v)\in F
\]
for all $v\in V$.

We denote by $C^\ell(V)$ the image of $V^{\otimes \ell}$ in $C(V)$. Though the Clifford algebra is not a direct sum of $C^\ell(V)$'s, it is a direct sum of even terms and odd terms; namely
\[
C(V)=C^+(V)\oplus C^-(V),
\]
where
\[
C^+(V)=\sum_{\ell\; \text{even}}C^\ell(V)\qand C^-(V)=\sum_{\ell\; \text{odd}}C^\ell(V).
\]
Also note that we actually have $C=\sum_{\ell=0}^nC^n(V)$ because for $\ell>n$ any element in $C^\ell(V)$ is written as a sum of lower degree terms.

It is known that $\dim_FC(V)=2^n$ and $\dim_FC^{\pm}(V)=2^{n-1}$. Note that $C^+(V)$ is a subalgebra of $C(V)$, called the even Clifford algebra. Both $C(V)$ and $C^+(V)$ are central simple algebras central over $F$ or over the quadratic etale algebra $F[x]\slash (x^2-d_V)$ where $d_V$ is the discriminant of $V$. (See \cite[2.10 Theorem, p.332]{Sch85} or \cite[Theorem 2.8, p.19]{Shimura}.)

The Clifford algebra is equipped with the natural involution $\ast$ by ``reversing the indices" of $v_1v_2\cdots v_\ell\in C^\ell(V)$, namely
\begin{equation}\label{eq:canonical_involution}
(v_1v_2\cdots v_\ell)^*=v_\ell v_{\ell-1}\cdots v_1
\end{equation}
for $v_i\in V$. This involution is called the canonical involution.  Certainly the canonical involution preserves both $C^+(V)$ and $C^-(V)$. Also we define
\begin{equation}\label{eq:alpha_involution}
\alpha:C(V)\longrightarrow C(V),\quad \alpha(x_++x_-)=x_+-x_-,
\end{equation}
where $x_+\in C^+(V)$ and $x_-\in C^-(V)$; namely $\alpha$ acts as the identity on $C^+(V)$ and as multiplication by $-1$ on $C^-(V)$. Then for all $x\in C(V)$ we define
\begin{equation}\label{eq:Clifford_involution}
\overline{x}=\alpha(x)^*=\alpha(x^*),
\end{equation}
which is called the Clifford involution. The map $x\mapsto\overline{x}$ is an involution on $C(V)$ and the map
\[
N:C(V)\longrightarrow C(V),\quad x\mapsto x\overline{x},
\]
is called the Clifford norm.

Let us mention the following easy lemma.
\begin{lemma}\label{lemma:orthogonal_vectors}
Let $v_1, v_2\in V$. Then in $C(V)$ we have
\[
v_1\cdot v_2=-v_2\cdot v_1+2\la v_1, v_2\ra.
\]
Hence in particular if $v_1$ and $v_2$ are orthogonal then $v_1\cdot v_2=-v_2\cdot v_1$.
\end{lemma}
\begin{proof}
This follows from $q(v_1+v_2)=(v_1+v_2)\cdot (v_1+v_2)$.
\end{proof}

\subsection{The groups $\GPin(V)$ and $\GSpin(V)$}
We can now define the groups $\GPin(V)$ and $\GSpin(V)$ as follows.
\begin{defn}
We define
\begin{align*}
\GPin(V)&:=\{g \in C(V)^{\times}\st \alpha(g)Vg^{-1}=V\};\\
\GSpin(V)&:=\{g \in C^+(V)^{\times}\st gVg^{-1}=V\},
\end{align*}
and call $\GPin(V)$ the general Pin group of $V$ and $\GSpin(V)$ the general Spin group of $V$.
\end{defn}

\begin{rem}
Sometimes in the literature, the group $\GPin(V)$ is called the Clifford group and $\GSpin(V)$ the special Clifford group, and denoted, respectively, by $\Gamma(V)$ and $S\Gamma(V)$ (or some other symbols). But we avoid this terminology because in representation theory of $p$-adic groups or in automorphic forms it seems to be more common to call $\GSpin(V)$ the general Spin group. To the best of our knowledge the notation $\GPin(V)$ and the name ``general Pin" have never been used in the literature but we have decided to use them because of their connection to the group called $Pin$.
\end{rem}

Since the map $\alpha$ is trivial on $C^+(V)$, we have the inclusion
\[
\GSpin(V)\subseteq\GPin(V).
\]
Note that $[\GPin(V) : \GSpin(V)]=2$. (See \cite[Theorem 3.7, p.23]{Shimura}.) Indeed, $\GPin(V)$ is not connected as an algebraic group and $\GSpin(V)$ is the identity component.

In the definition of $\GPin(V)$, the presence of $\alpha$ is crucial. To see it, for each $g\in\GPin(V)$ let us define
\begin{equation}\label{eq:canonical_projection}
P(g):V\longrightarrow V,\quad P(g)v=\alpha(g)vg^{-1}.
\end{equation}
We then have the short exact sequence
\[
1\xrightarrow{\quad} F^\times\xrightarrow{\quad}\GPin(V)\xrightarrow{\;P\;}\OO(V)\xrightarrow{\quad}1.
\]
Note that if $v\in V\subseteq C(V)$ is anisotropic, then $v\in\GPin(V)$ and $P(v):V\to V$ is the reflection in the hyperplane orthogonal to $v$, namely
\[
P(v)=r_v,
\]
where we recall the notation $r_v$ from the notation section. (See \cite[3.3 Theorem, p.225]{Sch85}.) We call the projection
\[
P:\GPin(V)\longrightarrow\OO(V)
\]
the canonical projection. (If one defines the group $\GPin(V)$ without $\alpha$, the corresponding map $\GPin(V)\to\OO(V)$ fails to be surjective.) Since $P$ is surjective and both $\GPin(V)$ and $\OO(V)$ have two connected components, we see that $P^{-1}(\SO(V))=\GSpin(V)$.

To wrap up, we have the following commutative diagram
\[
\begin{tikzcd}
1 \arrow[r] & \GL_1\arrow[r]&\GPin (V)\arrow[r,"P"]& \OO(V) \arrow[r] &1\\
1 \arrow[r] & \GL_1\ar[u, symbol={=}] \arrow[r]&\GSpin (V)\ar[u, symbol={\subseteq}]\arrow[r,"P"]& \SO(V)\ar[u, symbol={\subseteq}] \arrow[r] &1\rlap{\, ,}
\end{tikzcd}
\]
where the rows are exact. (To be precise, all the maps are morphisms of algebraic groups, and the rows are exact even for the $F$-rational points.) We set
\[
Z^\circ:=\ker P\simeq\GL_1.
\]

One can show that the restriction of the Clifford norm $N$ to $\GPin(V)$ has its image in $F^\times$, which gives a homomorphism
\begin{equation}\label{eq:Clifford_norm}
N:\GPin(V)\longrightarrow F^\times,\quad g\mapsto g\bar{g},
\end{equation}
which we again call the Clifford norm. Note that
\begin{equation}\label{eq:Norm_on_Z}
N(z)=z^2
\end{equation}
for all $z\in Z^\circ$.

Let us mention the following lemma, which says that each $g\in\GPin(V)$ is ``homogeneous".
\begin{lemma}\label{lemma:homogeneous}
Let $g\in\GPin(V)$. Then there exists anisotropic vectors $v_1,\dots, v_\ell$ such that
\[
g=v_1\cdots v_\ell,
\]
and in particular $g\in C^\ell(V)$. (Note that neither the vectors $v_1,\dots, v_\ell$ nor the degree $\ell$ are unique.)

Hence
\[
\GSpin(V)=\GPin(V)\cap C^+(V)\qand\GPin(V)\smallsetminus\GSpin(V)=\GPin(V)\cap C^-(V).
\]
\end{lemma}
\begin{proof}
For each anisotropic $v\in V$, let us write $r_v\in\OO(V)$ for the reflection in the hyperplane orthogonal to $v$. It is well-known that each element in $\OO(V)$ is a product of some $r_v$'s. Hence $P(g)\in\OO(V)$ is written as $P(g)=r_{v_1}\cdots r_{v_\ell}$ for some anisotropic $v_i$'s. Since an anisotropic vector $v\in V$ is in $\GPin(V)$ and $P(v)=r_v$ (\cite[3.3 Theorem, p.225]{Sch85}), we know that $g=zv_1\cdots v_\ell$ for some $z\in Z^\circ=F^\times$. The first assertion of the lemma follows. The second assertion immediately follows from the first one.
\end{proof}

\subsection{Sign character}
Let
\begin{equation}\label{eq:sign_character}
\sign:\GPin(V)\longrightarrow\{\pm 1\}
\end{equation}
be the homomorphism which sends the nonidentity component to $-1$, so that its kernel is $\GSpin(V)$. Lemma \ref{lemma:homogeneous} implies
\[
\sign=\alpha|_{\GPin(V)}.
\]
In particular, for all $g\in\GPin(V)$ we have
\[
\overline{g}=\sign(g) g^* \qand N(g)=\sign(g)gg^*,
\]
where we recall $g^*$ is the canonical involution.

\subsection{The centers}
Let us describe the centers of $\GPin(V)$ and $\GSpin(V)$. For this purpose, let $\{e_1,\dots,e_n\}$ be an orthogonal basis of $V$, and set
\begin{equation}\label{eq:shimura_zeta}
\zeta=e_1\cdots e_n.
\end{equation}
Then $\zeta$ has the following properties.
\begin{lemma}\label{lemma:shimura_zeta}
\quad
\begin{enumerate}[(a)]
\item $\alpha(\zeta)v\zeta^{-1}=-v$ for all $v\in V$, and hence $P(\zeta)=-1$;
\item
\[
\zeta^*=(-1)^{\frac{1}{2}n(n-1)}\zeta=\begin{cases}(-1)^k\zeta&\text{if $n=2k$};\\ (-1)^{k+1}\zeta&\text{if $n=2k-1$};\end{cases}
\]
\item $F\zeta$ is independent of the choice of the orthogonal basis;
\item $F+F\zeta$ is the center of $C(V)$ or of $C^+(V)$ according as $n$ is odd or even, respectively.
\end{enumerate}
\end{lemma}
\begin{proof}
See \cite[Lemma 2.7, p.19]{Shimura}.
\end{proof}

Clearly,
\begin{equation}\label{eq:where_zeta_lives}
\zeta\in\begin{cases}\GSpin(V)&\text{if $n=2k$}\\
\GPin(V)\smallsetminus\GSpin(V)&\text{if $n=2k-1$}.
\end{cases}
\end{equation}
We then have the following.
\begin{prop}\label{prop:center_of_Gpin}
Let $Z_{\GPin(V)}$ and $Z_{\GSpin(V)}$ be the centers of $\GPin(V)$ and $\GSpin(V)$, respectively.
\begin{enumerate}[(a)]
\item Assume $n=2k>2$. Then
\[
Z_{\GPin(V)}=F^\times\qand Z_{\GSpin(V)}=F^\times\cup F^\times\zeta.
\]
If $n=2$ then
\[
Z_{\GPin(V)}=F^\times\qand Z_{\GSpin(V)}=\GSpin(V).
\]
\item Assume $n=2k-1$. Then
\[
Z_{\GPin(V)}=F^\times\cup F^\times\zeta\qand Z_{\GSpin(V)}=F^\times.
\]
\end{enumerate}
In particular, $Z^\circ=\ker P$ is the connected component of the center of $\GPin(V)$ as well as that of $\GSpin(V)$.
\end{prop}
\begin{proof}
See \cite[Theorem 3.6, p.22]{Shimura}.
\end{proof}

Let us note that in the above, when $n=2$, we know that $\GSpin(V)$ itself is already commutative; to be precise $\GSpin(V)$ is the multiplicative group of the etale quadratic algebra $F[x]\slash(x^2-d_V)$, where $d_V$ is the discriminant of $V$. Also when $n=1$, we have $\GSpin(V)=F^\times$ and $\GPin(V)=F^\times\cup F^\times\zeta$.

For $n=2k>2$, though $\zeta$ is not in the center of $\GPin(V)$, it is not so far from it as follows.
\begin{lemma}\label{lemma:zeta_almost_commutes_V_even}
Assume $n=2k>2$. We have
\[
g\zeta=-\zeta g\quad\text{for all $g\in\GPin(V)\smallsetminus\GSpin(V)$}.
\]
\end{lemma}
\begin{proof}
By Lemma \ref{lemma:homogeneous}, each $g\in\GPin(V)\smallsetminus\GSpin(V)$ is written as $g=v_1\cdots v_{\ell}$ for some anisotropic vectors $v_1,\dots, v_{\ell}$, where $\ell$ is odd. Since $\dim_FV$ is even, by \eqref{eq:where_zeta_lives} we have $\alpha(\zeta)=\zeta$, where $\alpha$ is as in \eqref{eq:alpha_involution}. Hence for each $v_i$ we have $\zeta v_i\zeta^{-1}=-v_i$ by Lemma \ref{lemma:shimura_zeta} (a). Hence
\[
\zeta g\zeta^{-1}=\zeta v_1\cdots v_{\ell}\zeta^{-1}=(-1)^{\ell}v_1\cdots v_{\ell}=-g.
\]
The lemma is proven.
\end{proof}

It should be emphasized that the element $\zeta$ plays important roles in many parts of this paper.

\subsection{Involution $\sigma_V$}
Let us define an involution $\sigma_V$ on $\GPin(V)$ by
\begin{equation}\label{eq:involuions_sigma_n}
\sigma_V(g)=\begin{cases}g^*&\text{if $n=2k$};\\
\sign(g)^{k+1}g^*&\text{if $n=2k-1$}\end{cases}
\end{equation}
for $g\in\GPin(V)$, where we recall that $g^*$ is the canonical involution defined in \eqref{eq:canonical_involution} and $\sign$ is the sign character as defined in \eqref{eq:sign_character}. In particular, for $n=2k-1$ we have
\[
\sigma_V(g)=\begin{cases}g^*&\text{if $g\in\GSpin(V)$};\\
(-1)^{k+1}g^*&\text{if $g\in\GPin(V)\smallsetminus\GSpin(V)$};\end{cases}
\]
namely $\sigma_V$ is the canonical involution if $k$ is odd and the Clifford involution if $k$ is even.

The important property of the involution $\sigma_V$ that we use in this paper is that $\sigma_V$ preserves the semisimple conjugacy classes of $\GPin(V)$; namely $g$ and $\sigma_V(g)$ are conjugate in $\GPin(V)$ for all semisimple $g\in\GPin(V)$. This will be the main theorem of the next section. Here, let us prove the following.

\begin{lemma}\label{lemma:sigma_V_fixes_center}
For all $z\in Z_{\GPin(V)}$ we have
\[
\sigma_V(z)=z.
\]
\end{lemma}
\begin{proof}
Clearly if $z\in F^\times=Z^\circ$ then $\sigma_V(z)=z$. Hence the lemma follows if $n=2k$. Assume $n=2k-1$, so that $Z_{\GPin(V)}=F^\times\cup F^\times\zeta$, where $\zeta$ is as in \eqref{eq:shimura_zeta}, namely
\[
\zeta=e_1\cdots e_n,
\]
where $\{e_1,\dots, e_n\}$ is an orthogonal basis. Then
\[
\sigma_V(\zeta)=(-1)^{k+1}\zeta^*=(-1)^{k+1}(-1)^{k+1}\zeta=\zeta
\]
by Lemma \ref{lemma:shimura_zeta}.
\end{proof}

\subsection{Inclusions of $\GPin(W)$ and $\GSpin(W)$}
Let $W\subseteq V$ be a nondegenerate subspace of $V$. We have the natural inclusions
\[
C(W)\subseteq C(V)\qand C^+(W)\subseteq C^+(V).
\]
\begin{prop}\label{prop:GPin_inclusion}
The above inclusions restrict to the following inclusions:
\[
\GPin(W)\subseteq\GPin(V)\qand \GSpin(W)\subseteq\GSpin(V).
\]
\end{prop}
\begin{proof}
Let $g\in\GPin(W)$, so that $g\in C(W)$ is such that $\alpha(g) W g^{-1}=W$. We need to show $\alpha(g)Vg^{-1}=V$. But $V=W\oplus W^\perp$ and
\[
\alpha(g)(W\oplus W^\perp)g^{-1}=\alpha(g) W g^{-1}\oplus \alpha(g) W^\perp g^{-1}=W\oplus \alpha(g) W^\perp g^{-1}.
\]
Hence it suffices to show $\alpha(g) W^\perp g^{-1}=W^\perp$. To show it, notice that we can write $g=w_1\cdots w_\ell$ for some $w_1,\dots,w_\ell\in W$  by Lemma \ref{lemma:homogeneous}. Then by Lemma \ref{lemma:orthogonal_vectors}, we know that each $v\in W^\perp$ and $w_i\in W$ commute with each other and hence $v$ and $g$ commute, which implies $\alpha(g)vg^{-1}=v\alpha(g)g^{-1}=\pm v$. Hence $\alpha(g) W^\perp g^{-1}=W^\perp$.
\end{proof}

In particular, as a special case, if $V=W\oplus Fe$, where $e$ is anisotropic, which is the situation of our interest in our paper, we have the natural inclusions
\[
\GPin(W)\subseteq\GPin(V)\qand \GSpin(W)\subseteq\GSpin(V).
\]
Note that
\[
\GPin(W)=\GPin(V)_e\qand \GSpin(W)=\GSpin(V)_e,
\]
where $\GPin(V)_e$ is the stabilizer of $e$ in $\GPin(V)$ under the action of $\GPin(V)$ on $V$ via the canonical projection $P:\GPin(V)\to\OO(V)$, and similarly for $\GSpin(V)_e$.

Let us next assume we have an orthogonal sum decomposition
\[
V=V_1\oplus V_2,
\]
where both $V_1$ and $V_2$ are nondegenerate, so that we have both $\GPin(V_1)$ and $\GPin(V_2)$ as subgroups of $\GPin(V)$. The following can be readily verified.
\begin{lemma}\label{lemma:commuting_elements_in_GPin}
For each $g_1\in\GPin(V_1)$ and $g_2\in\GPin(V_2)$,
\[
g_1g_2g_1^{-1}=\begin{cases}g_2&\text{if $g_1\in\GSpin(V_1)$ or $g_2\in\GSpin(V_2)$};\\
-g_2&\text{otherwise};
\end{cases}
\]
namely if at least one of the $g_i$'s is in $\GSpin(V_i)$ then $g_1$ and $g_2$ commute.
\end{lemma}
\begin{proof}
One can prove the lemma, arguing analogously as the proof of Proposition \ref{prop:GPin_inclusion} by using Lemmas \ref{lemma:orthogonal_vectors} and \ref{lemma:homogeneous}. The detail is left to the reader.
\end{proof}

This lemma allows us to make the semidirect product
\[
\GPin(V_1)\ltimes\GPin(V_2)
\]
by letting $\GPin(V_1)$ act on $\GPin(V_2)$ by conjugation. Further the lemma implies that this semidirect product restricts to the direct products
\[
\GPin(V_1)\times\GSpin(V_2)\qand\GSpin(V_1)\times\GPin(V_2);
\]
namely if one of the $\GPin(V_i)$'s is restricted to $\GSpin(V_i)$ then the semidirect product becomes a direct product.
We then have the natural map
\begin{align*}
1\longrightarrow\{(z, z^{-1})\}\longrightarrow&\GPin(V_1)\ltimes\GPin(V_2)\longrightarrow\GPin(V_1\oplus V_2)\\
&\hspace{3.5em}(g_1, g_2)\hspace{3em}\mapsto\qquad g_1g_2,
\end{align*}
where $\{(z, z^{-1})\}$ is the obvious subgroup of the connected component $Z^\circ\times Z^\circ$ of the center of $\GPin(V_1)\ltimes\GPin(V_2)$.
%\subsection{Splitting of Siegel Levi}
%Let us assume that $V$ is split, so that $V=X\oplus X^*$, where $X$ is totally isotropic and $X^*$ is the algebraic dual of $X$ with the canonical pairing given by the symmetric bilinear form $\la-,-\ra$ of $V$. Let $M(X)\subseteq \SO(V)$ be the Siegel Levi of $\SO(V)$ stabilizing $X$, so that $M(X)\simeq\GL(X)$. We then have the following.
%\begin{lemma}\label{lemma:splitting_Levi}
%Keep the above notation. There is a splitting
%\[
%s:M(X)\simeq\GL(X)\longrightarrow\GSpin(V)
%\]
%such that
%\[
%N(s(g))={\det}_X(g)
%\]
%for all $g\in M(X)$, where ${\det}_X(g)$ is the determinant of $g$ viewed in $\GL(X)$.
%\end{lemma}
%\begin{proof}
%See \cite[Proposition 4.8, p.33]{Shimura}.
%\end{proof}
%
%One can, of course, generalize it to a splitting of the $\GL$-factor of any parabolic subgroup of $\SO(V)$. Namely, if we have a decomposition
%\[
%V=X_1\oplus\cdots\oplus X_\ell\oplus W\oplus X_1\oplus\cdots\oplus X_\ell,
%\]
%where $W$ and $X_i\oplus X_i^*$ are nondegenerate, then the $\GL$-part
%\[
%\GL(X_1)\times\cdots\times\GL(X_\ell)
%\]
%of the Levi splits in $\GSpin(V)$ such that the product of all the determinants is equal to the Clifford norm. Indeed, the corresponding Levi part of the parabolic subgroup of $\GSpin(V)$ is of the form
%\[
%\GL(X_1)\times\cdots\times\GL(X_\ell)\times\GSpin(W).
%\]
%Then the splitting of $\GL(X_1)\times\cdots\times\GL(X_\ell)$ is the obvious one.

\subsection{$\GPin(V)$ as a semidirect product}
To have a better understanding of the group $\GPin(V)$, let us mention the following, though we will not use it for the proof of our main theorems.
\begin{prop}\label{prop:semidirect_product}
Assume $n=2k-1$. One can choose $\zeta\in Z_{\GPin(V)}$ to be such that $\GPin(V)\simeq\GSpin(V)\times \{1, \zeta\}$ if and only if $\disc(V)=1$, where $\disc(V)$ is the discriminant of the quadratic space $V$ as usual.

Assume $n=2k$. If there exists $t\in\GPin(V)\smallsetminus\GSpin(V)$ such that $t^2=1$, then $\GPin(V)\simeq\GSpin(V)\rtimes\{1, t\}$, where the action of $t$ is by conjugation. Such $t$ exists when there exists $v\in V$ such that $q(v)=1$.
\end{prop}
\begin{proof}
Assume $n=2k-1$. It is an easy exercise to show that $\zeta^2=\disc(V)$ viewed modulo $F^{\times 2}$. Hence one can find $\zeta$ such that $\zeta^2=1$ if and only if $\disc(V)=1$. Now, choose such $\zeta$, so that the direct product $\GSpin(V)\times \{1, \zeta\}$ makes sense. It is then easy to see that the map $\GSpin(V)\times \{1, \zeta\}\to\GPin(V)$, $(g, \epsilon)\mapsto g\epsilon$, is an isomorphism.

Assume $n=2k$. Assume such $t$ exists, so that the semidirect product $\GSpin(V)\rtimes\{1, t\}$ makes sense. It is then easy to see that the map $\GSpin(V)\rtimes \{1, \zeta\}\to\GPin(V)$, $(g, \epsilon)\mapsto g\epsilon$, is an isomorphism. If there exists $v\in V$ such that $q(v)=1$ then one can simply set $t=v$.
\end{proof}

\subsection{The groups $\Pin(V)$ and $\Spin(V)$}
As the last thing in this section, let us mention how the groups $\GPin(V)$ and $\GSpin(V)$ are related to the more familiar $\Pin(V)$ and $\Spin(V)$. (Though we do not need the groups $\Pin(V)$ and $\Spin(V)$ for our purposes, we introduce them to justify our terminology ``general Pin group".) First, we have the Clifford norm
\[
\GPin(V)\longrightarrow F^\times,\quad g\mapsto N(g)=g\bar{g}.
\]
The Clifford norm descends to
\[
\OO(V)\longrightarrow F^\times\slash F^{\times 2}
\]
because $N(z)\in F^{\times 2}$ for $z\in Z^\circ=F^\times$ by \eqref{eq:Norm_on_Z}, which is called the spinor norm.

Now, by definition
\[
\Pin(V):=\ker(N:\GPin(V)\to F^\times),
\]
which is called the Pin group.\footnote{According to \cite[p.3]{ABS}, the term Pin was coined by J-P.\ Serre as a joke, though we do not know exactly what the joke was.} Via the projection $P:\GPin(V)\to\OO(V)$, we have the map
\[
1\longrightarrow \{\pm 1\}\longrightarrow \Pin(V)\xrightarrow{\;P\;} \OO(V).
\]
Note that the map $\Pin(V)\to\OO(V)$ is not necessarily surjective for the $F$-rational points, though it is surjective as a morphism of group schemes. Hence we have the
following commutative diagram of group schemes
\[
\begin{tikzcd}
1 \arrow[r] & \GL_1\arrow[r]&\GPin(V)\arrow[r,"P"]& \OO(V) \arrow[r] &1\\
1 \arrow[r] & \{\pm 1\} \ar[u, symbol={\subseteq}] \arrow[r]&\Pin (V)\ar[u, symbol={\subseteq}]\arrow[r,"P"]& \OO(V)\ar[u, symbol={=}] \arrow[r] &1\rlap{\, ,}
\end{tikzcd}
\]
where the bottom row is not necessarily exact for the $F$-rational points.

The Spin group is defined by
\[
\Spin(V):=\GSpin(V)\cap \Pin(V).
\]
Analogously to $\Pin(V)$ and $\GPin(V)$, we have the commutative diagram of group schemes
\[
\begin{tikzcd}
1 \arrow[r] & \GL_1\arrow[r]&\GSpin(V)\arrow[r,"P"]& \SO(V) \arrow[r] &1\\
1 \arrow[r] & \{\pm 1\} \ar[u, symbol={\subseteq}] \arrow[r]&\Spin (V)\ar[u, symbol={\subseteq}]\arrow[r,"P"]& \SO(V)\ar[u, symbol={=}] \arrow[r] &1\rlap{\, ,}
\end{tikzcd}
\]
where the bottom row is not necessarily exact for the $F$-rational points.

We hope that this discussion justifies our terminology ``general Pin group" and ``general Spin group".

Now, let $W\subseteq V$ be a nondegenerate subspace of codimension one. Analogously to the $\GPin$ and $\GSpin$ cases, there are natural inclusions $\Pin(W)\subseteq\Pin(V)$ and $\Spin(W)\subseteq\Spin(V)$. Hence one could certainly ask the multiplicity question for the pairs
\[
(G, H)=(\Pin(V), \Pin(W))\qand (\Spin(V), \Spin(W)),
\]
namely whether or not $\dim_{\C}\Hom_H(\pi, \tau)\leq 1$ for $\pi\in\Irr(G)$ and $\tau\in\Irr(H)$. However, our method of proof in this paper does not apply to these cases. This is because we will crucially use the fact that the canonical projections $P:\GPin(V)\to \OO(V)$ and $P:\GSpin(V)\to\SO(V)$ are surjective for the $F$-rational points.

\begin{rem}
The referee kindly pointed out that the multiplicity-at-most-one indeed fails for the Pin and Spin cases, though the authors do not know about it.
\end{rem}

%%%%%%%%%%%%%%%%%%%%%%%%%%%%%%%%%%%%%%%%%%%%%%%%%%%%%%%%%%%%%%%%%%%%%%%%

\section{On semisimple elements}

%%%%%%%%%%%%%%%%%%%%%%%%%%%%%%%%%%%%%%%%%%%%%%%%%%%%%%%%%%%%%%%%%%%%%%%%

In this section, we prove that each semisimple $g\in\GPin(V)$ is conjugate to $\sigma_V(g)$ in $\GPin(V)$; namely there exists $\eta\in\GPin(V)$ such that
\[
\eta \sigma_V(g)\eta^{-1}=g.
\]

\subsection{Basic idea}
The proof of the theorem requires the following $\OO(V)$ analogue due to Moeglin-Vigneras-Waldspurger (MVW).
\begin{lemma}\label{lemma:MVW_orthogonal_group}
For each (not necessarily semisimple) $h\in\OO(V)$, there exists $\beta\in\OO(V)$ such that
\[
\beta h^{-1}\beta^{-1}=h;
\]
namely $h$ and $h^{-1}$ are conjugate in $\OO(V)$. If $n=2k-1$ then the same holds for all $h\in\SO(V)$ because $\OO(V)=\SO(V)\times\{\pm1\}$ (direct product).
\end{lemma}
\begin{proof}
This is \cite[I.2 Proposition, p.79]{MVW}.
\end{proof}

Since $\sigma_V(g)=\pm N(g)g^{-1}$, where $N(g)$ is the Clifford norm, we have $P(\sigma_V(g))=P(g)^{-1}$, where $P:\GPin(V)\to\OO(V)$ is the canonical projection. Hence, by this lemma, we know that $P(g)$ and $P(\sigma_V(g))$ are conjugate in $\OO(V)$; namely there exists $\beta\in\OO(V)$ such that $\beta P(\sigma_V(g))\beta^{-1}=P(g)$. Let $\eta\in\GPin(V)$ be any element such that $P(\eta)=\beta$. Then we have
\[
\eta\sigma_V(g)\eta^{-1}=zg
\]
for some $z\in Z^\circ$. By applying the Clifford norm $N$ to both sides and using $N(\sigma_V(g))=N(g)$, we have $z^2=1$ by \eqref{eq:Norm_on_Z}, which gives
\begin{equation}\label{eq:up_to_sign}
\eta\sigma_V(g)\eta^{-1}=\pm g.
\end{equation}
Namely, $\sigma_V(g)$ and $g$ are conjugate ``up to $\pm 1$". In what follows, we will show the sign $\pm$ is indeed $+$ for semisimple $g$.

\subsection{Proof of MVW}
To show the sign in \eqref{eq:up_to_sign} is indeed $+$, we need to analyze the proof of the above lemma (Lemma \ref{lemma:MVW_orthogonal_group}) at least for semisimple $h\in\OO(V)$. So in this subsection we reproduce the proof of Lemma \ref{lemma:MVW_orthogonal_group} for a fixed semisimple $h\in\OO(V)$. The basic step is to compute the centralizer $\OO(V)_h$ of $h$ in $\OO(V)$, which is an old result by Steinberg.

\begin{prop}\label{prop:centralizer_of_h}
Let $h\in\OO(V)$ be semisimple. Then we have the following.
\begin{enumerate}[(1)]
\item There is an orthogonal sum decomposition
\[
V=V_1\oplus\cdots\oplus V_m\oplus V_+\oplus V_-,
\]
where each of $V_i$ and $V_{\pm}$ is an $h$-invariant subspace of the following type:
\begin{enumerate}[(a)]
\item $V_i=X_i\oplus X_i^*$, where $X_i$ is a vector space over a finite extension $A_i$ of $F$ and $X_i^*$ is the dual of $X_i$, and both $X_i$ and $X_i^*$ are $h$-invariant;
\item $V_i$ is a vector space over a finite extension $A_i$ of $F$ equipped with a (nondegenerate) Hermitian form $\lla-,-\rra_i$ with respect to an involution on $A_i$ which fixes $F$ pointwise;
\item Each of $V_{\pm}$ is a nondegenerate subspace of $V$, on which $h$ acts on $V_{\pm}$ as $\pm 1$, respectively.
\end{enumerate}
\item There is an isomorphism
\begin{equation}\label{eq:centralizer_isomorphism_O(V)}
\OO(V)_h\simeq G_1\times\cdots\times G_m\times\OO(V_+)\times\OO(V_-),
\end{equation}
where
\[
G_i=\begin{cases}\GL_{A_i}(X_i),&\text{if $V_i=X_i\oplus X_i^*$};\\ U_{A_i}(V_i),&\text{otherwise}.\end{cases}
\]
\end{enumerate}
Note in the above that some of the spaces $V_1\oplus\cdots \oplus V_m$, $V_+$ and $V_-$ could be zero.
\end{prop}
\begin{proof}
This has been known for decades and is cited in, say, \cite[bottom of p.315]{Wal12}. But since we have not been able to locate any reference with a complete proof for this precise form, we reproduce the proof in Appendix \ref{Appendix_A}.
\end{proof}

\begin{rem}
Let us make a couple of remarks about this proposition. First, if $V_i=X_i\oplus X_i^*$ then the group $\GL_{A_i}(X_i)$ should be rather viewed as the diagonal $\{(g, {g^*}^{-1})\in\GL_{A_i}(X_i)\times\GL_{A_i}(X_i^*)\}$, where $g^*$ is the adjoint of $g$ with respect to the canonical pairing. Second, $\dim_F(V_1\oplus\cdots\oplus V_m)$ is always even, and hence if $\dim_FV$ is odd then $V_+\oplus V_-\neq 0$.
\end{rem}

Let us explicate each of the cases (a), (b) and (c) in (1) of the above proposition. To ease the notation, we drop the subscript $i$.

Assume we are in (a), namely $G=\GL_{A}(X)$. Then $G$ is in the Siegel Levi $\GL_F(X)$ of the split even orthogonal group $\SO(X\oplus X^*)$, which is naturally a subgroup of $\SO(V)$. Then $\GL_{A}(X)$ is embedded in $\SO(X\oplus X^*)$ via
\[
\GL_{A}(X)\longrightarrow\GL_F(X)\times\GL_F(X^*),\quad g\mapsto (g, {g^*}^{-1}),
\]
where $g^*$ is the adjoint of $g$ with respect to the canonical pairing. The image of $h$ in $G$ is in the center $A$ of $\GL_A(X)$. Hence $h$ is actually $(h, h^{-1})$, since $h^*=h$ for the central $h$, and hence $h^{-1}$ is actually $(h^{-1}, h)$.

Now fix an $A$-basis $\{e_1,\dots, e_n\}$ of $X$ and its dual basis $\{e^*_1,\dots, e^*_n\}$ of $X^*$, so that $\la e_i, e_j^*\ra=\delta_{ij}$. Define
\[
u:X^*\longrightarrow X,\quad e_i^*\mapsto e_i,
\]
and set
\[
\beta=\begin{pmatrix}& u\\u^{-1}&\end{pmatrix},
\]
where the matrix is with respect to the fixed basis. Then $\beta\in\OO(X\oplus X^*)$, and moreover
\[
\beta(h, h^{-1})\beta^{-1}=(h^{-1}, h).
\]

Note that we have
\[
{\det}_F(\beta)=(-1)^{\frac{1}{2}\dim_F V},
\]
where ${\det}_F(\beta)$ is the determinant of $\beta$ viewed as an $F$-linear map. This can be computed as follows:
\begin{align*}
{\det}_F(\beta)&=N_{A\slash F}({\det}_{A}(\beta))\\
&=N_{A\slash F}((-1)^n)\\
&=(-1)^{n\dim_FA}\\
&=(-1)^{\frac{1}{2}\dim_F V}.
\end{align*}

It should be noted that we have the involution
\[
\GL_A(X)\longrightarrow\GL_A(X),\quad g\mapsto ug^*u^{-1}.
\]
One can see that this involution is $g\mapsto g^t$, where $g^t$ is the transpose of $g$ with respect to the fixed basis, and $h$, viewed as an element in $\GL_A(X)$, is a fixed point of the transpose map because $h$ is in the center $A$ of $\GL_A(X)$.

Next assume we are in (b); namely the space $V$ is equipped with a Hermitian bilinear form $\lla-,-\rra$ over a  quadratic extension $A/A'$ with $F\subseteq A'$ and $G=U_A(V)$ a unitary group. Further we have $U_A(V)\subseteq\SO(V)$. Here, since the unitary group is connected, it is in the special orthogonal group $\SO(V)$. Now let $e_1,\dots, e_n$ be an orthogonal $A$-basis of $V$ with respect to the Hermitian form $\lla-,-\rra$. Define
\[
\beta:V\longrightarrow V
\]
by
\[
\beta(a_1e_1+\cdots+a_ne_n)=\bar{a}_1e_1+\cdots+\bar{a}_ne_n,
\]
where the bar is the Galois conjugation for the quadratic extension $A/A'$. Since the image of our $h$ in $G$ is in the center of the unitary group $U_A(V)$, we have
\[
\beta h^{-1}\beta^{-1}=h.
\]
Moreover, since the vectors $e_1,\dots, e_n$ are also orthogonal with respect to the quadratic form $\la-,-\ra$ on $V$, we know that $\beta\in\OO(V)$.

Note that
\[
{\det}_F(\beta)=(-1)^{\frac{1}{2}\dim_F V},
\]
whose proof is essentially the same as the $\GL$-case.

Also note that we have
\[
\lla \beta v, \beta v'\rra=\lla v', v\rra
\]
for all $v, v'\in V$ and $\beta^2=1$. Hence for each $g\in U_A(V)$ we have $\beta g^{-1}\beta^{-1}\in U_A(V)$, and thus we have the involution
\[
U_A(V)\longrightarrow U_A(V),\quad g\mapsto \beta g^{-1}\beta^{-1},
\]
and $h$ is a fixed point of this involution.

Finally, assume we are in (c), so that we have $V=V_+\oplus V_-$. In this case, our $h$ is simply $(1, -1)\in\OO(V_+)\times\OO(V_-)$, and hence $h^{-1}=h$. So we can take
\[
\beta=(\beta_+, \beta_-)\in\OO(V_+)\times\OO(V_-),
\]
where $\beta_+\in\OO(V_+)$ and $\beta_-\in\OO(V_-)$ are arbitrary.

Now, we can glue together all the three cases. Namely, if $h\in\OO(V)$ is such that
\[
\OO(V)_h\simeq G_1\times\cdots\times G_m\times\OO(V_+)\times\OO(V_-),
\]
then we can write
\[
h=(h_1,\dots, h_m, h_+, h_-)\in G_1\times\cdots\times G_m\times\OO(V_+)\times\OO(V_-),
\]
where
\[
G_i\subseteq\SO(V_i)
\]
for $i=1,\dots, m$. For each $G_i$ let $\beta_i\in\OO(V_i)$ be the corresponding element constructed above, so that
\[
\beta_ih_i^{-1}\beta_i^{-1}=h_i.
\]
Then if we set
\[
\beta=(\beta_1,\dots,\beta_m, \beta_+, \beta_-),
\]
we have the desired
\[
\beta h^{-1}\beta^{-1}=h.
\]
Lemma \ref{lemma:MVW_orthogonal_group} is proven.

\quad

Let $\dim_FV_i=2k_i$, $\dim_FV_+=2k_+$ or $2k_+-1$, and $\dim_FV_-=2k_-$ or $2k_--1$. (Here note that the ``non-orthogonal space" $V_i$ in the decomposition $V=V_1\oplus\cdots\oplus V_m\oplus V_+\oplus V_-$ is either $X_i\oplus X_i^*$ or Hermitian, and hence $\dim_FV_i$ is always even.) Then we have
\[
{\det}_F(\beta)=(-1)^{k_1+\cdots+k_m}\det(\beta_+)\det(\beta_-).
\]
Since $\beta_+$ and $\beta_-$ can be chosen arbitrarily, we can choose $\beta$ so that
\[
{\det}_F(\beta)=(-1)^{k}
\]
by choosing $\beta_+$ and $\beta_-$ appropriately.

%\subsection{Involution on centralizer}
%From the above construction of $\beta$, we know
%\[
%\beta^2=1.
%\]
%Also for each $g\in\OO(V)_h$ one can readily verify that $\beta g\beta^{-1}\in\OO(V)_h$, so that we have the involution
%\[
%\sigma_h:\OO(V)_h\longrightarrow\OO(V)_h,\quad g\mapsto \beta g^{-1}\beta^{-1},
%\]
%and our $h$ is a fixed point of this involution.

\subsection{Lifting to $\GPin$}
Let $g\in\GPin(V)$, and set
\[
h:=P(g).
\]
Let $\beta$ be as above for this $h$, so $\beta h^{-1}\beta^{-1}=h$. We would like to ``lift" $\beta$ to $\GPin(V)$. For this purpose, we fix
\[
\eta\in\GPin(V)\quad\text{such that}\quad P(\eta)=\beta.
\]
Each such $\eta$ differs by an element in $Z^\circ$; namely if $\eta'$ is another such choice then $\eta'=z\eta$ for some $z\in Z^\circ$.
So far we know
\[
\eta\sigma_V(g)\eta^{-1}=\pm g.
\]
To show the sign $\pm$ is actually $+$, we consider separate cases accordingly as the structure of the centralizer $\OO(V)_h$.

\subsection{Without orthogonal factor}
Assume that $h$ is such that $\OO(V)_h$ has no orthogonal group $\OO(V_+)\times\OO(V_-)$, so that $\OO(V)_h$ is a product of (restrictions of scalar of) general linear groups and unitary groups. Note that in this case $\dim_FV$ is necessarily even.
\begin{prop}\label{prop:semisimple_conjugate_without_orthogonal_factor}
Let $g\in\GPin(V)$ be semisimple. Let $h=P(g)\in\OO(V)$ and let $\beta\in\OO(V)$ be such that $\beta h^{-1}\beta^{-1}=h$. Assume $h$ is such that $\OO(V)_h$ has no factor of the orthogonal groups $\OO(V_+)\times\OO(V_-)$. Then $g\in\GSpin(V)$, and
\[
\eta\sigma_V(g)\eta^{-1}=g
\]
where $\eta$ is any element in $\GPin(V)$ such that $P(\eta)=\beta$.
\end{prop}
\begin{proof}
We already know that $h\in\SO(V)$ because $\OO(V)_h$ has no $\OO(V_+)\times\OO(V_-)$-factor. Hence $g\in\GSpin(V)$.

To show $\eta\sigma_V(g)\eta^{-1}=g$, let $C\subseteq \OO(V)_h$ be the center of $\OO(V)_h$, so that we have the short exact sequence
\[
1\longrightarrow Z^{\circ}\longrightarrow P^{-1}(C)\longrightarrow C\longrightarrow 1.
\]
Since $\OO(V)_h$ is a product of (restrictions of scalar of) general linear groups and unitary groups, $C$ is an algebraic torus, which implies $P^{-1}(C)$ is an algebraic torus and hence is Zariski connected. Of course, $h\in C$. Further it can be readily verified that for any $c\in C$ we have $\beta c^{-1}\beta^{-1}=c$. Hence for any $\gamma\in P^{-1}(C)$, we have
\[
\eta\sigma_V(\gamma)\eta^{-1}=\epsilon(\gamma)\gamma\quad\text{for}\quad  \epsilon(\gamma)\in\{\pm1\}\subseteq Z^\circ,
\]
which gives rise to a morphism
\[
\epsilon: P^{-1}(C)\longrightarrow \{\pm 1\}\subseteq Z^\circ,\quad \gamma\mapsto \epsilon(\gamma)=\eta\sigma_V(\gamma)\eta^{-1}\gamma^{-1},
\]
of algebraic varieties. But $P^{-1}(C)$ is Zariski connected, and apparently $\epsilon(1)=1$. Thus we have $\epsilon(\gamma)=1$ for all $\gamma\in P^{-1}(C)$, and in particular $\epsilon(g)=1$.
\end{proof}

In the above proposition, since $P(\eta)=\beta$ and we know
\[
{\det}_F(\beta)=(-1)^k,
\]
where $k$ is such that $\dim_FV=2k$, we have either $\eta\in\GSpin(V)$ or $\eta\in\GPin(V)\smallsetminus\GSpin(V)$, depending on the parity of $k$.

\subsection{Orthogonal factor}
Next consider the case where $h\in\OO(V)$ is such that $\OO(V)_h=\OO(V_+)\times\OO(V_-)$; namely there is no factor of general linear groups or unitary groups, so that the decomposition of $V$ is
\[
V=V_+\oplus V_-.
\]
Let us also write
\[
h=(h_+, h_-),
\]
where $h_+=1_{V_+}$ and $h_-=-1_{V_-}$. Note that, in this case, we have $h^{-1}=h$, and can take
\[
\beta=(\beta_+, \beta_-)\in\OO(V_+)\times\OO(V_-)
\]
to be arbitrary.

Let $e_1,\dots, e_\ell$ be an orthogonal basis of $V_-$ and set
\[
\zeta_-=e_1\cdots e_\ell\in \GPin(V_-).
\]
Note that
\[
P(\zeta_-)=h_-=-1_{V_-}
\]
by Lemma \ref{lemma:shimura_zeta} (a). Hence, if $g\in\GPin(V)$ is such that $P(g)=h$ then we have
\[
g=z\zeta_-
\]
for some $z\in Z^\circ$. To see this, consider the composite
\[
\GSpin(V_+)\times\GPin(V_-)\longrightarrow\GPin(V_+\oplus V_-)\longrightarrow\OO(V_+\oplus V_-)
\]
and notice that the elements in $\GSpin(V_+)$ that map to $1_{V_+}$ are of the form $z$ for $z\in Z^\circ$.

\begin{prop}\label{prop:semisimple_conjugate_orthogonal_factor}
Keep the above notation and assumption. Let $\tau$ be either the Clifford involution or the canonical involution on $\GPin(V_+\oplus V_-)$. Let $k_-$ be such that
\[
\dim_FV_-=2k_- \text{ or }\ 2k_--1
\]
as before.
\begin{enumerate}[(a)]
\item Assume $\dim_FV_-$ is even. There exists $\eta\in\GPin(V_+\oplus V_-)$ such that
\[
\eta\tau(\zeta_-)\eta^{-1}=\zeta_-\qand P(\eta)=(\beta_+,\, \beta_-),
\]
where $(\beta_+, \beta_-)\in\OO(V_+)\times\OO(V_-)$ is such that
\[
\det(\beta_+)=\text{any}\qand \det(\beta_-)=(-1)^{k_-}.
\]
In particular, $\tau(\zeta_-)$ and $\zeta_-$ are conjugate in $\GPin(V_+\oplus V_-)$.
\item Assume both $\dim_FV_+$ and $\dim_FV_-$ are odd. Then for each $\varepsilon\in\{0, 1\}$ there exists $\eta\in\GPin(V_+\oplus V_-)$ such that
\[
\eta\tau(\zeta_-)\eta^{-1}=(-1)^{\varepsilon}\zeta_-.
\]
In particular all of $\tau(\zeta_-)$, $\zeta_-$ and $-\zeta_-$ are conjugate in $\GPin(V_+\oplus V_-)$.
\end{enumerate}
%
%Keep the above notation and assumption, let
%\[
%\htt\in\GPin(V_+\oplus V_-)
%\]
%be any of the preimages of $h=(h_+, h_-)$, so that
%\[
%\htt=z\zeta_-
%\]
%for some $z\in Z^\circ$. We then have the following.
%
%
%\begin{enumerate}[(a)]
%\item Assume $\dim_FV_+=2k_+$ and $\dim_FV_-=2k_-$. Then there exists $\eta\in\GPin(V_+\oplus V_-)$ such that
%\[
%\eta\tau(\htt)\eta^{-1}=\htt\qand P(\eta)=((-1)^{k_+}1_{V_+},\, (-1)^{k_-}1_{V_-}).
%\]
%\item Assume $\dim_FV_+=2k_+-1$ and $\dim_FV_-=2k_--1$. Then there exists $\eta\in\GPin(V_+\oplus V_-)$ such that
%\[
%\eta\tau(\htt)\eta^{-1}=\htt\qand P(\eta)=((-1)^{k_+}1_{V_+},\, (-1)^{k_-}1_{V_-}).
%\]
%Also there exists $\eta\in\GPin(V_+\oplus V_-)$ such that
%\[
%\eta\tau(\htt)\eta^{-1}=-\htt\qand P(\eta)=((-1)^{k_+}1_{V_+},\, (-1)^{k_-}1_{V_-}).
%\]
%
%\item Assume $\dim_F(V_+\oplus V_-)=2k$. Then there exists $\eta\in\GPin(V_+\oplus V_-)$ such that
%\[
%\eta\tau(\htt)\eta^{-1}=\htt\qand \det(P(\eta))=(-1)^k.
%\]
%Further if $\htt\notin\GSpin(V_+\oplus V_-)$ then there exists $\eta\in\GPin(V_+\oplus V_-)$ such that
%\[
%\eta\tau(\htt)\eta^{-1}=\htt\qand \det(P(\eta))=(-1)^k.
%\]
%\item Assume $\dim_F(V_+\oplus V_-)=2k-1$. Assume further that $\htt\in\GSpin(V)$. Then there exists $\eta\in\GPin(V_+\oplus V_-)$ such that
%\[
%\eta\tau(\htt)\eta^{-1}=\htt\qand \det(P(\eta))=(-1)^k.
%\]
%\end{enumerate}
\end{prop}
\begin{proof}
(a) Assume $\dim_FV_-=2k_-$ even. Then we have $\zeta_-\in\GSpin(V_-)$, so that the Clifford involution and the canonical involution coincide, which implies
\[
\tau(\zeta_-)=\zeta_-^*=(-1)^{k_-}\zeta_-
\]
by Lemma \ref{lemma:shimura_zeta}. Also by Lemma \ref{lemma:orthogonal_vectors}
\[
e_1\zeta_-e_1^{-1}=e_1(e_1\cdots e_{\ell})e_1^{-1}=-\zeta_-
\]
because $\ell$ is even, which implies
\[
e_1^{k_-}\zeta_-e_1^{-k_-}=(-1)^{k_-}\zeta_-.
\]
Hence if $V_+=0$, we can take $\eta=e_1^{k_-}$ and we are done because $P(e_1)$ is the reflection in the hyperplane orthogonal to $e_1$, so that $\det(P(e_1))=-1$.

Assume $V_+\neq 0$. For any anisotropic $x\in V_+$, we have
\[
x\zeta_-x^{-1}=\zeta_-
\]
again by using Lemma \ref{lemma:orthogonal_vectors}. Hence if we take
\[
\eta=x^{m}{e_1}^{k_-}
\]
for an arbitrary $m$ then we can see this $\eta$ has the desired properties.

(b) Assume both $\dim_FV_+$ and $\dim_FV_-$ are odd. Then we have $\zeta_-\in\GPin(V_-)$ but $\zeta_-\notin\GSpin(V_-)$. Hence depending on $k_-$ and depending on whether $\tau$ is the Clifford involution or the canonical involution, we have
\[
\tau(\zeta_-)=\zeta_-\quad\text{or}\quad \tau(\zeta_-)=-\zeta_-.
\]
Assume $\tau(\zeta_-)=\zeta_-$. Since $\dim_FV_+$ is odd and so $V_+\neq 0$, we know there is an anisotropic $x\in V_+$. Since $\dim_FV_-$ is odd, by using Lemma \ref{lemma:orthogonal_vectors} we know that
\[
x\zeta_-x^{-1}=-\zeta_-.
\]
Hence if we take
\[
\eta=x^\varepsilon,
\]
then this $\eta$ has the desired properties. If $\tau(\zeta_-)=-\zeta_-$ then we can take $\eta=x^{\varepsilon+1}$.

\end{proof}

\subsection{General Case}
Now, we consider the general case. Let us first set up our notation. Let $g\in\GPin(V)$ be semisimple such that for $h:=P(g)$ we have
\[
\OO(V)_h\simeq G_1\times\cdots\times G_m\times \OO(V_+)\times\OO(V_-)
\]
and
\[
V=V_1\oplus\cdots\oplus V_m\oplus V_+\oplus V_-.
\]
Define the integers $k_i, k_+$ and $k_-$ as before, namely $\dim_FV_i=2k_i$, etc. Also let us write
\[
V'=V_1\oplus\cdots\oplus V_m
\]
and
\[
h=(h',h_+, h_-)\in\OO(V')\times\OO(V_+\oplus V_-).
\]
We can then write
\[
g=g'g_o=g_og'
\]
where $g'\in\GSpin(V')$ is such that $P(g')=h'$ and $g_o\in\GPin(V_+\oplus V_-)$ is such that $P(g_o)=(h_+, h_-)$. (Here $g'$ is in $\GSpin(V)$ by Proposition \ref{prop:semisimple_conjugate_without_orthogonal_factor} and we have $g'g_o=g_og'$ by Lemma \ref{lemma:commuting_elements_in_GPin}.)

With this notation, we can state our main theorem of this section as follows.
\begin{theorem}\label{thm:semisimple_conjugate}
For a semisimple element $g\in\GPin(V)$ as above, there exists $\eta\in\GPin(V)$ such that
\[
\eta\sigma_V(g)\eta^{-1}=g;
\]
namely $g$ and $\sigma_V(g)$ are conjugate in $\GPin(V)$.

Further, if $g\in\GSpin(V)$ then the conjugating element $\eta$ can be chosen so that
\[
P(\eta)=(\beta_1,\dots,\beta_m,\beta_{+},\beta_-)\in\OO(V_1)\times\cdots\times\OO(V_m)\times\OO(V_+)\times\OO(V_-)
\]
with the property that
\[
\det(\beta_i)=(-1)^{k_i},\quad \det(\beta_+)=\text{any},\qand \det(\beta_-)=(-1)^{k_-}.
\]
In particular, by choosing $\det(\beta_+)$ appropriately we have
\[
\det(P(\eta))=(-1)^k,
\]
where $k$ is such that $n=2k$ or $n=2k-1$. (Note that since  $g\in\GSpin(V)$, $\dim_FV_-$ is necessarily even.)
\end{theorem}

To prove the theorem, let us first mention that for odd $\GPin(V)$ we have only to consider $g\in \GSpin(V)$ thanks to the following lemma.
\begin{lemma}
Assume $n=2k-1$. If each semisimple $g\in\GSpin(V)$ is conjugate to $\sigma_V(g)$ in $\GPin(V)$ then each semisimple $g\in\GPin(V)$ is conjugate to $\sigma_V(g)$ in $\GPin(V)$.
\end{lemma}
\begin{proof}
Recall we have the disjoint union
\[
\GPin(V)=\GSpin(V)\cup\GSpin(V)\zeta,
\]
where $\zeta$ is the central element as in \eqref{eq:shimura_zeta}. Since $\sigma_V(\zeta)=\zeta$ by Lemma \ref{lemma:sigma_V_fixes_center} and $\zeta$ is in the center, if $g\in\GSpin(V)$ is conjugate to $\sigma_V(g)$ in $\GPin(V)$, then $g\zeta$ is conjugate to $\sigma_V(g\zeta)$ by the same conjugating element.
\end{proof}

With this lemma, we can prove the theorem as follows.
\begin{proof}[Proof of Theorem \ref{thm:semisimple_conjugate}]
By Proposition \ref{prop:semisimple_conjugate_without_orthogonal_factor} we know that there is $\eta'\in\GPin(V')$ such that
\[
\eta'g'{\eta'}^{-1}=g'
\]
and
\[
P(\eta')=(\beta_1,\dots,\beta_m)\in\OO(V_1)\times\cdots\times\OO(V_m)
\]
with the property that $\det(\beta_i)=(-1)^{k_i}$.

Assume $\dim_FV_-$ is even, so that $g_o\in\GSpin(V_+\oplus V_-)$ and $g\in\GSpin(V)$. Note that
\[
g_o=z\zeta_-\in\GPin(V_-)\subseteq\GPin(V_+\oplus V_-)
\]
for some $z\in Z^\circ$.
Since $\sigma_V$ acts trivially on $Z^\circ$, by Proposition \ref{prop:semisimple_conjugate_orthogonal_factor} there exists $\eta_o\in\GPin(V_+\oplus V_-)$ such that
\[
\eta_o\sigma_V(g_o)\eta_o^{-1}=g_o
\]
and
\[
P(\eta_o)=(\beta_+, \beta_-)
\]
with the property that $\det(\beta_+)$ can be arbitrary and $\det(\beta_-)=(-1)^{k_-}$.

We can then compute
\begin{align*}
(\eta'\eta_o)\sigma_V(g)(\eta'\eta_o)^{-1}
&=(\eta'\eta_o)\sigma_V(g'g_o)(\eta'\eta_o)^{-1}\\
&=\eta'\eta_o\sigma_V(g')\sigma_V(g_o)\eta_o^{-1}{\eta'}^{-1}\\
&=\eta'\sigma_V(g')(\eta_o\sigma_V(g_o)\eta_o^{-1}){\eta'}^{-1}\\
&=\eta'\sigma_V(g')g_o{\eta'}^{-1}\\
&=\eta'\sigma_V(g'){\eta'}^{-1}g_o\\
&=g'g_o,
\end{align*}
where for the third equality we used $\sigma_V(g')\in\GSpin(V')$ and Lemma \ref{lemma:commuting_elements_in_GPin}, and for the fourth one we used $g_o\in\GSpin(V_+\oplus V_-)$ and the same lemma. Hence by taking
\[
\eta=\eta'\eta_o,
\]
we can see that this $\eta$ has the desired properties.

Next assume $\dim_FV_-$ is odd and $\dim_FV$ is even, so that necessarily $\dim_FV_+$ is odd. In this case $g\notin\GSpin(V)$, and we have only to show that $g$ and $\sigma_V(g)$ are conjugate and do not have to specify any particular property for the conjugating element. Assume $V'=0$. Then we already know it from Proposition \ref{prop:semisimple_conjugate_orthogonal_factor}. So assume $V'\neq 0$. Then again by Proposition \ref{prop:semisimple_conjugate_orthogonal_factor}, there exists $\eta_o\in\GPin(V_+\oplus V_-)$ such that
\[
\eta_o\sigma_V(g_o)\eta_o^{-1}=-g_o.
\]
Also let $\eta'\in\GPin(V')$ be as above. We can then compute
\begin{align*}
(\eta'\eta_o)\sigma_V(g)(\eta'\eta_o)^{-1}
&=(\eta'\eta_o)\sigma_V(g'g_o)(\eta'\eta_o)^{-1}\\
&=\eta'\eta_o\sigma_V(g')\sigma_V(g_o)\eta_o^{-1}{\eta'}^{-1}\\
&=\eta'\sigma_V(g')(\eta_o\sigma_V(g_o)\eta_o^{-1}){\eta'}^{-1}\\
&=-\eta'\sigma_V(g')g_o{\eta'}^{-1}\\
&=(-1)^2\eta'\sigma_V(g'){\eta'}^{-1}g_o\\
&=g'g_o,
\end{align*}
where for the third equality we used $\sigma_V(g')\in\GSpin(V')$ and Lemma \ref{lemma:commuting_elements_in_GPin}, and for the fourth one we used $g_o\notin\GSpin(V_+\oplus V_-)$ and the same lemma. Hence we can take $\eta=\eta'\eta_o$. (The sign change happens twice in the above computation, and this is the rationale for our choice of $\eta_o$.)

Finally, if both $\dim_FV_-$ and $\dim_FV$ are odd, then necessarily $g\in\GPin(V)\smallsetminus\GSpin(V)$. This case is taken care of by the above lemma.
\end{proof}

\begin{rem}\label{rem:choice_of_involution}
For $n=2k$ we have chosen $\sigma_V$ to be the canonical involution. However, we might as well choose $\sigma_V$ to be the Clifford involution. One can then see that the above argument works even if we use the Clifford involution, which implies that for $n=2k$ all of $g$, $\bar{g}$ and $g^*$ are conjugate in $\GPin(V)$ at least for semisimple $g$.
\end{rem}

%\begin{cor}\label{cor:conjugate_GPin}
%Each semisimple $g\in\GPin(V)$ is conjugate to $\sigma_V(g)$ in $\GPin(V)$.
%\end{cor}
%\begin{proof}
%If $n$ is not of the form $n=2k$ with $k$ odd, then this is the above theorem because for these cases we have $\sigma_V=\sigma_V$. If $n=2k$ with $k$ odd, then $\sigma_V(g)=e\sigma_V(g)e^{-1}$ and hence $\sigma_V(g)$ is conjugate to $g$.
%\end{proof}

The theorem implies the following two conjugacy statements for $\GSpin(V)$.

\begin{cor}\label{cor:conjugate_GSpin}
Each semisimple $g\in\GSpin(V)$ is conjugate to $e^k\sigma_V(g)e^{-k}$ in $\GSpin(V)$.
\end{cor}
\begin{proof}
Note that by the theorem, we know that $g$ and $\sigma_V(g)$ are conjugate in $\GPin(V)$ and hence $g$ and $e^k\sigma_V(g)e^{-k}$ are conjugate in $\GPin(V)$. Thus we have to show that the conjugating element can be chosen from $\GSpin(V)$. Also note that if $k$ is even then $e^k\in Z^{\circ}$ so that $e^k\sigma_V(g)e^{-k}=\sigma_V(g)$, and if $k$ is odd then $e^k\sigma_V(g)e^{-k}=e\sigma_V(g)e^{-1}$.

Assume $n=2k$ with $k$ even, so that $e^k\sigma_V(g)e^{-k}=\sigma_V(g)$. From the theorem we know that $\eta\sigma_V(g)\eta^{-1}=g$ with $\det(P(\eta))=(-1)^k=1$. But this implies $\eta\in\GSpin(V)$.

Assume $n=2k$ with $k$ odd, so that $e^k\sigma_V(g)e^{-k}=e\sigma_V(g)e^{-1}$. From the theorem we know $g=\eta\sigma_V(g)\eta^{-1}$, where $\det(P(\eta))=-1$. Hence $g=\eta e^{-1} e\sigma_V(g)e^{-1} e\eta^{-1}$. But $\det(P(e))=-1$ because $P(e)$ is the reflection in the hyperplane orthogonal to $e$. Hence $\det(P(\eta e^{-1}))=1$, which implies $\eta e^{-1}\in\GSpin(V)$.

Assume $n=2k-1$ with $k$ even, so that $e^k\sigma_V(g)e^{-k}=\sigma_V(g)$. Since $g\in\GSpin(V)$, necessarily $\dim_FV_-$ is even. Hence $\dim_FV_+$ is odd and in particular nonzero. Thus in the theorem we can take $\det(\beta_+)$ to be arbitrary, which implies we can take $\eta$ from $\GSpin(V)$.

Assume $n=2k-1$ with $k$ odd, so that $e^k\sigma_V(g)e^{-k}=e\sigma_V(g)e^{-1}$. As in the above case, we can take $\det(\beta_+)$ to be arbitrary, so we can take $\eta$ from $\GPin(V)\smallsetminus\GSpin(V)$, so that $\eta e\in\GSpin(V)$.
\end{proof}

%\begin{cor}\label{cor:conjugate_GSpin_2}
%Each semisimple $g\in\GSpin(V)$ is conjugate to $e^k\sigma_V(g)e^{-k}$ in $\GSpin(V)$.
%\end{cor}
%\begin{proof}
%For $n=2k$ this is the above corollary. Assume $n=2k-1$. If $k$ is even then this is immediate because $e^k\in Z^\circ$. Assume $k$ is odd, in which case $e^k\sigma_V(g)e^{-k}=e\sigma_V(g)e^{-1}$. Since $g$ is conjugate to $\sigma_V(g)$, there exists $\eta\in\GSpin(V)$ such that $\sigma_V(g)=\eta g\eta^{-1}$, so $e\sigma_V(g)e^{-1}=e\eta g\eta^{-1} e^{-1}$. Now let $\zeta=e_1\cdots e_n$ as usual, which is in the center of $\GPin(V)$. Then
%\begin{align*}
%e\eta g\eta^{-1} e^{-1}
%&=\zeta e\eta g\eta^{-1} e^{-1}\zeta^{-1}\\
%&=e_1\cdots e_{n-1}\eta g\eta^{-1}(e_1\cdots e_{n-1})^{-1}\\
%&=(e_1\cdots e_{n-1}\eta) g(e_1\cdots e_{n-1}\eta)^{-1},
%\end{align*}
%where for the second equality we used $\zeta e=e_1\cdots e_{n-1}e^2$ with $e^2\in Z^\circ$. Since $n$ is even, $e_1\cdots e_{n-1}\in\GSpin(V)$. The corollary follows.
%\end{proof}

Let us mention the following, though we do not use it in this paper.
\begin{cor}\label{cor:conjugate_GSpin_2}
Assume $n=2k-1$. Each semisimple $g\in\GSpin(V)$ is conjugate to $\sigma_V(g)$ in $\GSpin(V)$.
\end{cor}
\begin{proof}
This can be shown as in the above corollary. Namely for $n=2k-1$, the conjugating element $\eta$ in the theorem can be chosen both from $\GSpin(V)$ and $\GPin(V)\smallsetminus\GSpin(V)$.
\end{proof}

\subsection{Centralizer for $\GPin$}
Let $g\in\GPin(V)$ be semisimple and $h=P(g)\in\OO(V)$ as before. Also we write
\[
\OO(V)_h\simeq G'\times \OO(V_+)\times\OO(V_-),
\]
where $G'=G_1\times\cdots\times G_m$ is a product of (restrictions of scalar of) general linear groups and unitary groups, and we write
\[
V=V'\oplus V_+\oplus V_-=V_1\oplus\cdots\oplus V_m\oplus V_+\oplus V_-
\]
for the corresponding decomposition of $V$. Note that since $G'$ is Zariski connected, we have $G'\subseteq\SO(V')$.

Let $\GPin(V)_g$ be the centralizer of $g$ in $\GPin(V)$. It is immediate that
\[
P(\GPin(V)_g)\subseteq \OO(V)_h.
\]
This inclusion could be strict, depending on $g$, as follows.
\begin{lemma}\label{lemma:centralizer_GPin}
Keeping the above notation, we have
\[
P(\GPin(V)_g)=\begin{cases}G'\times \OO(V_+)\times\SO(V_-),&\text{if $\dim_FV_-$ is even};\\
G'\times \SO(V_+)\times\OO(V_-),&\text{if $\dim_FV_-$ is odd}.\end{cases}
\]
\end{lemma}
\begin{proof}
First recall that
\[
h=(h',h_+,h_-)\in G'\times \OO(V_+)\times\OO(V_-),
\]
where in particular $h_-=-1_{V_-}$. Let $\gamma\in P^{-1}(\OO(V)_h)$, so that $P(\gamma)P(g)P(\gamma)^{-1}=P(g)$, which implies $\gamma g\gamma^{-1}=zg$ for some $z\in Z^\circ$. By taking the Clifford norm on both sides, we have $z=\pm 1$, namely $\gamma g\gamma^{-1}=\pm g$. It suffices to show $\gamma g\gamma^{-1}=g$ if and only if $P(\gamma)$ is in $G'\times \OO(V_+)\times\SO(V_-)$ if $\dim_FV_-$ is even, and is in $G'\times \SO(V_+)\times\OO(V_-)$ if $\dim_FV_-$ is odd.

Let $C^\circ$ be the connected component of the center of $\OO(V)_h$. We then have
\[
h\in C^\circ h_-.
\]
Note that $P^{-1}(C^\circ)$ is an algebraic torus because it is an extension of the torus $C^{\circ}$ by the torus $Z^{\circ}$. Consider the morphism
\[
\varphi:P^{-1}(C^\circ)\zeta_-\longrightarrow \{\pm 1\}\subseteq Z^\circ,\quad a\mapsto \gamma a\gamma^{-1}a^{-1}.
\]
Since $P^{-1}(C^\circ)\zeta_-$ is Zariski connected, this morphism has to be constant; namely either $\varphi(a)=1$ for all $a$ or $\varphi(a)=-1$ for all $a$.

Now, we can write
\[
\gamma=\gamma'\gamma_+\gamma_-,
\]
where $\gamma'\in\GSpin(V')$, $\gamma_+\in\GPin(V_+)$ and $\gamma_-\in\GPin(V_-)$. Here it should be mentioned that $\gamma'$ is in $\GSpin(V')$ because $P(\gamma')\in G'\subset\SO(V')$.

Assume $\dim_FV_-$ is even, so that $\zeta_-$ is in the center of $\GSpin(V_-)$. If $\gamma_-\in\GSpin(V_-)$ then we have
\begin{align*}
\gamma \zeta_-&=\gamma'\gamma_+\gamma_-\zeta_-\\
&=\gamma'\gamma_+\zeta_-\gamma_-\\
&=\gamma'\zeta_-\gamma_+\gamma_-\quad (\text{Lemma \ref{lemma:commuting_elements_in_GPin}})\\
&=\zeta_-\gamma'\gamma_+\gamma_- \quad (\text{Lemma \ref{lemma:commuting_elements_in_GPin}})\\
&=\zeta_-\gamma.
\end{align*}
Hence $\varphi(\zeta_-)=1$, which implies the morphism $\varphi$ is identically 1. Thus $\gamma\in\GPin(V)_g$. If $\gamma_-\in\GPin(V_-)\smallsetminus\GSpin(V_-)$, then in the above computation of $\gamma \zeta_-$, we instead have $\gamma_-\zeta_-=-\zeta_-\gamma_-$ by Lemma \ref{lemma:zeta_almost_commutes_V_even}, which implies $\gamma\zeta_-=-\zeta_-\gamma$. Hence $\varphi(\zeta_-)=-1$; namely $\varphi$ is identically $-1$. Hence $\gamma\notin\GPin(V)_g$. This completes the proof when $\dim_FV_-$ is even.

Assume $\dim_FV_-$ is odd, so that $\zeta_-$ is in the center of $\GPin(V_-)$. If $\gamma_+\in\GSpin(V_+)$ then we have the same computation as above to show $\gamma\zeta_-=\zeta_-\gamma$, which shows $\gamma\in\GPin(V)_g$. If $\gamma_+\in\GPin(V_+)\smallsetminus\GSpin(V_+)$, then we instead have $\gamma_+\zeta_-=-\zeta_-\gamma_+$ by Lemma \ref{lemma:commuting_elements_in_GPin}, which implies $\gamma\zeta_-=-\zeta_-\gamma$ and hence $\gamma\notin\GPin(V)_g$.

\end{proof}

%%%%%%%%%%%%%%%%%%%%%%%%%%%%%%%%%%%%%%%%%%%%%%%%%%%%%%%%%%%%%%%%%%%%%%%%

\section{Contragredients}

%%%%%%%%%%%%%%%%%%%%%%%%%%%%%%%%%%%%%%%%%%%%%%%%%%%%%%%%%%%%%%%%%%%%%%%%

In this section we consider representations of $\GPin(V)$ and $\GSpin(V)$, and show that they are ``essentially self-dual".

\subsection{Sign character of representation}
Recall in \eqref{eq:sign_character} that we have defined the sign character $\sign:\GPin(V)\to\{\pm 1\}$ which sends the nonidentity component to $-1$. If we consider $\{\pm 1\}$ as a subset of $\C^\times$, we can view $\sign$ as a character on $\GPin(V)$. If we view $\{\pm 1\}$ as a subset of $F^\times=Z^\circ$, we can view $\sign$ as a homomorphism $\GPin(V)\to\GPin(V)$. We view $\sign$ in these two different ways depending on the context as follows.

For each $\pi\in\Irr(\GPin(V))$ we define the sign twist
\[
\sign\otimes\, \pi
\]
by $(\sign\otimes\,\pi)(g)=\sign(g)\pi(g)$, where we are viewing $\sign$ as a character.

We then define
\begin{equation}\label{eq:sign_of_pi}
\sign_{\pi}:\GPin(V)\longrightarrow \{\pm 1\}
\end{equation}
by
\[
\sign_{\pi}=\begin{cases}\one&\text{if $\pi(-1)=1$}\\\sign&\text{if $\pi(-1)=-1$},\end{cases}
\]
and call it the sign character associated with $\pi$. Then
\[
\sign_{\pi}(g)=\pi(\sign(g))=\omega_{\pi}(\sign(g))
\]
for all $g\in\GPin(V)$, where $\sign(g)\in\{\pm 1\}$ is viewed inside $F^\times$.

%Also it should be noted that
%\[
%\bar{g}=\sign(g) g^*,
%\]
%where $\sign(g)\in\{\pm 1\}$ is viewed inside $F^\times$, and $\bar{g}$ is the Clifford involution and $g^*$ is the canonical involution.

\subsection{Restrictions from $\GPin(V)$ to $\GSpin(V)$}
We make clear how representations of $\GPin(V)$ and $\GSpin(V)$ are related.

Assume $n=2k$. Let $\tau\in\Irr(\GSpin(V))$. Consider the induced representation $\Ind_{\GSpin(V)}^{\GPin(V)}\tau$. By elementary representation theory, there are two possibilities: either $\Ind_{\GSpin(V)}^{\GPin(V)}\tau$ is irreducible or $\Ind_{\GSpin(V)}^{\GPin(V)}\tau=\pi_1\oplus\pi_2$ for some $\pi_1,\pi_2\in\Irr(\GPin(V))$. The former is the case if and only if
\[
\tau^{\delta}\not\simeq\tau,
\]
where $\tau^\delta$ is the twist of $\tau$ by any element $\delta\in\GPin(V)\smallsetminus\GSpin(V)$, namely $\tau^\delta(g)=\tau(\delta g\delta^{-1})$ for $g\in\GSpin(V)$. The latter is the case if and only if $\pi_1$ and $\pi_2$ are two different extensions of $\tau$ to $\GPin(V)$, in which case we have $\pi_1=\sign\otimes\, \pi_2$. Furthermore any $\pi\in\Irr(\GPin(V))$ arises in either way. The following lemma will be used later.
\begin{lemma}\label{lemma:twist_by_sign_even_GPin}
Assume $n=2k$ and $\pi\in\Irr(\GPin(V))$. Then
\[
\pi\simeq \sign_{\pi}\otimes\, \pi,
\]
so that if $\omega_{\pi}(-1)=-1$ then we have $\sign\otimes\,\pi\simeq \pi$.
\end{lemma}
\begin{proof}
Let $\zeta\in Z_{\GSpin(V)}$ be the central element as in \eqref{eq:shimura_zeta} and let $\pi^\zeta$ be the representation of $\GPin(V)$ defined by $\pi^{\zeta}(g)=\pi(\zeta g\zeta^{-1})$, so that $\pi\simeq\pi^{\zeta}$. But Lemma \ref{lemma:zeta_almost_commutes_V_even} implies $\zeta g\zeta^{-1}=\sign(g)g$ for all $g\in\GPin(V)$. Hence we have $\pi^{\zeta}(g)=\pi(\sign(g)g)=\sign_{\pi}(g)\pi(g)$, namely $\pi\simeq \sign_{\pi}\otimes\, \pi$.

%If $\omega_{\pi}(-1)=1$ then $\sign_{\pi}=\one$, and there is nothing to show. Assume $\omega_{\pi}(-1)=-1$, so that $\sign_{\pi}=\sign$ and we have to show $\pi\simeq\sign\otimes\,\pi$. Assume $\pi\not\simeq\sign\otimes\,\pi$. From the above discussion we know that the restriction $\tau:=\pi|_{\GSpin(V)}$ is irreducible, and hence $\tau\simeq\tau^\delta$ for any $\delta\in\GPin(V)\smallsetminus\GSpin(V)$. Let $\zeta=e_1\cdots e_n$ be as in \eqref{eq:shimura_zeta}, so that $\zeta$ is in the center of $\GSpin(V)$. Let $\delta=e_1$. Then one can see $e_1\zeta e_1^{-1}=-\zeta$. Since $\tau\simeq\tau^\delta$, their central characters are equal, which implies
%\[
%\omega_\tau(\zeta)=\omega_{\tau}(-\zeta),
%\]
%which gives $\omega_{\tau}(-1)=1$. But $\omega_{\tau}$ and $\omega_{\pi}$ agree on $Z^\circ$, and hence we must have $\omega_{\pi}(-1)=1$, which is a contradiction.
\end{proof}

%\begin{rem}
%The above lemma also reflects the fact that if $n=2k$ then $g$ and $\sign(g)g$ are conjugate in $\GPin(V)$. Indeed, Lemma \ref{lemma:zeta_almost_commutes_V_even} implies $\zeta g\zeta^{-1}=\sign(g)g$ for all $g\in\GPin(V)$, where $\zeta\in Z_{\GSpin(V)}$. Hence by Harish-Chandra's regularity theorem, we have $\pi\simeq\pi^{\alpha}$, where $\pi^{\alpha}$ is defined by $\pi^{\alpha}(g)=\pi(\sign(g)g)$. But one can see $\pi^\alpha=\sign_{\pi}\otimes\,\pi$, which proves the above lemma.
%\end{rem}

Assume $n=2k-1$, in which case we know
\[
\GPin(V)=\GSpin(V)\cup\GSpin(V)\zeta,
\]
where $\zeta$ is as in \eqref{eq:shimura_zeta}. Let $\tau\in\Irr(\GSpin(V))$. Since $\zeta$ is in the center of $\GPin(V)$ and in particular commutes with all the elements in $\GSpin(V)$, we know that $\tau^\zeta\simeq\tau$. Namely, $\tau$ always admits two different extensions $\pi_1$ and $\pi_2=\sign\otimes\, \pi_1$, and any  $\pi\in\Irr(\GPin(V))$ arises in this way.

\subsection{Character twists}
Let $\omega:F^\times\longrightarrow\C^\times$ be a character. For each $\pi\in\Irr(\GPin(V))$, we define
\[
\omega\otimes\pi:=(\omega\circ N)\otimes \pi,
\]
namely $(\omega\otimes\pi)(g)=\omega(N(g))\pi(g)$ for $g\in\GPin(V)$, where we recall $N:\GPin(V)\to F^\times$ is the Clifford norm \eqref{eq:Clifford_norm}. Also for $\tau\in\Irr(\GSpin(V))$ we define $\omega\otimes\tau$ in the same way.

\subsection{Contragredients}
In this subsection, we prove Theorem \ref{thm:contragredient_introduction}, which follows from Theorem \ref{thm:semisimple_conjugate} and Harish-Chandra's regularity theorem.

%its corollaries combined with the following lemma.
%\begin{lemma}\label{lemma:Bernsetin_Zelvinski_contragredient}
%Let $G$ be a $p$-adic reductive group and $G_{s}$ the set of semisimple elements of $G$. Assume $\sigma:G\to G$ is a homeomorphism such that
%\begin{enumerate}[(a)]
%\item $\sigma(G_{s})=G_{s}$;
%\item $\sigma(g)^{-1}=\sigma(g^{-1})$ for all $g\in G$;
%\item $\sigma^2=1$;
%\item $g$ and $\sigma(g)$ are conjugate for all $g\in G_s$.
%\end{enumerate}
%Then every $G$-invariant distribution on $G_{s}$ is invariant under $\sigma$.
%\end{lemma}
%\begin{proof}
%This is a very special case of \cite[Theorem 6.13]{BZ76}. First, note that since $G_{s}$ is open in $G$ it is an $\ell$-space in the sense of \cite{BZ76} by \cite[Lemma 1.2]{BZ76}. Second, note that the conjugation action of $G$ on itself is constructive by \cite[Theorem A, p.57]{BZ76}, and hence the action of $G$ on $G_{s}$ is constructive. (Note that an action of $G$ on $X$ is said to be constructive if its graph as defined in \cite[Section 6.1]{BZ76} is constructive, namely a finite union of locally closed sets. Hence if the action is restricted to an open subset of $X$ then the resulting action is also constructive.)
%
%One can then check that the conditions of the lemma satisfy the conditions of \cite[Theorem 6.13]{BZ76} with $X=G_{s}$, $\gamma$ the conjugation action of $G$ on $G_{s}$, $g^\sigma=\sigma(g)^{-1}$, $n=2$ and $g_0=1$ in the notation there.
%\end{proof}

To describe our theorem, for $\pi\in\Irr(\GPin(V))$ and $\tau\in\Irr(\GSpin(V))$, we define $\pi^{\sigma}\in\Irr(\GPin(V))$ and $\tau^{\sigma}\in\Irr(\GSpin(V))$ by
\[
\pi^{\sigma}(g)=\pi(\sigma_V(g)^{-1}) \qand
\tau^{\sigma}(g)=\begin{cases}\tau(e^k\sigma_V(g)^{-1}e^{-k}),&\text{if $n=2k$};\\ \tau(\sigma_V(g)^{-1}),&\text{if $n=2k-1$},\end{cases}
\]
respectively.

We then have the following
\begin{theorem}\label{thm:contragredient_GPin}
\quad
\begin{enumerate}[(a)]
\item Let $\pi\in\Irr(\GPin(V))$. Then
\[
\pi^\vee\simeq\pi^{\sigma}=\begin{cases}\omega_{\pi}^{-1}\otimes\pi,&\text{if $n=2k$};\\ \sign_{\pi}^{k}\omega_{\pi}^{-1}\otimes\pi,&\text{if $n=2k-1$}.\end{cases}
\]
Here by $\sign_{\pi}^{k}\omega_{\pi}^{-1}$ we mean the character on $\GPin(V)$ defined by
\[
g\mapsto \sign_{\pi}(g)^{k}\cdot\omega_{\pi}(N(g))^{-1},
\]
where strictly speaking the central character $\omega_\pi$ is restricted to the connected component $Z^\circ$ of the center for $n=2k-1$.
\item Let $\tau\in\Irr(\GSpin(V))$. Then
\[
\tau^\vee\simeq\tau^{\sigma}=\begin{cases}\omega_{\tau}^{-1}\otimes\tau^{\delta}&\text{if $n=2k$ with $k$ odd};\\
\omega_{\tau}^{-1}\otimes\tau&\text{otherwise},\end{cases}
\]
where $\delta$ is any element in $\GPin(V)\smallsetminus\GSpin(V)$.
\end{enumerate}
\end{theorem}
\begin{proof}
For each semisimple $g\in\GPin(V)$, we know by Theorem \ref{thm:semisimple_conjugate} that $\sigma_V(g)$ and $g$ are conjugate in $\GPin(V)$. Hence the assertion $\pi^{\vee}\simeq\pi^{\sigma}$ can be proven by Harish-Chandra's regularity theorem. Indeed, if $\Theta_{\pi^\vee}$ is the distribution character of $\pi^\vee$ then we have $\Theta_{\pi^\vee}(g)=\Theta_\pi(g^{-1})$ for $g\in\GPin(V)$. Since $\sigma_V(g)$ and $g$ are conjugate, we have $\Theta_\pi(g^{-1})=\Theta_{\pi}(\sigma_V(g)^{-1})$, the latter being the distribution character of $\pi^\sigma$. Hence $\pi^\vee\simeq\pi^{\sigma}$.

Now, since
\[
\sigma_V(g)^{-1}=\begin{cases}{g^*}^{-1}=\sign(g)N(g)^{-1}g,&\text{if $n=2k$};
\\\sign(g)^{k+1}{g^*}^{-1}=\sign(g)^{k}N(g)^{-1}g,&\text{if $n=2k-1$},\end{cases}
\]
we have
\[
\pi^{\sigma}=\begin{cases}\omega_{\pi}^{-1}\otimes\pi&\text{if $n=2k$};\\ \sign_{\pi}^{k}\omega_{\pi}^{-1}\otimes\pi&\text{if $n=2k-1$},\end{cases}
\]
where we used Lemma \ref{lemma:twist_by_sign_even_GPin} for $n=2k$.

The case for $\GSpin(V)$ can be proven in the same way. Namely, first one can show $\tau\simeq\tau^{\sigma}$, and then show $\tau^{\sigma}$ is described as in the theorem by using $\sign(g)=1$ for $g\in\GSpin(V)$ and $e\in\GPin(V)\smallsetminus\GSpin(V)$.

%Next consider $\GSpin(V)$. So let $\tau\in\Irr(\GSpin(V))$. Assume $n=2k$. Let $\delta\in\GPin(V)\smallsetminus\GSpin(V)$ be fixed. If $\tau^\delta\simeq\tau$ then $\tau$ extends to an irreducible admissible representation $\pi$ of $\GPin(V)$. Since $\pi^\vee=\omega_{\pi}^{-1}\otimes\pi$, we certainly have $\tau^\vee=\omega_{\tau}^{-1}\otimes\tau$ simply by restricting to $\GSpin(V)$. If $\tau^\delta\not\simeq\tau$ then $\Ind_{\GSpin(V)}^{\GPin(V)}\tau$ is irreducible, so that
%\[
%(\Ind_{\GSpin(V)}^{\GPin(V)}\tau)^{\vee}=\omega_{\tau}^{-1}\otimes\Ind_{\GSpin(V)}^{\GPin(V)}\tau.
%\]
%Now
%\[
%(\Ind_{\GSpin(V)}^{\GPin(V)}\tau)\Big|_{\GSpin(V)}=\tau\oplus\tau^{\delta}.
%\]
%By taking the contragredients of both sides, we know that $\tau^\vee=\omega_{\tau}^{-1}\otimes\tau$ or $\omega_{\tau}^{-1}\otimes\tau^{\delta}$.
%
%Assume $n=2k-1$. In this case, since the restriction of $\pi$ to $\GSpin(V)$ is always irreducible, the result immediately follows from the $\GPin(V)$ case, noting that $\sign_{\pi}$ is always trivial when restricted to $\GSpin(V)$.
\end{proof}

\begin{rem}
If $\omega_{\pi}=\one$ then the representation $\pi$ factors through the canonical projection $P:\GPin(V)\to\OO(V)$ and hence it can be viewed as a representation of $\OO(V)$. Then the theorem recovers the well-known fact that every irreducible admissible representation of $\OO(V)$ is self-dual.
\end{rem}

\begin{rem}
In the above proof, we used Harish-Chandra's regularity theorem, which is known only for characteristic zero. This is crucial because we can prove Theorem \ref{thm:semisimple_conjugate} only for the semisimple elements $g\in\GPin(V)$. If we could prove the analogous theorem for all the elements, then we could get away with Harish-Chandra's regularity theorem, which would allow us to extend the above theorem to any $p$-adic field of characteristic different from 2.
\end{rem}

\begin{rem}
For any smooth (not necessarily irreducible) representation $\pi$, define $\pi^{\sigma}$ as above. Then the assignment $\pi\mapsto \pi^{\sigma}$ is a covariant exact functor on the category of smooth representations of $\GPin(V)$. Indeed, this is an MVW-involution for $\GPin(V)$. Similarly, $\tau\mapsto\tau^{\sigma}$ is an MVW-involution for $\GSpin(V)$. It should be noted that the existence of an MVW-involution for $\GSpin(V)$ is also proven in \cite{Mondal-Nadimpalli} by a completely different method from ours. Also see \cite{Prasad} more on MVW-involution in general.
\end{rem}

\begin{rem}
If $n=2k-1$ and $\disc(V)=1$, then by Proposition \ref{prop:semidirect_product} we have $\GPin(V)\simeq\GSpin(V)\times\{1, \zeta\}$, where $\zeta^2=1$. Hence any $\pi\in\Irr(\GPin(V))$ is of the form $\tau\otimes\chi$, where $\tau=\pi|_{\GSpin(V)}\in\Irr(\GSpin(V))$ and $\chi=\omega_{\pi}|_{\{1, \zeta\}}$. Thus we must have $\pi^\vee\simeq\tau^\vee\otimes\chi$. But by the above theorem we must have $\tau^\vee=\omega_{\tau}^{-1}\otimes\tau$. Hence we must have
\[
\sign_{\pi}^k\omega_{\pi}^{-1}\otimes\pi\simeq (\omega_{\tau}^{-1}\otimes\tau)\otimes \chi.
\]
One can indeed verify this by using $N(\zeta)=(-1)^k\zeta^2=(-1)^k$, so that $(\sign_{\pi}^k\omega_{\pi}^{-1})(\zeta)=1$. The detail is left to the reader.
\end{rem}

%\begin{rem}
%Later in Proposition \ref{prop:contragredient_GSpin}, we will show that for $n=2k$ we always have $\tau^\vee\simeq\omega_{\tau}^{-1}\otimes\tau$ if $k$ is even, and always have $\tau^\vee\simeq\omega_{\tau}^{-1}\otimes\tau^{\delta}$ if $k$ is odd, though we do not need this fact for our purposes.
%\end{rem}

%%%%%%%%%%%%%%%%%%%%%%%%%%%%%%%%%%%%%%%%%%%%%%%%%%%%%%%%%%%%%%%%%%%%%%%%

\section{The group $\GPint(V)$.}

%%%%%%%%%%%%%%%%%%%%%%%%%%%%%%%%%%%%%%%%%%%%%%%%%%%%%%%%%%%%%%%%%%%%%%%%

In this section, we define a group $\GPint(V)$, which is isomorphic to the direct product $\GPin(V)\times\{\pm 1\}$ but defined more intrinsically so that we can naturally reduce our situation to the classical group situations of \cite{AGRS} and \cite{Wal12}. We also recall the analogous groups for the classical groups treated in \cite{AGRS} and \cite{Wal12}. In this section, we deviate from our convention that $V$ is a quadratic space over $F$ but we allow $V$ to be other spaces.

\subsection{General linear group}
Let $V$ be a vector space over $F$ of $\dim_FV=n$, and fix a basis $\{e_1,\dots,e_n\}$ of $V$. Let $V^*$ be the dual space of $V$ and write $\{e_1^*,\dots,e_n^*\}$ for the dual basis. Note that for each $g\in\GL(V)$ there exists a unique $g^*\in\GL(V^*)$ such that $\la gv, v^*\ra=\la v, g^*v^*\ra$ for all $v\in V$ and $v^*\in V^*$, where $\la-,-\ra$ is the canonical pairing. We have the embedding
\[
\Delta:\GL(V)\longrightarrow\GL(V)\times\GL(V^*),\quad g\mapsto \Delta g:=(g, {g^*}^{-1}).
\]
Let
\[
u:V^*\longrightarrow V
\]
be defined by $u(e_i^*)=e_i$ and
\begin{equation}\label{eq:beta_for_GL}
\beta:V\oplus V^*\longrightarrow V\oplus V^*,\quad \beta=\begin{pmatrix}&u\\u^{-1}&\end{pmatrix},
\end{equation}
where the matrix is with respect to the fixed basis. We define the involution
\begin{equation}\label{eq:tau_for_GL}
\tau_V:\GL(V)\longrightarrow\GL(V),\quad g\mapsto g^t,
\end{equation}
where $g^t$ is the transpose of $g$ with respect to our fixed basis. The important property of $\tau_V$ that is used in \cite{AGRS} is that $\tau_V$ preserves the (semisimple) conjugacy classes. Also note that
\begin{equation}\label{eq:beta_conjugation_GL}
\beta\Delta g\beta^{-1}=\Delta (g^t)^{-1}.
\end{equation}

We define
\[
\GLt(V)=\big\la\Delta g,\, \beta\st g\in\GL(V)\big\ra\subseteq\Aut(V\oplus V^*),
\]
namely the group generated by $\Delta g$'s and $\beta$ inside $\Aut(V\oplus V^*)$. We then have the short exact sequence
\[
1\longrightarrow\GL(V)\longrightarrow\GLt(V)\xrightarrow{\;\,\chi\;\,}\{\pm1\}\longrightarrow 1,
\]
where the inclusion is $g\mapsto \Delta g$ and the surjection $\chi$ sends all the $\Delta g$'s to $1$ and $\beta$ to $-1$. Note that
\[
\GLt(V)\simeq\GL(V)\rtimes\{1, \beta\},
\]
where $\beta$ acts on $\GL(V)$ as inverse-transpose. We define an action of $\GLt(V)$ on the set $\GL(V)\times (V\oplus V^*)$ by
\begin{gather}\label{eq:action_of_GLt}
\begin{aligned}
\Delta g\cdot (h, v+v^*)&=(ghg^{-1},\, \Delta g(v+ v^*))\\
\beta\cdot (h, v+v^*)&=(\tau_V(h), \beta(v+v^*))=(h^t,\, \beta(v+v^*)).
\end{aligned}
\end{gather}
Here the action of $\beta$ on $\GL(V)$ is via the involution $\tau_V$. It is important that the action of $\beta$ preserves the semisimple conjugacy classes of $\GL(V)$. Also \eqref{eq:beta_conjugation_GL} guarantees that this is indeed an action.

Let
\[
W=\Span\{e_1,\dots, e_{n-1}\}\subseteq V,
\]
and set $e=e_n$, so that $V\oplus V^*=(W\oplus Fe)\oplus(W^*\oplus Fe^*)$. Note that $\GL(V)_e=\GL(W)$, where we recall that $\GL(V)_e$ is the stabilizer of $e$. We have
\[
\GLt(V)_{e+e^*}=\big\la \Delta g,\ \beta\st g\in\GL(W)\big\ra,
\]
and define
\[
\GLt(W):=\GLt(V)_{e+e^*}.
\]
Note that we have
\[
1\longrightarrow\GL(W)\longrightarrow\GLt(W)\xrightarrow{\;\,\chi\;\,}\{\pm1\}\longrightarrow 1,
\]
and
\[
\GLt(W)\simeq\GL(W)\rtimes\{1, \beta\},
\]
where the action of $\beta$ on $\GL(W)$ is the conjugation action inside $\GLt(W)$. (Note that by definition $\GLt(W)$ is the stabilizer $\GLt(V)_{e+e^*}$, and is not the $\GL(W)$ analogue of $\GLt(V)$, though they are certainly isomorphic. We prefer to view $\GLt(W)$ as the subgroup of $\GLt(V)$ stabilizing $e+e^*$.)

\subsection{Unitary and orthogonal group}
Let $(V, \la-,-\ra)$ be a Hermitian space over a quadratic extension $E$ of $F$ or a quadratic space over $F$, and $G(V)$ the corresponding isometry group, so that if $V$ is Hermitian then $G(V)=U(V)$ (unitary group), and if $V$ is quadratic then $G(V)=\OO(V)$ (orthogonal group). If $V$ is quadratic, we set $E=F$.  There is an $F$-linear map
\[
\beta:V\longrightarrow V
\]
such that $\la \beta v, \beta v'\ra=\la v', v\ra$ for all $v, v'\in V$. (Such $\beta$ differs by an element in $G(V)$; namely if $\beta'$ is another such $F$-linear map then there exists $h\in G(V)$ such that $\beta'=h\beta$.) We fix our $\beta$ as follows. Let $\{e_1,\dots, e_n\}$ be an orthogonal basis of $V$. Define
\begin{equation}\label{eq:beta_for_classical_group}
\beta:V\longrightarrow V,\quad \beta(a_1e_1+\cdots+a_ne_n)=\bar{a}_1e_1+\cdots+\bar{a_n}e_n,
\end{equation}
where $a_i\in E$ and the bar indicates the Galois conjugate. Note that $\beta^2=1$. If $G(V)$ is a unitary group, $\beta\notin G(V)$. If $G(V)$ is an orthogonal group, $\beta$ is simply the identity, which is in $G(V)$, but we consider $\beta$ not as an element in $G(V)$ but another element distinct from the element in $G(V)$.

We define
\[
\Gt(V)=\big\la g,\, \beta\st g\in G(V)\big\ra,
\]
namely the group generated by $G(V)$ and $\beta$. To be more precise, if $G(V)$ is unitary, $\Gt(V)$ is viewed inside $\Aut(V)$. If $G(V)$ is orthogonal, one may simply consider $\Gt(V)$ as the group generated by $G(V)$ and $\beta$ with the relations $g\beta=\beta g$ for all $g\in G(V)$ and $\beta^2=1$. We have the short exact sequence
\[
1\longrightarrow G(V)\longrightarrow\Gt(V)\xrightarrow{\;\,\chi\;\,}\{\pm1\}\longrightarrow 1,
\]
where the surjection $\chi$ sends $\beta$ to $-1$. Note that
\[
\Gt(V)\simeq G(V)\rtimes\{1, \beta\},
\]
where $\beta$ acts on $G(V)$ by conjugation, and hence if $G(V)$ is orthogonal, this is merely a direct product.

Let
\begin{equation}\label{eq:tau_for_classical}
\tau_V:G(V)\longrightarrow G(V),\quad \tau_V(g)=\beta g^{-1}\beta^{-1}.
\end{equation}
This is an involution on $G(V)$ which preserves the semisimple conjugacy classes of $G(V)$.

We define an action of $\Gt(V)$ on the set $G(V)\times V$ by
\begin{gather}\label{eq:action_of_Gt(V)}
\begin{aligned}
g\cdot (h, v)&=(ghg^{-1}, gv)\\
\beta\cdot(h, v)&=(\tau_V(h), -\beta v)=(\beta h^{-1}\beta^{-1}, -\beta v)
\end{aligned}
\end{gather}
for $g\in G(V)$ and $(h, v)\in G(V)\times V$. As in the $\GL(V)$ case, the action of $\beta$ on $G(V)$ is via $\tau_V$, so that this action preserves the semisimple conjugacy classes. One can readily verify that this is indeed an action.

Let $W=\{e_1,\dots,e_{n-1}\}$, so that $V=W\oplus Ee$, where we set $e=e_n$ as usual. We have $G(V)_e=G(W)$, and
\[
\Gt(V)_e=\big\la g,\, r_e\beta\st g\in G(W)\big\ra,
\]
where we recall that $r_e$ is the reflection in the hyperplane orthogonal to $e$, and define
\[
\Gt(W)\simeq \Gt(V)_e.
\]
Note that we have
\[
1\longrightarrow G(W)\longrightarrow\Gt(W)\xrightarrow{\;\,\chi\;\,}\{\pm1\}\longrightarrow 1,
\]
and
\[
\Gt(W)\simeq G(W)\rtimes\{1, r_e\beta\},
\]
where $r_e\beta$ acts on $G(W)$ by conjugation viewed inside $\Gt(V)$, and $\chi$ sends $r_e\beta$ to $-1$. It should be noted that $r_e$ and $\beta$ commute and $(r_e\beta)^2=1$. (Note that $\Gt(W)$ is not the group analogous to $\Gt(V)$, though certainly isomorphic to it.)

%Finally, we define another involution
%\begin{equation}\label{eq:tau_e_on_classical_group}
%\tau_V_e:G(V)\longrightarrow G(V),\quad \tau_V_e(g)=(r_e\beta) g^{-1}(r_e\beta)^{-1}
%\end{equation}
%for $g\in G(V)$. Since $r_e$ acts trivially on $W$, we have $\tau_V_e(G(W))=G(W)$. This involution is precisely the action of the element $r_e\beta$ in the stabilizer $\Gt(V)_e$.

\subsection{Special orthogonal group}
We consider the special orthogonal group $\SO(V)$. We keep the notation of the previous subsection, assuming $V$ is quadratic. We define
\[
\SOt(V)=\big\la g,\, r_e^k\beta\st g\in\SO(V)\big\ra\subseteq\OOt(V),
\]
where we recall $k$ is such that $\dim_FV=2k$ or $2k-1$, and $r_e$ is the reflection in the hyperplane orthogonal to $e$. We have
\[
1\longrightarrow \SO(V)\longrightarrow\SOt(V)\xrightarrow{\;\,\chi\;\,}\{\pm1\}\longrightarrow 1,
\]
and
\[
\SOt(V)\simeq \SO(V)\rtimes\{1, r_e^k\beta\},
\]
where $r_e^k\beta$ acts on $\SO(V)$ by conjugation viewed inside $\OOt(V)$.

\begin{rem}
Let us compare our setup with that of Waldspurger (\cite[p.314]{Wal12}). He defines the group
\[
\Gb=\{(g, \epsilon)\in\OO(V)\times\{\pm 1\}\st \det g=\epsilon^k\}.
\]
One can readily verify that $\Gb$ is isomorphic to the direct product $\SO(V)\times\{\pm 1\}$ if $k$ is even, and isomorphic to $\OO(V)$ if $k$ is odd. One can then see that our $\SOt(V)$ is isomorphic to $\Gb$ via the map $g\mapsto (g, 1)$ and $r_e^k\beta\mapsto (r_e^k, -1)$.
\end{rem}

Note that the involution
\begin{equation}\label{eq:tau_for_special_orthogonal}
\SO(V)\longrightarrow \SO(V),\quad g\mapsto r_e^k\tau_V(g)r_e^{-k}=r_e^kg^{-1}r_e^{-k}.
\end{equation}
preserves the (semisimple) conjugacy classes of $\SO(V)$. (This is essentially the $\SO(V)$ analogue of Corollary \ref{cor:conjugate_GSpin}, and can be shown not just for the semisimple elements but for all elements.)

We define an action of $\SOt(V)$ on $\SO(V)\times V$ by restricting the action of $\OOt(V)$, so that
\begin{gather}\label{eq:action_of_SOt(V)}
\begin{aligned}
g\cdot (h, v)&=(ghg^{-1}, gv)\\
r_e^k\beta\cdot(h, v)&=(r_e^k h^{-1} r_e^{-k}, -r_e^k(v))
\end{aligned}
\end{gather}
for $g\in\SO(V)$ and $(h, v)\in\SO(V)\times V$. In particular, the action of $r_e^\beta$ is via the involution that preserves the semisimple classes in $\SO(V)$. (Indeed, the group $\SOt(V)$ is so chosen that we have this property.)

We then have
\[
\SOt(V)_e=\big\la g,\, r_{e_{n-1}}^{k-1}r_e\beta\st g\in\SO(W)\big\ra,
\]
and define
\[
\SOt(W):=\SOt(V)_e.
\]
We have
\[
1\longrightarrow \SO(W)\longrightarrow\SOt(W)\xrightarrow{\;\,\chi\;\,}\{\pm1\}\longrightarrow 1,
\]
and
\[
\SOt(W)\simeq \SO(W)\rtimes\{1, r_{e_{n-1}}^{k-1}r_e\beta\},
\]
where $r_{e_{n-1}}^{k-1}r_e\beta$ acts by conjugation viewed inside $\OOt(V)$, and $\chi$ sends $r_{e_{n-1}}^{k-1}r_e\beta$ to $-1$. It should be noted that $r_{e_{n-1}}$, $r_e$ and $\beta$ all commute and $(r_{e_{n-1}}^{k-1}r_e\beta)^2=1$. (Note that $\SOt(W)$ is not the $\SO(W)$ analogue of $\SOt(V)$. This time, it is not completely immediate that it is even isomorphic to the $\SO(W)$ analogue of $\SOt(V)$, though they are indeed isomorphic. But since we do not use this fact, we omit the proof.)

\subsection{Vanishing theorem for classical groups}
Let
\begin{align*}
G(V)&=\GL(V),\ \UU(V),\ \OO(V)\ \text{or}\ \SO(V);\\
\Gt(W)&=\GLt(W),\ \Ut(W),\ \OOt(W)\ \text{or}\ \SOt(W),
\end{align*}
respectively. We can write
\[
\Gt(W)=G(W)\rtimes\{1, \beta_W\},
\]
where
\[
\beta_W=\begin{cases}
\beta,&\text{for $\GL(V)$};\\
r_e\beta,&\text{for $\UU(V)$};\\
r_e\beta,&\text{for $\OO(V)$};\\
r_{e_{n-1}}^{k-1}r_e\beta&\text{for $\SO(V)$},
\end{cases}
\]
where $\beta_W$ acts on $G(W)$ by conjugation viewed inside $\Gt(V)$.

We define the involution
\[
\tau_W:G(V)\longrightarrow G(V)
\]
by the action of $\beta_W$ as defined in \eqref{eq:action_of_GLt}, \eqref{eq:action_of_Gt(V)} and \eqref{eq:action_of_SOt(V)}, respectively. Namely,
\[
\tau_W(g)=\begin{cases}
\tau_V(g)=g^t,&\text{for $\GL(V)$};\\
(r_e\beta) g^{-1}(r_e\beta)^{-1},&\text{for $\UU(V)$};\\
r_eg^{-1}r_e^{-1},&\text{for $\OO(V)$};\\
(r_{e_{n-1}}^{k-1}r_e)g^{-1}(r_{e_{n-1}}^{k-1}r_e)^{-1},&\text{for $\SO(V)$}.
\end{cases}
\]
We let $\Gt(W)$ act on the set $G(V)$ by using this action of $\beta_W$, so that
\[
h\cdot g=hgh^{-1}\qand \beta_W\cdot g=\tau_W(g),\qquad h\in G(W),\; g\in G(V).
\]
This is indeed an action. Note that here $G(V)$ is viewed merely as a set instead of a group. (Let us note that the involution $\tau_W$ is not explicitly mentioned in \cite{AGRS} or \cite{Wal12}.)

Aizenbud, Gourevitch, Rallis and Schiffmann (\cite{AGRS}) and Waldspurger (\cite{Wal12}) proved their multiplicity-at-most-one theorem by reducing it to the following non-existence of invariant distributions.
\begin{prop}\label{pro:AGRS_non-exitence_of_distribution}
If we denote by $\Scal'(G(V))^{\Gt(W), \chi}$ the space of the distributions on which $\Gt(W)$ acts via $\chi$ then
\[
\Scal'(G(V))^{\Gt(W), \chi}=0.
\]
Namely, every $G(W)$-invariant distribution on $G(V)$ is also invariant under the involution $\tau_W$.
\end{prop}

This proposition is rephrased as follows. Let $\Scal'(G(V))^{G(W)}$ be the space of $G(W)$-invariant distributions on $G(V)$. Since $\tau_W(G(W))=G(W)$ this space is closed under the action of $\tau_W$. Further because $\tau_W^2=1$, we have the eigenspace decomposition
\[
\Scal'(G(V))^{G(W)}=\Scal'(G(V))^{G(W), +}\oplus\Scal'(G(V))^{G(W), -},
\]
where $\tau_W$ acts as $\pm 1$ on $\Scal'(G(V))^{G(W), \pm}$, respectively. Then
\[
\Scal'(G(V))^{\Gt(W),\chi}=\Scal'(G(V))^{G(W), -}.
\]
Namely, the proposition asserts that the $-1$-eigenspace of $\tau_W$ in $\Scal'(G(V))^{G(W)}$ is zero, which means that every $G(W)$-invariant distribution on $G(V)$ is also invariant under $\tau_W$.

It has been shown by \cite{AGRS} and \cite{Wal12} that the vanishing of this space of distributions follows from the following vanishing theorem, which is the main technical theorem proven in \cite{AGRS} and \cite{Wal12}.
\begin{prop}\label{pro:AGRS_non-exitence_of_distribution_times_V}
Assume $G(V)=\GL(V)$. Then
\[
\Scal'(\GL(V)\times (V\oplus V^*))^{\GLt(V), \chi}=0.
\]
Assume $G(V)=\UU(V), \OO(V)$ or $\SO(V)$. Then
\[
\Scal'(G(V)\times V)^{\Gt(V), \chi}=0.
\]
\end{prop}

Let us note that in this proposition the space $W$ no longer appears, and the involution $\tau_W$ does not play any direct role in the proof after all. Probably, this is why the involution $\tau_W$ is not explicitly mentioned in \cite{AGRS} or \cite{Wal12}.

\quad

\subsection{GPin}
We need to prove the analogue of Proposition \ref{pro:AGRS_non-exitence_of_distribution} for $\GPin(V)$. In this subsection we set up our notation for $\GPin(V)$. So we go back to our convention that $(V, q)$ is a quadratic space over $F$ as before.

Since the conjugation action of $\GPin(V)$ on itself factors through $\OO(V)$, the $\GPin$ analogue of $\OOt(V)$ might as well be defined to be equal to $\OOt(V)$. However, to make the distinction between the two, we define
\[
\GPint(V)=\big\la g,\, \beta\st g\in\GPin(V)\big\ra
\]
with the relations $g\beta=\beta g$ for all $g\in\GPin(V)$ and $\beta^2=1$. We have the short exact sequence
\[
1\longrightarrow \GPin(V)\longrightarrow\GPint(V)\xrightarrow{\;\,\chi\;\,}\{\pm1\}\longrightarrow 1,
\]
where the surjection $\chi$ sends $\beta$ to $-1$, and
\[
\GPint(V)\simeq\GPin(V)\times\{1, \beta\},
\]
which is a direct product.

Recall in \eqref{eq:Clifford_involution} that we have defined the involution
\[
\sigma_V:\GPin(V)\longrightarrow\GPin(V),\quad
\sigma_V(g)=
\begin{cases}g^*,&\text{if $n=2k$};\\
\sign(g)^{k+1}g^*,&\text{if $n=2k-1$},
\end{cases}
\]
and have shown that $\sigma_V$ preserves the semisimple conjugacy classes of $\GPin(V)$.
We then define the action of $\GPint(V)$ on $\GPin(V)\times V$ (viewed merely as a set) by
\begin{gather}\label{eq:action_on_V_GPin}
\begin{aligned}
g\cdot (h, v)&=(ghg^{-1}, P(g)v)\\
\beta\cdot (h, v)&=(\sigma_V(h), -v)
\end{aligned}
\end{gather}
for $g\in\GPin(V)$ and $(h, v)\in\GPin(V)\times V$. Note that $\sigma_V(ghg^{-1})=g\sigma_V(h)g^{-1}$ because $\sigma_V(g)=\pm N(g)g^{-1}$, which implies that the above action is indeed an action. Just as in the case of classical groups, it is important that the action of $\beta$ preserves the semisimple conjugacy classes of $\GPin(V)$.

Let $W=\{e_1,\dots, e_{n-1}\}$, so that $V=W\oplus Fe$ with $e=e_n$ as before. We then have $\GPin(V)_e=\GPin(W)$ and
\[
\GPint(V)_e=\big\la g,\, e\beta\st g\in\GPin(W)\big\ra.
\]
We define
\[
\GPint(W):=\GPint(V)_e.
\]
We then have
\[
1\longrightarrow \GPin(W)\longrightarrow\GPint(W)\xrightarrow{\;\,\chi\;\,}\{\pm1\}\longrightarrow 1,
\]
where the surjection $\chi$ sends $e\beta$ to $-1$, and
\[
\GPint(W)\simeq\GPin(W)\rtimes\{1, e\beta\},
\]
where $e\beta$ acts on $\GPin(W)$ by conjugation viewed inside $\GPint(V)$. Note that $e$ and $\beta$ commute and $(e\beta)^2=1$. (Just as the classical group case, $\GPint(W)$ is not $\GPin(W)$ analogue $\GPin(V)$. Moreover, $\GPint(W)$ is not even isomorphic to the $\GPin(W)$ analogue of $\GPin(V)$. This is because the element $e$ does not commute with $\GPin(W)$ but only ``almost commutes" as in Lemma \ref{lemma:commuting_elements_in_GPin}. Yet, this fact will play no role in this paper.)

We define an involution
\begin{equation}\label{eq:tau_e_on_GPin}
\tau_W:\GPin(V)\longrightarrow\GPin(V),\quad \tau_W(g)=e\sigma_V(g)e^{-1}
\end{equation}
for $g\in\GPin(V)$. This involution is precisely the action of the element $e\beta\in\GPint(V)$. Since $e$ is orthogonal to $W$, Lemma \ref{lemma:commuting_elements_in_GPin} implies that $\tau_W(\GPin(W))=\GPin(W)$.

%Similarly, we define
%\[
%\GSpint(V)=\la\Int(g),\,\sigma_V\st g\in\GSpin(V)\ra\subseteq \Aut(\GSpin(V)),
%\]
%where this time $\sigma_V$ is simply the Clifford involution, which is actually equal to the canonical involution $^*$ because $\alpha$ is trivial when restricted to $\GSpin(V)$. We again have the analogous short exact sequence. Also again we have the natural inclusion
%\[
%\GSpint(W)\subseteq\GSpint(V)
%\]
%inside $\Aut(\GSpin(V))$.

The main technical goal of this paper is to prove the following theorem, which is precisely the analogue of Proposition \ref{pro:AGRS_non-exitence_of_distribution_times_V}.
\begin{theorem}\label{thm:main_theorem_distribution}
Let $\GPint(W)\simeq\GPin(W)\rtimes\{1, e\beta\}$ act on $\GPin(V)$ (viewed merely as a set) by letting $\GPin(W)$ act by conjugation and $e\beta$ by $\tau_W$. Then we have
\[
\Scal'(\GPin(V))^{\GPint(W), \chi}=0.
\]
In other words, every $\GPin(W)$-invariant distribution on $\GPin(V)$ is also invariant under the involution $\tau_W$.
\end{theorem}

%\subsection{Isomorphism between $\GPint(V)$ and $\OOt(V)$}
%The groups $\GPint(V)$ and $\OOt(V)$ are naturally isomorphic. Indeed, let
%\begin{equation}\label{eq:iso_GPint_OOt}
%P:\GPint(V)\longrightarrow\OOt(V)
%\end{equation}
%be defined by
%\[
%P(\Int(g))=\Int(P(g))\qand P(\sigma_V)=\sigma,
%\]
%where $P(g)$ is the canonical projection of $g\in\GPin(V)$. It is straightforward to verify that $P$ is an isomorphism of the groups, because both $\Int(g)$ and $\sigma_V$, which are maps $\GPin(V)\to\GPin(V)$, naturally descend to maps $Z^\circ\backslash\GPin(V)\to Z^\circ\backslash\GPin(V)$; namely we have the following commutative diagram
%\[
%\begin{tikzcd}
%\GPin(V)\ar[r, "\gt"]\ar[d, "P"']&\GPin(V)\ar[d, "P"]\\
%\OO(V)\ar[r, "P(\gt)"]&\OO(V)
%\end{tikzcd}
%\]
%for all $\gt\in\GPint(V)$. For the isomorphism \eqref{eq:iso_GPint_OOt} we use the same symbol $P$ as the canonical projection, but this is a natural choice because we have the following commutative diagram
%\[
%\begin{tikzcd}
%\GPin(V)\ar[r, "\Int"]\ar[d, "P"']&\GPint(V)\ar[d, "P"]\\
%\OO(V)\ar[r, "\Int"]&\OOt(V)\rlap{\, ,}
%\end{tikzcd}
%\]
%where the horizonal arrows are the maps $g\mapsto\Int(g)$.

%%%%%%%%%%%%%%%%%%%%%%%%%%%%%%%%%%%%%%%%%%%%%%%%%%%%%%%%%%%%%%%%%%%%%%%%

\section{Reduction to the vanishing of distributions}

%%%%%%%%%%%%%%%%%%%%%%%%%%%%%%%%%%%%%%%%%%%%%%%%%%%%%%%%%%%%%%%%%%%%%%%%

In this section, we reduce our main theorem to the above vanishing theorem of the distributions (Theorem \ref{thm:main_theorem_distribution}). The key technical ingredient is the following, which is \cite[Corollary 1.1]{AGRS}.
\begin{lemma}\label{lemma:vanishing_involution_implies_at_most_one_first_lemma}
Let $G$ be an lctd group and $H\subseteq G$ a closed subgroup, both unimodular. Assume there exits an involution $\sigma:G\to G$ such that $\sigma(H)=H$ and every distribution on $G$ invariant under the conjugation action of $H$ is also fixed by $\sigma$; namely if $T\in\Scal'(G)^{H}$, then $\sigma\cdot T=T$, where the action of $\sigma$ on $T$ is defined in the obvious way. Then for all $\pi\in\Irr(G)$ and $\tau\in\Irr(H)$, we have
\[
\dim_{\CC}\Hom_H(\pi,\, \tau^\vee)\cdot \dim_{\CC}\Hom_H(\pi^\vee,\, \tau)\leq 1,
\]
where $\pi^\vee$ and $\tau^\vee$ are the contragredients.
\end{lemma}

By using this lemma, we have the following.
\begin{prop}\label{prop:vanishing_implies_main_theorem}
Theorem \ref{thm:main_theorem_distribution} (vanishing theorem of distributions) implies Theorem \ref{thm:A} (multiplicity-at-most-one theorem) for $\GPin$.
\end{prop}
\begin{proof}
Set
\[
(G, H)=(\GPin(V), \GPin(W))
\]
and $\sigma$ to be as in \eqref{eq:tau_e_on_GPin}, so that $\sigma(H)=H$. Also the vanishing theorem (Theorem \ref{thm:main_theorem_distribution}) says that every $H$-invariant distribution on $G$ is invariant under $\sigma$. Hence all the conditions of the above lemma are satisfied, which implies
\[
\dim_{\CC}\Hom_H(\pi,\, \tau^\vee)\cdot \dim_{\CC}\Hom_H(\pi^\vee,\, \tau)\leq 1
\]
for all $\pi\in\Irr(G)$ and $\tau\in\Irr(H)$. (Note that at this point we cannot conclude $\dim_{\CC}\Hom_H(\pi,\, \tau^\vee)\leq 1$ and $\dim_{\CC}\Hom_H(\pi^\vee,\, \tau)\leq 1$ because, if one of the Hom spaces is zero, then the other could still have dimension $>1$.)

To derive Theorem \ref{thm:A} for $\GPin$, it suffices to show
\begin{equation}\label{eq:hom_space_identity}
\Hom_{H}(\pi,\, \tau^\vee)= \Hom_{H}(\pi^\vee,\, \tau),
\end{equation}
because then the above inequality becomes
\[
\left(\dim_{\CC}\Hom_{H}(\pi,\, \tau^\vee)\right)^2\leq 1,
\]
and since $\tau$ is arbitrary we can re-choose our $\tau$ to be $\tau^\vee$.

To show \eqref{eq:hom_space_identity}, we use our description of the contragredients from Theorem \ref{thm:contragredient_GPin}. First, if the central characters $\omega_{\pi}$ and $\omega_{\tau^\vee}=\omega_{\tau}^{-1}$ do not agree on $Z^\circ$ then both $\Hom_{H}(\pi,\, \tau^\vee)$ and $\Hom_{H}(\pi^\vee,\, \tau)$ are zero, and hence \eqref{eq:hom_space_identity} trivially holds. So assume $\omega_\pi$ and $\omega_{\tau}^{-1}$ agree on $Z^\circ$, namely
\[
\omega_{\pi}|_{Z^\circ}=\omega_{\tau}^{-1}|_{Z^{\circ}}.
\]
Assume $\dim_FV=2k$ and $\dim_FW=2k-1$. We know from Theorem \ref{thm:contragredient_GPin} that
\[
\pi^\vee=\omega_{\pi}^{-1}\otimes\pi\qand \tau^\vee=\sign_{\tau}^{k}\omega_{\tau}^{-1}\otimes\tau,
\]
where $\sign_\tau$ is the sign character of $\tau$ as defined in \eqref{eq:sign_of_pi}. Also by Lemma \ref{lemma:twist_by_sign_even_GPin}, we have
\[
\omega_{\pi}^{-1}\otimes\pi=\sign_{\pi}^{k}\omega_{\pi}^{-1}\otimes \pi.
\]
But since $\omega_{\pi}|_{Z^\circ}=\omega_{\tau}^{-1}|_{Z^{\circ}}$ and so $\omega_\pi(-1)=\omega_{\tau}(-1)$, we have
\[
\pi^\vee=\sign_{\tau}^{k}\omega_{\pi}^{-1}\otimes \pi.
\]
We then have
\begin{align*}
\Hom_{H}(\pi,\, \tau^\vee)
&=\Hom_H(\pi,\, \sign_\tau^{k}\omega_{\tau}^{-1}\otimes\tau)\\
&=\Hom_H(\sign_\tau^{k}\omega_{\tau}\otimes\pi,\, \tau)\\
&=\Hom_H(\sign_\pi^{k}\omega_{\pi}^{-1}\otimes\pi,\, \tau)\\
&=\Hom_H(\pi^\vee,\, \tau),
\end{align*}
where for the third inequality we used $\omega_{\pi}|_{Z^\circ}=\omega_{\tau}^{-1}|_{Z^{\circ}}$. Similarly, we can obtain \eqref{eq:hom_space_identity} when $\dim_FV$ is odd.

%Next consider the $\GSpin$ case. Again we may assume $(\omega_{\pi}\omega_{\tau})|_{Z^\circ}=\one$ because otherwise the Hom spaces are zero and \eqref{eq:hom_space_identity} trivially holds. Assume $\dim_FV=2k$ and $\dim_FW=2k-1$. We know from Theorem \ref{thm:contragredient_GPin} that
%\[
%\pi^\vee=\omega_{\pi}^{-1}\otimes\pi\quad\text{or}\quad \omega_{\pi}^{-1}\otimes\pi^\delta,
%\]
%where $\delta$ is any element in $\GPin(V)\smallsetminus\GSpin(V)$. Also
%\[
%\tau^\vee=\omega_{\tau}^{-1}\otimes\tau.
%\]
%If $\pi^\vee=\omega_{\pi}^{-1}\otimes\pi$, the proof is essentially the same as the $\GPin$ case and left to the reader. (This case is actually simpler because there involves no $\sign_{\pi}$.) Assume $\pi^\vee=\omega_{\pi}^{-1}\otimes\pi^\delta$. Let us choose $\delta$ from the nonidentity component of the center of $\GPin(W)$, so that $\tau^\delta=\tau$. (For example, we may choose $\delta=\zeta$, where $\zeta$ is as in \eqref{eq:shimura_zeta} but with respect to $W$.) We then have
%\begin{align*}
%\Hom_{H}(\pi,\, \tau^\vee)
%&=\Hom_H(\pi,\, \omega_{\tau}^{-1}\otimes\tau)\\
%&=\Hom_H(\omega_{\tau}\otimes\pi,\, \tau)\\
%&=\Hom_H(\omega_{\pi}^{-1}\otimes\pi,\, \tau)\\
%&=\Hom_H(\omega_{\pi}^{-1}\otimes\pi^\delta,\, \tau^\delta)\\
%&=\Hom_H(\pi^\vee,\, \tau).
%\end{align*}
%The case $\dim_FV$ is odd is similar.
\end{proof}

%$$gg^{*}  =N(g)\in F^{\times}$$
% and so $N(g^{-1})g=g^{*-1}=g^{-*}$ which implies that
%\begin{center}
%\begin{align*}
%\pi(g^{-*}) & =\pi(N(g)^{-1}g)\\
% & =\omega(N(g)^{-1})\pi(g)\\
% & =(\omega^{-1}\circ N \otimes\pi)(g)
%\end{align*}
%\par\end{center}
%
%Similarly, we have that
%\begin{center}
%\begin{align*}
%\pi(g^{-*})^{\delta} & =\pi(\delta g^{-*}\delta^{-1} )\\
%&=\pi(\delta N(g)^{-1} g \delta^{-1})\\
% & =\omega(N(g)^{-1})\pi(\delta g \delta^{-1})\\
% & =(\omega^{-1}\circ N \otimes\pi^{\delta})(g)
%\end{align*}
%\par\end{center}

\section{Reduction to classical group situations}

In this section, we reduce Theorem \ref{thm:main_theorem_distribution} to the classical group situations of \cite{AGRS} and \cite{Wal12}. To be more precise, first we reduce Theorem \ref{thm:main_theorem_distribution} to the vanishing assertion
\begin{equation}\label{eq:chi_distribution_zero}
\Scal'(\GPin(V)\times V)^{\GPint(V),\chi}=0,
\end{equation}
which is the analogue of Proposition \ref{pro:AGRS_non-exitence_of_distribution_times_V}. Then we will show that this vanishing assertion can be reduced to the classical group situations.

\subsection{Elimination of $W$}
In this subsection, we reduce Theorem \ref{thm:main_theorem_distribution} to the above vanishing assertion \eqref{eq:chi_distribution_zero}; namely we eliminate the space $W$ from Theorem \ref{thm:main_theorem_distribution}.  The key ingredient is the following version of Frobenius descent.

\begin{lemma}[Frobenius descent]\label{lemma:Frobenius_descent}
Let $G$ be an lctd group which is unimodular. Let $X$ and $Y$ be lctd spaces on which $G$ acts. Further assume that the action of $G$ on $Y$ is transitive. Suppose we have a continuous $G$-equivariant map
\[
\phi:X \rightarrow Y,
\]
namely $\varphi(g\cdot x)=g\cdot\varphi(x)$ for all $g\in G$ and $x\in X$. Fix $y\in Y$. Assume the stabilizer $G_y\subseteq G$ of $y$ is unimodular. Let $\chi:G\to\CC^1$ be a character of $G$.  Then there is a canonical isomorphism
\[
\mathcal{S}'(X)^{G,\chi}\simeq \mathcal{S}'(\phi^{-1}(y))^{G_y,\chi}.
\]
\end{lemma}
\begin{proof} See \cite[Theorem 2.2]{AGRS}, \cite[Section 1.5]{Ber} or \cite[Sections 2.21-2.36]{BZ76}.
\end{proof}

We can then prove the following.
\begin{prop}\label{prop:removing_W_from_vanishing_assertion}
We have a natural inclusion
\[
\Scal'(\GPin(V))^{\GPint(W), \chi}\subseteq \Scal'(\GPin(V)\times V)^{\GPint(V), \chi}.
\]
Hence, if $\Scal'(\GPin(V)\times V)^{\GPint(V), \chi}=0$ then $\Scal'(\GPin(V))^{\GPint(W), \chi}=0$.
\end{prop}
\begin{proof}
Recall that we have the orthogonal decomposition $V=W\oplus Fe$, where $e$ is anisotropic. Let
\[
X:=\{(g,v) \in \GPin(V)\times V \st\langle v,v\rangle=\langle e,e \rangle \}.
\]
Because of the way we have defined the action of $\GPint(V)$ on $V$ \eqref{eq:action_on_V_GPin}, one can readily see that $X$ is invariant under $\GPint(V)$. Hence we have
\[
\Scal'(X)^{\GPint(V), \chi}\subseteq\Scal'(\GPin(V)\times V)^{\GPint(V), \chi}
\]
because $X$ is closed in $\GPin(V)\times V$.

Next, let
\[
Y:=\{v\in V\st \langle v,v\rangle=\langle e,e \rangle\},
\]
which can be readily seen to be invariant under the action of $\GPint(V)$. By Witt's theorem, we know $\OO(V)$ acts transitively on $Y$ and hence $\GPint(V)$ acts transitively on $Y$.

Now, consider the projection
\[
\phi:X\longrightarrow Y,\quad (g, v)\mapsto v,
\]
which is $\GPint(V)$-equivariant. Recall that the stabilizer $\GPint(V)_e=\GPint(W)$ of $e$ is unimodular. Hence by the Frobenius descent (Lemma \ref{lemma:Frobenius_descent}) applied to this $\phi$, we obtain the canonical isomorphism
\begin{equation}\label{eq:canonical_iso_Forb1}
\Scal'(X)^{\GPint(V), \chi} \simeq \Scal'(\phi^{-1}(e))^{\GPint(V)_e,\chi}.
\end{equation}
By the obvious identification $\GPin(V)\times\{e\}\simeq\GPin(V)$ of sets, we have
\[
\Scal'(\phi^{-1}(e))^{\GPint(V)_e,\chi}\simeq\Scal'(\GPin(V))^{\GPint(W),\chi}.
\]
Hence we have
\[
\Scal'(\GPin(V))^{\GPint(W),\chi}\simeq \Scal'(X)^{\GPint(V), \chi}\subseteq \Scal'(\GPin(V)\times V)^{\GPint(V), \chi}.
\]
The lemma is proven.
\end{proof}

\subsection{Reduction to semisimple orbit}
By the above proposition, the proof of our main theorem is now reduced to proving the vanishing assertion \eqref{eq:chi_distribution_zero}. We will show this by reducing to the classical group situations of \cite{AGRS} and \cite{Wal12}.

%The proof is by induction on $\dim_FV$. The base step is simple because the group is commutative as follows.
%\begin{lemma}
%Assume $\dim V=1$. Then \eqref{eq:chi_distribution_zero} holds.
%\end{lemma}
%\begin{proof}
%If $\dim V=1$ then $\GPin(V)=F^\times\cup F^\times\zeta$, where $\zeta\in V$ is any nonzero vector. Hence $\GPin(V)$ is commutative, so that the conjugation action of $\GPin(V)$ on $\GPin(V)$ is trivial. Also the Clifford involution $\sigma$ acts trivially on $\GPin(V)$. Hence $\mathcal{S}'({\GPin(V)}\times V)^{\widetilde{\GPin(V)},\chi}=0$.
%\end{proof}

The basic idea is to show that any distribution in $\mathcal{S}'({\GPin(V)}\times V)^{\widetilde{\GPin(V)},\chi}$ is supported in a smaller set by using Harish-Chandra's descent, for which we also need the following lemma due to Bernstein, which is called Bernstein's localization principle in \cite{AGRS}.
\begin{lemma}[Bernstein's localization principle]
Let $G$ be an lctd group. Let $\phi:X\to Y$ be a continuous map between lctd spaces $X$ and $Y$. Assume $G$ acts on $X$ by preserving the fibers of $\phi$. Let $\chi:G\to\CC^1$ be a character of $G$.  Suppose for all $y\in Y$ we have
\[
\mathcal{S}'(\phi^{-1}(y))^{G,\chi}=0.
\]
Then
\[
\Scal'(X)^{G, \chi}=0.
\]
\end{lemma}
\begin{proof}
See \cite[Corollary 2.1]{AGRS}, which is taken from \cite[Section 1.4]{Ber}.
\end{proof}

Bernstein's localization principle essentially says that if there is any continuous $\phi:X\to Y$ such that each fiber is preserved under the action of $G$ then the vanishing of the space of distributions $\Scal'(X)^{G,\chi}$ can be shown ``fiber-by-fiber".

By using this twice, we can prove the following.
\begin{prop}\label{prop:reduction_to_semisimple_orbit}
Let $\OO(V)_{s}$ be the set of semisimple elements in $\OO(V)$. Define a map
\[
\theta:\GPin(V)\times V\longrightarrow \OO(V)_{s}
\]
by
\[
(g, v)\mapsto P(g)_s,
\]
where $P(g)_s$ is the semisimple part of $P(g)$ under the Jordan decomposition.

If
\[
\Scal'(\theta^{-1}(\gamma))^{\GPint(V), \chi}=0
\]
for all semisimple conjugacy class $\gamma\subseteq\OO(V)_{s}$, then
\[
\Scal'(\GPin(V)\times V)^{\GPint(V),\chi}=0.
\]
\end{prop}
\begin{proof}
Let $Y$ be the space of polynomials of degree at most $n=\dim V$, which is a lctd space. Consider the map
\[
\phi:\GPin(V)\times V\longrightarrow\GPin(V)\xrightarrow{\quad P\quad}\OO(V)\longrightarrow Y,
\]
where the first map is the projection on $\GPin(V)$, the second one is the canonical projection, and the third one sends each element in $\OO(V)$ to its characteristic polynomial, so that $\phi$ sends each $(g,v)\in\GPin(V)\times V$ to the characteristic polynomial of $P(g)$. One can see that $\phi$ is continuous.

Let $f\in Y$ be a polynomial. We have
\[
\phi^{-1}(f)=\Fcal_f\times V,
\]
where
\[
\Fcal_f=\{g\in\GPin(V)\st \text{the char.\ poly.\ of $P(g)$ is $f$}\}.
\]
Since $P(\sigma_V(g))=P(g)^{-1}$ and $P(g)^{-1}$ is conjugate to $P(g)$ in $\OO(V)$ by Lemma \ref{lemma:MVW_orthogonal_group}, $P(\sigma_V(g))$ and $P(g)$ have the same characteristic polynomial. Thus the fiber $\phi^{-1}(f)$ is preserved by $\GPint(V)$. Hence by Bernstein's localization principle, if
\[
\Scal'(\Fcal_f\times V)^{\Gt,\chi}=0
\]
for all $f\in Y$, then $\Scal'(\GPin(V)\times V)^{\GPint(V),\chi}=0$.

Next let $P(\Fcal_f)_{s}$ be the subset of $P(\Fcal_f)$ consisting of semisimple elements, and denote by $P(\Fcal_f)_{s}\slash\!\sim$ the set of conjugacy classes in $P(\Fcal_f)_{s}$. It is well-known that each semisimple conjugacy class is closed in the group, and hence closed in $P(\Fcal_f)_s$, when $P(\Fcal_f)_s$ is given the subspace topology. Moreover, it follows from \cite[Chap.\ IV]{SS70} that $P(\Fcal_f)$ contains only a finite number of semisimple conjugacy classes. Hence in $P(\Fcal_f)_{s}$ each semisimple conjugacy class is open, which implies the quotient space $P(\Fcal_f)_{s}\slash\!\sim$ is discrete and in particular lctd.

Now, consider the map
\[
\theta:\Fcal_f\times V\longrightarrow P(\Fcal_f)\longrightarrow P(\Fcal_f)_{s}\longrightarrow P(\Fcal_f)_{s}\slash\!\sim,
\]
where the second map takes each $h\in P(\Fcal_f)\subseteq\OO(V)$ to its semisimple part $h_s$ for the Jordan decomposition $h=h_sh_u$, and the third map is the canonical surjection. This map $\theta$ is indeed continuous as in \cite[proof of Lemma 5.1, p.1426]{AGRS}. Now for each conjugacy class $\gamma\in P(\Fcal_f)_{s}\slash\!\sim$, the fiber $\theta^{-1}(\gamma)$ is invariant under $\GPint(V)$ because the involution $\sigma_V$ preserves the semisimple conjugacy classes. Hence by applying Bernstein's localization principle, if
\[
\Scal'(\theta^{-1}(\gamma))^{\GPint(V),\chi}=0
\]
for all semisimple conjugacy class $\gamma$ of $\OO(V)$, then
\[
\Scal'(\GPin(V)\times V)^{\GPint(V),\chi}=0.
\]
Of course this $\theta^{-1}(\gamma)$ is the same as the one in the proposition.
\end{proof}

\subsection{Reduction to $\OO(V)$ situation}
The next step is to reduce the vanishing assertion
\[
\Scal'(\theta^{-1}(\gamma))^{\GPint(V), \chi}=0
\]
to even a smaller support than $\theta^{-1}(\gamma)$ again by applying the Frobenius descent and Bernstein's localization principle, and then reduce to the classical group situations.

Let us first set up a new notation. We let
\[
\Ucal\subseteq\GPin(V)
\]
be the set of unipotent elements in $\GPin(V)$. For each $g\in\GPin(V)$ we let
\[
\Ucal_g=\{u\in\Ucal\st gu=ug\},
\]
namely the centralizer of $g$ in $\Ucal$. Note that both $\Ucal$ and $\Ucal_g$ are closed in $\GPin(V)$. Also the restriction to $\Ucal$ of the canonical projection $P:\GPin(V)\to\OO(V)$ is one-to-one, which allows us to identify the set of unipotent elements in $\OO(V)$ with $\Ucal$.

Note that for each $g\in\GPin(V)$ the stabilizer $\GPint(V)_g$ of $g$ in $\GPint(V)$ is the group generated by $\GPin(V)_g$ and $\eta\beta$, where $\eta\in\GPin(V)$ is such that $\eta\sigma_V(g)\eta^{-1}=g$. In particular, if $g$ is semisimple, so that $\GPin(V)_g$ is reductive, then $\GPint(V)_g$ is unimodular, having the reductive $\GPin(V)_g$ as an index 2 subgroup.

\begin{lemma}
Assume
\[
\Scal'(Z^\circ g\,\Ucal_g\times V)^{\GPint(V)_{g}, \chi}=0
\]
for all semisimple $g\in\GPin(V)$. Then
\[
\Scal'(\GPin(V)\times V)^{\GPint(V),\chi}=0.
\]
\end{lemma}
\begin{proof}
Let $\gamma\subseteq\OO(V)_s$ be the semisimple conjugacy class of $P(g)$. The lemma is proven by applying the Frobenius descent to the map
\[
\theta:\theta^{-1}(\gamma)\longrightarrow \gamma, \quad (g, v)\mapsto P(g)_s,
\]
where $(g, v)\in\theta^{-1}(\gamma)\subseteq\GPin(V)\times V$. To be precise, for each semisimple $g\in\GPin(V)$ we have
\[
\theta^{-1}(P(g))=Z^\circ g\,\Ucal_g\times V,
\]
and hence by the Frobenius descent
\[
\Scal'(\theta^{-1}(\gamma))^{\GPint(V), \chi}\simeq\Scal'(Z^\circ g\,\Ucal_g\times V)^{\GPint(V)_{g}, \chi}.
\]
\end{proof}

Let us eliminate $Z^\circ$ from the above lemma.
\begin{lemma}
If $\Scal'(zg\,\Ucal_g\times V)^{\GPint(V)_{g}, \chi}=0$ for all $z\in Z^\circ$, then $\Scal'(Z^\circ g\,\Ucal_g\times V)^{\GPint(V)_{g}, \chi}=0$.
\end{lemma}
\begin{proof}
This can be proven by applying Bernstein's localization principle to the map
\[
Z^\circ g\,\Ucal_g\times V\longrightarrow Z^\circ,\quad zg\, \Ucal_g\mapsto z,
\]
because each fiber $zg\,\Ucal_g\times V$ is preserved by $\GPint(V)_g$.
\end{proof}

Hence our vanishing assertion reduces to
\[
\Scal'(g\,\Ucal_g\times V)^{\GPint(V)_{g}, \chi}=0
\]
for all semisimple $g\in\GPin(V)$ because in the above lemma $zg$ is semisimple for all $z\in Z^\circ$ and $\Ucal_{zg}=\Ucal_g$.

\quad

Note that we have the obvious projection
\[
P:\GPint(V)\longrightarrow\OOt(V),\quad g\mapsto P(g),\; \beta\mapsto \beta,
\]
whose kernel is $Z^\circ$, which acts trivially on the space $\Scal'(g\,\Ucal_g\times V)$. Here we use the same symbol $P$ as the canonical projection because it is actually an extension of the canonical projection to $\GPint(V)$, and this should never cause any confusion. Next, for each semisimple $g$ we have the bijection
\[
g\,\Ucal_g\longrightarrow P(g\,\Ucal_g)
\]
induced by the canonical projection $P$. This is indeed a bijection because $\GPin(V)$ and $\OO(V)$ share the same unipotent elements. Further, this map intertwines the actions of $\GPint(V)_g$ and $P(\GPint(V)_g)$, which implies
\[
\Scal'(g\,\Ucal_g\times V)^{\GPint(V)_{g}, \chi}\simeq\Scal'(P(g\,\Ucal_g)\times V)^{P(\GPint(V)_g), \chi},
\]
because the kernel of the projection $P$ acts trivially. Note that
\[
\Scal'(P(g\,\Ucal_g)\times V)^{P(\GPint(V)_g), \chi}\subseteq\Scal'(P(\GPin(V)_g)\times V)^{P(\GPint(V)_g), \chi}.
\]
Hence to show our main theorem, it suffices to show
\begin{equation}\label{eq:vanishing_for_orthogonal}
\Scal'(P(\GPin(V)_g)\times V)^{P(\GPint(V)_g), \chi}=0
\end{equation}
for all semisimple $g\in\GPin(V)$. This is essentially the $\OO(V)$ situation of \cite{AGRS}

\section{End of proof}
This vanishing assertion \eqref{eq:vanishing_for_orthogonal} is more or less proven in \cite{AGRS}. Unfortunately, however, we do not always have $P(\GPin(V)_g)=\OO(V)_{P(g)}$ as we have seen in Lemma \ref{lemma:centralizer_GPin}. (Indeed, if $P(\GPin(V)_g)=\OO(V)_{P(g)}$ then we would have only to show $\Scal'(\OO(V)_{P(g)}\times V)^{\OO(V)_{P(g)}, \chi}=0$, which is precisely the assertion proven in \cite{AGRS}.) Accordingly we need to modify \cite{AGRS}. The difference is that $P(\GPin(V)_g)$ might have a factor of $\SO$ as in Lemma \ref{lemma:centralizer_GPin}, for which we need the result of Waldspurger \cite{Wal12} for the $\SO$ case.

In this subsection, we set
\[
h:=P(g)\in\OO(V)\qand \OO'(V)_h:=P(\GPin(V)_g)
\]
to ease the notation. As in Proposition \ref{prop:centralizer_of_h} we have an orthogonal sum decomposition
\[
V=V_1\oplus\cdots\oplus V_m\oplus V_+\oplus V_-
\]
such that
\[
\OO_h\simeq G_1\times\cdots\times G_m\times\OO(V_+)\times\OO(V_-),
\]
where
\[
G_i=\begin{cases}\GL_{A_i}(X_i),&\text{if $V_i=X_i\oplus X_i^*$};\\ U_{A_i}(V_i),&\text{otherwise}.\end{cases}
\]
We let
\[
\OO'(V_+)=\begin{cases}\OO(V_+),&\text{if $\dim_FV_-$ is even}\\\SO(V_+),&\text{if $\dim_FV_-$ is odd},\end{cases}
\]
and
\[
\OO'(V_-)=\begin{cases}\SO(V_-),&\text{if $\dim_FV_-$ is even}\\\OO(V_-),&\text{if $\dim_FV_-$ is odd}.\end{cases}
\]
We then have
\[
\OO'(V)_h=P(\GPin(V)_g)\simeq G_1\times\cdots\times G_m\times\OO'(V_+)\times\OO'(V_-)
\]
by Lemma \ref{lemma:centralizer_GPin}.

Let us denote the image of $h$ under this isomorphism by
\[
(h_1,\dots,h_m, h_+, h_-).
\]
For $i=1,\dots, m$, let $\tau_i$ be our involution for $G_i$ as defined in \eqref{eq:tau_for_GL} and \eqref{eq:tau_for_classical}. Then
\[
\tau_i(h_i)=h_i
\]
because each of $h_i$'s is in the center of $G_i$ and the corresponding involution fixes the center pointwise. On $\OO'(V_{+})$, we define an involution $\tau_{+}$ as follows: If $\OO'(V_+)=\OO(V_+)$ then $\tau_+$ is as in \eqref{eq:tau_for_classical}. If $\OO'(V_+)=\SO(V_+)$ then $\tau_+$ is as in \eqref{eq:tau_for_special_orthogonal}. In either case, we have
\[
\tau_+(h_+)=h_+.
\]
We similarly define an involution $\tau_-$ on $\OO'(V_-)$.

For each $i=1,\dots,m$, we let $\beta_i$ be the corresponding element as in \eqref{eq:beta_for_GL} or \eqref{eq:beta_for_classical_group}, and $\Gt_i$ the group generated by $G_i$ and $\beta_i$ as before. For $\OO'(V_{\pm})$, we define $\OOpt(V_{\pm})$ accordingly as it is special orthogonal or not.

Recall that there exists
\begin{equation}\label{eq:conjugating_element_gamma}
\gamma=(\gamma_1,\dots,\gamma_m,\gamma_+,\gamma_-)\in\OO(V_1)\times\cdots\times\OO(V_m)\times\OO(V_+)\times\OO(V_-)
\end{equation}
such that $\gamma h^{-1}\gamma^{-1}=h$. (Note that $\gamma$ depends on our fixed $h$.) Hence
\[
\OOt(V)_h=\big\la a, \gamma\beta\st a\in\OO(V)_h\big\ra,
\]
which implies
\[
\OOpt(V)_h=P(\GPint(V)_g)=\big\la a, \gamma\beta\st a\in P(\GPin(V)_g)\big\ra.
\]

Here, $\gamma$ is not unique. For $i=1,\dots, m$, we choose $\gamma_i=\beta_i$ for $i=1,\dots,m$. As for the orthogonal factor, we know that $\gamma_+$ and $\gamma_-$ can be arbitrary. We choose $\gamma_{\pm}$ as follows; Fix an anisotropic vector $e_{\pm}\in V_{\pm}$ for each $\pm$ and set
\[
\gamma_{\pm}=\begin{cases}1 &\text{if $\OO'(V_{\pm})=\OO(V_{\pm})$};\\ r_{e_{\pm}}^{k_\pm}&\text{if $\OO'(V_{\pm})=\SO(V_{\pm})$},\end{cases}
\]
where we recall $k_{\pm}$ is such that $\dim_FV_{\pm}=2k_{\pm}$ or $2k_{\pm}-1$, and $r_{e_{\pm}}$ is the reflection in the hyperplane orthogonal to $e_{\pm}$.

We then have the natural injection
\begin{equation}\label{eq:OOt_g_inclusion}
P(\GPint(V)_g)\longrightarrow \Gt_1\times\cdots\times\Gt_m\times\OOpt(V_+)\times\OOpt(V_-)
\end{equation}
by sending
\[
\gamma\beta\mapsto (\gamma_1,\cdots,\gamma_m,\gamma_+\beta_+, \gamma_-\beta_-).
\]
Note that we have the natural commutative diagram
\[
\begin{tikzcd}
\Gt_1\times\cdots\times\Gt_m\times\OOpt(V_+)\times\OOpt(V_-)\ar[r]&\overbrace{\{\pm 1\}\times\cdots\times \{\pm 1\}}^{\text{$m+2$-times}}\\
P(\GPint(V)_g)\ar[r]\ar[u]&\{\pm 1\}\ar[u, "\Delta"']
\rlap{\hspace{1em} ,}
\end{tikzcd}
\]
where the top arrow maps the component $\Gt_i\smallsetminus G_i$ to $-1\in\{\pm 1\}$ in the $i$-factor (and similarly for $\OOpt(V_{\pm})$), the bottom arrow $\OOpt(V)_h\smallsetminus\OO'(V)_h$ to $-1$, and the right arrow is the diagonal embedding. We then have the natural isomorphism
\begin{align*}
&\Scal'(\OO'(V)_h\times V)^{\OOpt(V)_h, \chi}\\
&\simeq\Scal'\big((G_1\times V_1)\times\cdots\times (G_m\times V_m)\times(\OO'(V_+)\times V_+)\times(\OO'(V_-)\times V_-)\big)^{\OOpt(V)_h, \chi},
\end{align*}
where on the right-hand side the group $\OOpt(V)_h$ acts on the set $(G_1\times V_1)\times\cdots\times (G_m\times V_m)\times(\OO'(V_+)\times V_+)\times(\OO'(V_-)\times V_-)$
via the left vertical arrow in the above diagram. Recall that we are trying to show this space is zero.

In what follows, the orthogonal factor $\OO'(V_+)\times\OO'(V_-)$ does not play a special role anymore. Hence for notational convenience let us write
\[
G_{m+1}=\OO'(V_+)\qand G_{m+2}=\OO'(V_-),
\]
and similarly $V_{m+1}=V_+$, $V_{m+2}=V_-$, etc. Further, we reset our $m$ to be $m+2$, so that we simply have
\begin{equation}\label{eq:centralizer_of_h}
\OO'(V)_h\simeq G_1\times\cdots\times G_m.
\end{equation}

For each $i=1,\dots, m$, let
\[
\chi_i:\Gt_i\longrightarrow\{\pm 1\}
\]
be the sign character that sends $\gamma_i$ to $-1$. Also let $\epsilon_i\in\{0, 1\}$, so that $\chi_i^{\epsilon_i}$ is either trivial or $\chi_i$. For each $m$-tuple $(\epsilon_1,\dots,\epsilon_m)$, we set
\[
\Scal'\big((G_1\times V_1)\times\cdots\times (G_m\times V_m)\big)^{(\epsilon_1,\dots,\epsilon_m)}
\]
to be the space of distributions on which $\Gt_1\times\cdots\times \Gt_m$ acts by the character
\[
\chi_1^{\epsilon_1}\otimes\cdots\otimes\chi_m^{\epsilon_m}.
\]

Now assume
\[
T\in\Scal'\big((G_1\times V_1)\times\cdots\times (G_m\times V_m)\big)^{\OOpt(V)_h, \chi}
\]
is nonzero. Let
\[
\Scal'_T=(\Gt_1\times\cdots\times\Gt_m)T,
\]
namely the $(\Gt_1\times\cdots\times\Gt_m)$-module generated by $T$. We then have
\[
\Scal'_T\subseteq\bigoplus_{(\epsilon_1,\dots,\epsilon_m)} \Scal'\big((G_1\times V_1)\times\cdots\times (G_m\times V_m)\big)^{(\epsilon_1,\dots,\epsilon_m)}.
\]
Since $\OO'(V)_h$, viewed as a subgroup of $\Gt_1\times\cdots\times\Gt_m$, acts by the nontrivial character $\chi$ on $T$, there exists at least one $\epsilon_i\neq 0$ such that the projection
\[
\Scal'_T\longrightarrow  \Scal'\big((G_1\times V_1)\times\cdots\times (G_m\times V_m)\big)^{(\epsilon_1,\dots,\epsilon_m)}
\]
is nonzero.

We will show
\[
\Scal'\big((G_1\times V_1)\times\cdots\times (G_m\times V_m)\big)^{(\epsilon_1,\dots,\epsilon_m)}=0,
\]
which will be a contradiction and hence there is no nonzero $T$. First, by permuting the indices, we may assume $\epsilon_1\neq 0$. Let
\[
A\in \Scal'\big((G_1\times V_1)\times\cdots\times (G_m\times V_m)\big)^{(\epsilon_1,\dots,\epsilon_m)}
\]
be nonzero. Then there exists a simple tensor
\[
\varphi_1\otimes\cdots\otimes\varphi_m\in\Scal(G_1\times V_1)\otimes\cdots\otimes\Scal(G_m\times V_m)
\]
such that $A(\varphi_1\otimes\cdots\otimes\varphi_m)\neq 0$, which implies the composite
\begin{align*}
&\Scal(G_1\times V_1)\longrightarrow \Scal(G_1\times V_1)\otimes\cdots\otimes (G_m\times V_m)\xrightarrow{\;A\;}\C\\
&\qquad\varphi\qquad\mapsto\qquad \varphi\otimes\varphi_2\otimes\cdots\otimes\varphi_m
\end{align*}
is nonzero. But this composite is a distribution on $G_1\times V_1$ on which $\Gt_1$ acts via the character $\chi_1$, which implies
\[
\Scal'(G_1\times V_1)^{\Gt_1, \chi_1}\neq 0.
\]
This contradicts Proposition \ref{pro:AGRS_non-exitence_of_distribution_times_V}, the result obtained by \cite{AGRS} or by \cite{Wal12}. Hence $T=0$.

Thus we have shown $\Scal'(\OO'(V)_h\times V)^{\OOpt(V)_h, \chi}=0$, which completes the proof of our main theorem (Theorem \ref{thm:A}) for $\GPin$.

\quad

\section{GSpin case}

In this section, we prove our main theorem for $\GSpin(V)$. The proof is essentially the same as the GPin case but we need to make appropriate modifications. The most notable difference is that instead of the group $\GPint(V)$, we use an appropriate subgroup $\GSpint(V)$, which is the analogue of $\SOt(V)$. Let us first recall our basic set up. As before, $V$ is a quadratic space with
\[
\dim_FV=n=\begin{cases}2k\\ 2k-1.\end{cases}
\]
We fix an orthogonal basis $e_1,\dots,e_{n-1}, e_n$, and assume $W=\Span\{e_1,\dots,e_{n-1}\}$. We often write $e=e_n$. We set
\[
\zeta=e_1\cdots e_n.
\]

\subsection{The group $\GSpint(V)$}
We define the group $\GSpint(V)$, which plays the role of $\GPint(V)$. First recall that
\[
\GPint(V)=\big\la g,\, \beta\st g\in\GPin(V)\big\ra,
\]
namely the group generated by $g$'s and $\beta$ modulo the relations $g\beta=\beta g$ and $\beta^2=1$.

We then define
\[
\GSpint(V)=\big\la g,\, e^k\beta\st g\in\GSpin(V)\big\ra\subseteq\GPint(V),
\]
so that we have
\[
1\longrightarrow \GSpin(V)\longrightarrow\GSpint(V)\xrightarrow{\;\,\chi\;\,}\{\pm1\}\longrightarrow 1,
\]
where the surjection $\chi$ sends $e^k\beta$ to $-1$, and
\[
\GSpint(V)\simeq\GSpin(V)\times\{1, e^k\beta\},
\]
where $e^k\beta$ acts on $\GSpin(V)$ by conjugation viewed inside $\GPint(V)$. Since $\GSpint(V)$ is a subgroup of $\GPint(V)$, it acts on $\GSpin(V)\times V$ (viewed merely as a set) by restricting the action of $\GPint(V)$ as
\begin{gather}\label{eq:action_of_GSpint}
\begin{aligned}
g\cdot (h, v)&=(ghg^{-1}, P(g)v)\\
e^k\beta\cdot (h, v)&=(e^k\sigma_V(h)e^{-k}, -P(e)^kv),
\end{aligned}
\end{gather}
where $(h, v)\in\GSpin(V)\times V$.

We let $\GSpint(V)_e$ be the stabilizer of $e\in V$ under the action of $\GSpint(V)$ on $V$ as usual. Analogously to the $\SO(V)$ case, One can then show
\[
\GSpint(V)_e=\big\la g, e_{n-1}^{k-1}e\beta\st g\in\GSpin(W)\big\ra,
\]
and we define
\[
\GSpint(W):=\GSpint(V)_e.
\]
We have
\[
1\longrightarrow \GSpin(W)\longrightarrow\GSpint(W)\xrightarrow{\;\,\chi\;\,}\{\pm1\}\longrightarrow 1,
\]
where the surjection $\chi$ sends $e_{n-1}^{k-1}e\beta$ to $-1$, and
\[
\GSpint(W)\simeq\GSpin(W)\rtimes\{1, e_{n-1}^{k-1}e\beta\},
\]
where the action of $e_{n-1}^{k-1}e\beta$ is by conjugation viewed inside $\GPint(V)$.

We define an involution
\[
\tau_W:\GSpin(V)\longrightarrow\GSpin(V),\quad \tau_W(g)=(e_{n-1}^{k-1}e)\sigma_V(g)(e_{n-1}^{k-1}e)^{-1},
\]
for $g\in\GSpin(V)$. This is the action of $e_{n-1}^{k-1}e\beta\in\GSpint(V)_e$ on $\GSpin(V)$. Since $e$ commutes with all the elements in $\GSpin(W)$, we have $\tau_W(\GSpin(W))=\GSpin(W)$.

We have the canonical projection
\[
P:\GSpint(V)\longrightarrow \SOt(V),\quad g\mapsto P(g),\;e^k\beta\mapsto r_e^k\beta,
\]
which is nothing but the restriction of the canonical projection $P:\GPint(V)\to\OOt(V)$. We then have
\[
P(\GSpint(V)_e)=\SOt(V)_e.
\]

Let $g\in\GSpin(V)$ be semisimple, and set $h:=P(g)\in\SO(V)$. If
\[
\OO(V)_h\simeq G_1\times\cdots\times G_m\times\OO(V_+)\times\OO(V_-)
\]
as before, then
\[
\SO(V)_h\simeq G_1\times\cdots\times G_m\times S(\OO(V_+)\times\OO(V_-)),
\]
where
\[
S(\OO(V_+)\times\OO(V_-))=(\OO(V_+)\times\OO(V_-))\cap\SO(V_+\oplus V_-)
\]
by Proposition \ref{prop:centralizer_h_SO}. Note that we necessarily have $\dim_FV_-$ even, because $h\in\SO(V)$, and $h_+=1$ and $h_-=-1$. Apparently,
\[
P(\GSpin(V)_g)\subseteq\SO(V)_h.
\]
But this inclusion can be strict if there is an orthogonal factor $\OO(V_+)\times\OO(V_-)$. To be precise, we have the following.
\begin{lemma}\label{lemma:centralizer_GSpin}
Keeping the above notation, we have
\[
P(\GSpin(V)_g)\simeq G_1\times\cdots\times G_m\times \SO(V_+)\times\SO(V_-).
\]
\end{lemma}
\begin{proof}
Since $\dim_FV_-$ is even, we know
\[
P(\GPin(V)_g)\simeq G_1\times\cdots\times G_m\times \OO(V_+)\times\SO(V_-)
\]
by Lemma \ref{lemma:centralizer_GPin}. Hence the lemma follows because $\GSpin(V)_g=\GPin(V)_g\cap\GSpin(V)$.
\end{proof}

\subsection{Vanishing of distribution}
Analogously to the GPin case, the main technical result to be proven is the following vanishing assertion of distributions:
\begin{equation}\label{eq:vanishing_distribution_GSpin}
\Scal'( \GSpin(V))^{\GSpint(W), \chi}=0,
\end{equation}
where $\GSpint(W)\simeq \GSpin(W)\times\{1, e_{n-1}^{k-1}e\beta\}$ acts on $\GPin(V)$ by restricting the actions \eqref{eq:action_of_GSpint}. In particular, the element $e_{n-1}^{k-1}e\beta$ acts via the involution $\tau_W$, which preserves $\GSpin(W)$ setwise.

Indeed, this implies the following, which is the analogue of Proposition \ref{prop:vanishing_implies_main_theorem}
\begin{prop}
The above vanishing assertion \eqref{eq:vanishing_distribution_GSpin} implies Theorem \ref{thm:A} (multiplicity-at-most-one theorem) for $ \GSpin$.
\end{prop}
\begin{proof}
Set $G=\GSpin(V)$ and $H=\GSpin(W)$. The proof is essentially the same as Proposition \ref{prop:vanishing_implies_main_theorem}. First the vanishing assertion \eqref{eq:vanishing_distribution_GSpin} implies
\[
\dim_{\C}\Hom_H(\pi,\, \tau^\vee)\cdot\dim_{\C}\Hom_H(\pi^\vee,\, \tau)\leq 1.
\]
Hence it suffices to show
\[
\Hom_H(\pi,\, \tau^\vee)=\Hom_H(\pi^\vee,\, \tau).
\]
If $\omega_{\pi}|_{Z^\circ}\neq\omega_{\tau}^{-1}|_{Z^{\circ}}$ then both sides are zero and the equality trivially holds.

So we may assume $\omega_{\pi}|_{Z^\circ}=\omega_{\tau}^{-1}|_{Z^{\circ}}$. Assume $\dim_FV=2k$ and $\dim_FW=2k-1$. We know from Theorem \ref{thm:contragredient_GPin} that
\[
\pi^\vee=\omega_{\pi}^{-1}\otimes\pi\quad\text{or}\quad \omega_{\pi}^{-1}\otimes\pi^\delta,
\]
where $\delta$ is any element in $\GPin(V)\smallsetminus\GSpin(V)$. Also
\[
\tau^\vee=\omega_{\tau}^{-1}\otimes\tau.
\]
If $\pi^\vee=\omega_{\pi}^{-1}\otimes\pi$, the proof is essentially the same as the $\GPin$ case and left to the reader. (This case is actually simpler because there involves no $\sign_{\pi}$.) Assume $\pi^\vee=\omega_{\pi}^{-1}\otimes\pi^\delta$. Let us choose $\delta$ from the nonidentity component of the center of $\GPin(W)$, so that $\tau^\delta=\tau$. (For example, we may choose $\delta=e_1\cdots e_{n-1}$.) We then have
\begin{align*}
\Hom_{H}(\pi,\, \tau^\vee)
&=\Hom_H(\pi,\, \omega_{\tau}^{-1}\otimes\tau)\\
&=\Hom_H(\omega_{\tau}\otimes\pi,\, \tau)\\
&=\Hom_H(\omega_{\pi}^{-1}\otimes\pi,\, \tau)\\
&=\Hom_H(\omega_{\pi}^{-1}\otimes\pi^\delta,\, \tau^\delta)\\
&=\Hom_H(\pi^\vee,\, \tau).
\end{align*}
The case $\dim_FV=2k-1$ is similar.
\end{proof}

Recall the action of $\GSpint(V)$ on $\GSpin(V)\times V$ is defined in \eqref{eq:action_of_GSpint}.

\begin{prop}
We have a natural inclusion
\[
\Scal'(\GSpin(V))^{\GSpint(W), \chi}\subseteq \Scal'(\GSpin(V)\times V)^{\GSpint(V), \chi}.
\]
Hence if
\[
\Scal'(\GSpin(V)\times V)^{\GSpint(V), \chi}=0
\]
then $\Scal'(\GSpin(V))^{\GSpint(W), \chi}=0$.
\end{prop}
\begin{proof}
This can be proven in the same way as Proposition \ref{prop:removing_W_from_vanishing_assertion}. Namely let
\begin{align*}
X&:=\{(g,v)\in\GSpin(V)\times V\st \la v, v\ra=\la e, e\ra\}\\
Y&:=\{v\in V\st \la v, v\ra=\la e, e\ra\},
\end{align*}
and consider the projection
\[
\phi:X\longrightarrow Y.
\]
By Witt's theorem, $\GSpin(V)$ acts transitively on $Y$ and hence by the Frobenius descent we have
\[
\Scal'(X)^{\GSpint(V), \chi}\simeq \Scal'(\GSpin(V)\times\{e\})^{\GSpint(V)_e, \chi},
\]
where the left-hand side is a subspace of $\Scal'(\GSpin(V)\times V)^{\GSpint(V), \chi}$. But clearly
\[
\Scal'(\GSpin(V)\times\{e\})^{\GSpint(V)_e, \chi}\simeq \Scal'(\GSpint(V))^{\GSpint(W), \chi}.
\]
The proposition follows.
\end{proof}

\subsection{Reducing to classical group situation}
By the above proposition, it suffices to show
\[
\Scal'(\GSpin(V)\times V)^{\GSpint(V), \chi}=0.
\]
Arguing in the same way as the $\GPin$ case, this vanishing assertion reduces to
\[
\Scal'(g\,\Ucal_g\times V)^{\GSpint(V)_g, \chi}=0
\]
for all semisimple $g\in\GSpin(V)$. Since the canonical projection $P$ is bijective on $g\,\Ucal_g$, we have the natural isomorphism
\[
\Scal'(g\,\Ucal_g\times V)^{\GSpint(V)_g, \chi}\simeq\Scal'(P(g\,\Ucal_g)\times V)^{P(\GSpint(V)_g), \chi}.
\]
Since we have
\[
\Scal'(P(g\,\Ucal_g)\times V)^{P(\GSpint(V)_g), \chi}
\subseteq\Scal'(P(\GSpin(V)_g)\times V)^{P(\GSpint(V)_g), \chi},
\]
it suffices to show
\[
\Scal'(P(\GSpin(V)_g)\times V)^{P(\GSpint(V)_g), \chi}=0.
\]
We know $P(\GSpin(V)_g)$ is as in Lemma \ref{lemma:centralizer_GSpin} and $P(\GSpint(V)_g)$ is generated by $P(\GSpin(V)_g)$ and the element $\gamma\beta$, where $\gamma$ is as in \eqref{eq:conjugating_element_gamma}. Note that since the orthogonal factor of $P(\GSpin(V)_g)$ is $\SO(V_+)\times\SO(V_-)$, we always choose $\gamma_{\pm}=r_{e_{\pm}}^{k_{\pm}}$. Then the rest of the proof is the same as the GPin case. The proof is complete.

\quad

\appendix
\section{Centralizer of semisimple element}\label{Appendix_A}
In this appendix, we reproduce the proof of Proposition \ref{prop:centralizer_of_h}, which gives the explicit description of the centralizer $\OO(V)_h$ of a semisimple element $h\in\OO(V)$. Though this is well-known already from the 60's (\cite{SS70}), we reproduce the proof in detail because we have not been able to locate a proof in the literature to the precision we need. The beginning part of our proof is borrowed from \cite[p.79-82]{MVW}.

\quad

Let $p(x)\in F[x]$ be the minimum polynomial of $h$, and let
\[
A:=F[x]\slash (p(x)).
\]
Since $h$ is invertible, $p(x)$ has a nonzero constant term, which means $x$ is invertible in $A$. Hence we have the natural isomorphism
\[
A=F[x]\slash (p(x))\simeq F[x, x^{-1}]\slash (p(x)),
\]
where on the right-hand side by $(p(x))$ we actually mean the ideal $p(x)F[x, x^{-1}]$. On $F[x, x^{-1}]$ we have the involution defined by $x\mapsto x^{-1}$. Since $p(x)$ is a minimum polynomial of an element $h$ in the orthogonal group $\OO(V)$, one can see that $p(x^{-1})=a x^mp(x)$ for some $a\in F$, where $m$ is the degree of $p$. (To see this, consider the eigenvalues of $p(x)$ over the algebraic closure.) Namely the involution preserves the ideal $p(x)F[x, x^{-1}]$, which gives rise to the involution
\[
\sigma:A\longrightarrow A,\quad x\mapsto x^{-1}.
\]
We often use the exponential notation $f^\sigma$ instead of $\sigma(f)$ for $f\in A$.

We view the space $V$ as an $A$-module in the obvious way, namely $f\cdot v=f(h)v$ for $f\in A$ and $v\in V$. Then for each $f\in A$
\[
f^\sigma\cdot v=f(h^{-1})v
\]
and
\[
\la f\cdot v, v'\ra=\la v, f^\sigma\cdot v'\ra
\]
for $v, v\in V$.

Since $h$ is semisimple, we can write $p(x)=p_1(x)\cdots p_k(x)$, where $p_i(x)$'s are distinct irreducible polynomials, so that we have
\[
F[x]\slash (p(x))=F[x]\slash (p_1(x))\times\cdots\times F[x]\slash (p_k(x)),
\]
where each
\[
A_i:=F[x]\slash (p_i(x))
\]
is a field because $p_i(x)$ is irreducible. Let $V_i=\ker p_i(h)$. Then we can write
\[
V=V_1\oplus\cdots\oplus V_k.
\]
Since $hV_i=V_i$, we can view each $V_i$ as an $A_i$-module via $q(x)\cdot v_i=q(h)v_i$ for $q(x)\in A_i$, and hence as an $A$-module via the canonical surjection $A\to A_i$.

Since $(p(x)^\sigma)=(p(x))$ viewed in $F[x, x^{-1}]$, for each $i$ we have $(p_i(x)^{\sigma})=(p_{\sigma(i)}(x))$ for some $\sigma(i)\in\{1,\dots, k\}$. There are two possibilities: either $\sigma(i)=i$ or $\sigma(i)\neq i$. Assume $\sigma(i)=i$. In this case, $\sigma$ restricts to an involution on the field $A_i$. Then $V_i$ is orthogonal to all $V_j$ with $i\neq j$, because for each $a_i\in A_i$ we have $\la a_iv_i, v_j\ra=\la v_i, a_i^{\sigma}v_j\ra=0$ for all $v_i\in V_i$ and $v_j\in V_j$ with $i\neq j$. On the other hand, assume $\sigma(i)\neq i$. One can then similarly see that $V_i\oplus V_{\sigma(i)}$ is orthogonal to all the other $V_j$'s and both $V_i$ and $V_{\sigma(i)}$ are isotropic. Let us set
\[
B_i=\begin{cases} A_i\times A_{\sigma(i)},&\text{if $\sigma(i)\neq i$};\\ A_i,&\text{if $\sigma(i)=i$}.\end{cases}
\]

Let us first consider the case $\sigma(i)=i$, so that $B_i=A_i$ is a field with the involution $\sigma$.
\begin{lemma}\label{lemma:appendix_unitary}
Assume $B_i=A_i$. Then there is an $A_i$-Hermitian form
\[
\lla-,-\rra_i:V_i\times V_i\longrightarrow A_i
\]
with respect to $\sigma$, namely
\[
\lla av, v'\rra_i=a\lla v, v'\rra_i\qand \lla v, v'\rra_i^\sigma=\lla v', v\rra_i
\]
for all $v, v'\in V_i$ and $a\in A_i$, such that
\[
\la -, -\ra=\tr_{A_i/F}(\lla -, -\rra_i),
\]
where $\tr_{A_i/F}:A_i\to F$ is the trace form.
\end{lemma}
\begin{proof}
We suppress the subscript $i$ to ease the notation, so $A=A_i$, etc. It is elementary to show that any $F$-linear functional $\ell:A\to F$ is written as $\ell(a)=\tr_{A/F}(a\alpha)$ for some $\alpha\in A$. Now, for each fixed $v, v'\in V$ consider the $F$-linear functional $A\to F$ by $a\mapsto \la av, v'\ra$. Then there exists some $\lla v, v'\rra\in A$ such that
\[
\la av, v'\ra=\tr_{A/F}(a\lla v, v'\rra)
\]
for all $a\in A$. One can readily see that the assignment $\lla-,-\rra:V\times V\to A$ is a nondegenerate Hermitian form on $V$ over $A$ with respect to the involution $\sigma$.
\end{proof}

In the above lemma, it should be noted that if the involution $\sigma$ on $A_i$ is trivial then the polynomial $p_i(x)$ has to be either $p_i(x)=x-1$ or $p_i(x)=x+1$, in which case $A=F$ and the Hermitian form $\lla-,-\rra_i$ on $V_i$ is simply the restriction of our symmetric bilinear form $\la-,-\ra$. For $p_i(x)=x-1$ we set $V_+=V_i$ and $A_+=A_i$, and for $p_i(x)=x+1$ we set $V_-=V_i$ and $A_-=A_i$. (Of course $V_+$ or $V_-$ can be zero, depending on $h$.)

\quad

Next consider the case $\sigma(i)\neq i$. Let us set $j=\sigma(i)$, so that
\[
B_i=A_i\times A_j.
\]
We then have the field isomorphism
\[
\sigma:A_i=F[x]\slash(p_i(x))\xrightarrow{\;\sim\;} F[x]\slash(p_j(x))=A_j,\quad f(x)\mapsto f(x)^\sigma.
\]
Note that under this isomorphism we have $x\mapsto x^{-1}$. By identifying $A_j$ with $B_i$ under this isomorphism, we can write
\[
B_i=A_i\times A_i.
\]
Since the identification of $A_j$ with $A_i$ is made via $\sigma$, the involution $\sigma$ acts on $B_i=A_i\times A_i$ as switching the two factors.

Let $(h_i, h_j)\in A_i\times A_j$ be the image of $h$ in $B_i$. Since the isomorphism $A_i\to A_j$ maps $x$ to $x^{-1}$, under the identification $B_i=A_i\times A_i$ we have
\[
(h_i, h_j)=(h_i, h_i^{-1}).
\]
We often write
\[
h_i=(h_i, h_i^{-1})
\]
by slight abuse of notation. With this said, we have the following.
\begin{lemma}\label{lemma:appendix_GL}
Assume $B_i=A_i\times A_i$, so that $V_i$ and $V_{\sigma(i)}$ are $A_i$-vector spaces. Recall both $V_i$ and $V_{\sigma(i)}$ are totally isotropic such that the restriction of our symmetric form $\la-,-\ra$ on the sum $V_i\oplus V_{\sigma(i)}$ is nondegenerate. Then there exists a nondegenerate $A_i$-bilinear pairing
\[
\lla-,-\rra_i:V_i\times V_{\sigma(i)}\longrightarrow A_i
\]
such that
\[
\la-,-\ra=\tr_{A_i/F}(\lla-,-\rra_i).
\]
Via this bilinear pairing, we have the identification
\[
V_{\sigma(i)}=V_i^*=\Hom_F(V_i, F).
\]
\end{lemma}
\begin{proof}
The proof is essentially the same as the other case. Again let us suppress the subscript $i$, and write $V_{\sigma(i)}=V_\sigma$. For each fixed $v\in V$ and $v'\in V_\sigma$, define the $F$-linear form on $A$ by
\[
a\mapsto \la av, v'\ra.
\]
Then there exists a unique element $\lla v,v'\rra\in A$ such that
\[
\la av, v'\ra=\tr_{A/F}(a\lla v,v'\rra).
\]
The assignment $\lla-,-\rra:V\times V_\sigma\to A$ is indeed a nondegenerate $A$-bilinear pairing.
\end{proof}

In the above case, let us write
\[
X_i=V_i\qand X_i^*=V_{\sigma(i)}.
\]
It should be noted that we have the natural isomorphism
\[
\Hom_{A_i}(X_i, A_i)\xrightarrow{\;\sim\;}\Hom_F(X_i, F),\quad \ell\mapsto \tr_{A_i/F}\circ\ell,
\]
of $F$-vector spaces. Hence the dual $X_i^*$ can be interpreted either over $F$ or over $A_i$.

Now, by re-choosing the indices we can write
\[
V=(X_1\oplus X_1^*)\oplus\cdots\oplus (X_\ell\oplus X_\ell^*)\oplus V_{\ell+1}\oplus\cdots\oplus V_m\oplus V_+\oplus V_-,
\]
and
\[
A=B_1\times\cdots \times B_\ell\times A_{\ell+1}\times\cdots\times A_m\times A_+\times A_-,
\]
where
\begin{enumerate}[(a)]
\item for $i=1,\dots, \ell$, we have $B_i=A_i\times A_i$, and $X_i$ is an $A_i$-vector space and $X_i^*$ its dual,
\item for $i=\ell+1,\dots, m$, we have that $A_i$ is a field and $V_i$ is equipped with a Hermitian form over $A_i$, and
\item $A_{\pm}=F$ and $V_{\pm}$ is a nondegenerate quadratic subspace of $V$.
\end{enumerate}

Our involution $\sigma$ on $A$ restricts an involution on each $B_i$, and we write
\[
\sigma=\sigma_1\otimes\cdots\otimes\sigma_m\otimes\sigma_+\otimes\sigma_-,
\]
where on $B_i=A_i\times A_i$ the involution $\sigma_i$ switches the two factors, on $B_i=A_i$ the involution $\sigma_i$ is of the second kind and on $A_{\pm}$ the involution $\sigma_{\pm}$ is trivial.

If we view our $h$ as an element of $A$, we can write
\[
h=(h_1,\cdots,h_m, h_+, h_-),
\]
where $h_i\in B_i$ and $h_{\pm}=1_{V_{\pm}}$. Recall by our convention that if $B_i=A_i\times A_i$ then
\[
h_i=(h_i, h_i^{-1}).
\]
Then
\begin{align*}
\sigma(h)=h^{-1}&=(h_1^{-1},\cdots, h_m^{-1}, h_+, h_-)\\
&=(\sigma_1(h_1),\cdots,\sigma_m(h_m),\sigma_+(h_+), \sigma_-(h_-)),
\end{align*}
where if $B_i=A_i\times A_i$ then $\sigma_i(h_i)$ is actually
\[
\sigma_i(h_i, h_i^{-1})=(h_i^{-1}, h_i),
\]
because $\sigma_i$ switches the two factors of $A_i\times A_i$. Also note that if $B_i=A_i\neq A_{\pm}$ then $\sigma_i$ is a Galois conjugation, and hence $h_i\in A_i$ is such that
\[
\sigma_i(h_i)=h_i^{-1},
\]
namely $h_i$ is a norm one element in $A_i$.

We then have the following.
\begin{prop}
The centralizer $\OO(V)_h$ is of the form
\[
\OO(V)_h\simeq G_1\times\cdots\times G_m\times\OO(V_+)\times\OO(V_-),
\]
where
\[
G_i=\begin{cases}\GL_{A_i}(X_i), &\text{if $B_i=A_i\times A_i$};\\
U_{A_i}(V_i),&\text{if $B_i=A_i$}.\end{cases}
\]
Here by $\GL_{A_i}(V_i)$ we actually mean the ``diagonal"
\[
\GL_{A_i}(X_i)\simeq\{(g_i, {g_i^*}^{-1})\st g_i\in\GL_{A_i}(X_i)\}\subseteq \GL_{A_i}(X_i)\times \GL_{A_i}(X_i^*),
\]
where $g_i^*$ is the adjoint of $g_i$ with resect to the canonical pairing $X_i\times X_i^*\to A_i$, and by $U_{A_i}(V_i)$ we mean the unitary group for the Hermitian space $V_i$ over $A_i$.

Further, if $B_i=A_i\times A_i$ then each $h_i=(h_i, h_i^{-1})$ is viewed as the central element $h_iI_{X_i}$ of $\GL_{A_i}(X_i)$, and if $B_i=A_i$ (including $A_{\pm}$) then each $h_i$ is the central element $h_iI_{V_i}$ of $U_{A_i}(V_i)$.
\end{prop}
\begin{proof}
Let $g\in\OO(V)_h$. Assume $B_i=A_i\times A_i$. Since $V_i=\ker p_i(h)$, one can readily see that $g$ preserves each of the spaces $X_i$ and $X_i^*$. Let $g_i$ be the restriction of $g$ on $X_i$ and $g_i'$ that on $X_i^*$. Note that at this point, $g_i$ and $g_i'$ are only $F$-linear.

Then $(g_i, g_i')\in \GL_F(X_i)\times\GL_F(X_i^*)$ commutes with $h$ if and only if $g_i$ and $g_i'$ are $A_i$-linear because $A_i$ is the field $F[x]\slash (p_i(x))$ which acts via the evaluation at $x=h$. Further $(g_i, g_i')$ preserves the original form $\la-,-\ra$ if and only if
\[
\lla g_i v_i, g_i'v_i\rra_i=\lla v_i, v_i'\rra
\]
for all $v_i\in V_i$ and $v_i^*\in V_i^*$, where $\lla-,-\rra_i$ is the canonical pairing. Hence we must have $g_i'={g_i^*}^{-1}$, where $g_i^*$ is the adjoint of $g_i$ with respect to $\lla-,-\rra_i$. This shows that the set of all $(g_i, g_i')$ commuting with $h$ is of the form
\[
\{(g_i, {g_i^*}^{-1})\st g_i\in\GL_{A_i}(X_i)\},
\]
which is isomorphic to $\GL_{A_i}(X_i)$.

Assume $B_i=A_i$ (including $A_{\pm}$). Then one can see that $g$ preserves the space $V_i=\ker p_i(h)$. Let $g_i\in \GL_F(V_i)$ be the restriction of $g$ to $V_i$. Then $g_i$ commutes with $h$ if and only if $g_i$ is $A_i$-linear. Also $g_i$ preserves the original form $\la-,-\ra$ if and only if it preserves the form $\lla-,-\rra_i$. This shows that $g_i\in U_{A_i}(V_i)$.
\end{proof}

One can see that this is precisely Proposition \ref{prop:centralizer_of_h}.

\quad

Let us mention that if $B_i=A_i\times A_i$ then by Lemma \ref{lemma:appendix_GL} we know that $\GL_{A_i}(X_i)$ is in the Siegel Levi of the special orthogonal group $\SO(X_i\oplus X_i^*)$, and in particular
\[
\GL_{A_i}(X_i)\subseteq \SO(X_i\oplus X_i^*).
\]
If $B_i=A_i$ but not equal to $A_{\pm}$, then by Lemma \ref{lemma:appendix_unitary} we have
\[
U_{A_i}(V_i)\subseteq\SO(V_i).
\]
Note that $U_{A_i}(V_i)$ is in the special orthogonal group $\SO(V_i)$ instead of just the orthogonal group $\OO(V_i)$ because the unitary group $U_{A_i}(V_i)$ is connected.

\quad

Finally, let us mention the $\SO(V)$-analogue of the above proposition, whose proof is left to the reader.
\begin{prop}\label{prop:centralizer_h_SO}
Keep the above notation. Let $h\in\SO(V)$ be semisimple. The centralizer $\SO(V)_h$ is of the form
\[
\SO(V)_h\simeq G_1\times\cdots\times G_m\times S\big(\OO(V_+)\times\OO(V_-)\big),
\]
where
\[
G_i=\begin{cases}\GL_{A_i}(X_i), &\text{if $B_i=A_i\times A_i$};\\
U_{A_i}(V_i),&\text{if $B_i=A_i$},\end{cases}
\]
and
\[
S\big(\OO(V_+)\times\OO(V_-)\big)=\big(\OO(V_+)\times\OO(V_-)\big)\cap\SO(V_+\oplus V_-),
\]
and further $\dim_FV_-$ is always even.
\end{prop}

\section{Summary of involutions}\label{Appendix_B}
In this appendix, we summarize the involutions we use in this paper.\\

\noindent{\bf Canonical involution $g^*$}: For $g\in\GPin(V)$, the canonical involution $g^*$ is defined by reversing the order of the vectors that appear in $g$ viewed in the Clifford algebra $C(V)$; namely if we write $g=v_1v_2\cdots v_\ell$, where $v_i\in V$, then
\[
g^*=(v_1v_2\cdots v_\ell)^*=v_\ell v_{\ell-1}\cdots v_1.
\]

\quad

\noindent{\bf Clifford involution $\overline{g}$}: For $g\in\GPin(V)$, the Clifford involution $\overline{g}$ is defined as the ``signed canonical involution"; namely for $g=v_1v_2\cdots v_\ell$,
\[
\overline{g}=(-1)^\ell(v_1v_2\cdots v_\ell)^*=(-1)^\ell v_\ell v_{\ell-1}\cdots v_1.
\]
In other words
\[
\overline{g}=\sign(g)g^*,
\]
where $\sign:\GPin(V)\to\{\pm 1\}$ is the sign map that sends the nonidentity component to $-1$. Note that the Clifford norm $N:\GPin(V)\to F^\times$ is defined by $N(g)=g\,\overline{g}$ so that
\[
g^{-1}=\frac{1}{N(g)}\overline{g}.
\]

\quad

\noindent{\bf Involution $\sigma_V$}: The involution $\sigma_V$ is the involution on $\GPin(V)$ defined by
\[
\sigma_V(g)=\begin{cases}g^*&\text{if $n=2k$};\\
\sign(g)^{k+1}g^*&\text{if $n=2k-1$}.\end{cases}
\]
The important property of $\sigma_V$ is that it preserves the semisimple conjugacy classes of $\GPin(V)$. This, in particular, implies $\pi^\vee\simeq\pi^\sigma$ for all $\pi\in\Irr(\GPin(V))$, where $\pi^{\sigma}(g):=\pi(\sigma_V(g)^{-1})$. Also, this property allows us to reduce the vanishing of invariant distributions to semisimple orbits by using Bernstein's localization principle.
\quad

\begin{rem}
All the three involutions (canonical, Clifford and $\sigma_V$) are equal on $\GSpin(V)$.
\end{rem}

\quad

\noindent{\bf Involution $e^{k}\sigma_V(g)e^{-k}$}: The involution $g\mapsto e^k\sigma_V(g) e^{-k}$ on $\GPin(V)$ is also defined. This involution preserves the semisimple conjugacy classes of $\GSpin(V)$, and hence plays the same role as $\sigma_V$ of the $\GPin(V)$ case.

\quad

\noindent{\bf Group $\GPint(V)$ and Involution $\tau_W$}: The group $\GPint(V)$ is defined as
\[
\GPint(V)=\big\la g,\,\beta\st g\in\GPin(V)\big\ra
\]
with the relations $g\beta=\beta g$ and $\beta^2=1$, namely
\[
\GPint(V)=\GPin(V)\times\{1,\beta\}.
\]
The action of $\GPint(V)$ on the set $\GPin(V)\times V$ is defined as in \eqref{eq:action_on_V_GPin}. In particular, $\beta$ acts on $\GPin(V)$ via the involution $\sigma_V$.

Assume $V=W\oplus Fe$, where $e$ is anisotropic. We set
\[
\GPint(W):=\GPint(V)_e=\big\la g, e\beta\st g\in \GPin(W)\big\ra,
\]
so that
\[
\GPint(W)\simeq\GPin(W)\rtimes\{1, e\beta\}.
\]
The involution $\tau_W$ on $\GPin(V)$ is defined by $\tau_W(g)=e\sigma_V(g) e^{-1}$, and the element $e\beta$ acts on $\GPin(V)$ via this involution.

Since $\tau_W(\GPin(W))=\GPin(W)$, the involution $\tau_W$ acts on the space
\[
\Scal'(\GPin(V))^{\GPin(W)}
\]
of the $\GPin(V)$ invariant distributions. We showed that the $-1$-eigenspace of the involution $\tau_W$ vanishes, which is equivalent to the assertion
\[
\Scal'(\GPin(V))^{\GPint(W),\chi}=0.
\]
However, we reduce this vanishing assertion to
\[
\Scal'(\GPin(V)\times V)^{\GPint(V),\chi}=0,
\]
where the space $W$ no longer appears. Hence the involution $\tau_W$ does not play any direct role in our proof.

\quad

\noindent{\bf Group $\GSpint(V)$ and Involution $\tau_W$}: The group $\GSpint(V)$ is defined as
\[
\GSpint(V)=\big\la g,\,e^k\beta\st g\in\GPin(V)\big\ra\subseteq\GPint(V),
\]
so that
\[
\GSpint(V)\simeq\GSpin(V)\rtimes\{1, e^k\beta\}.
\]
The action of $\GSpint(V)$ on the set $\GSpin(V)\times V$ is simply the restriction of the action of $\GPint(V)$ as in \eqref{eq:action_of_GSpint}. In particular, $e^k\beta$ acts on $\GSpin(V)$ via the involution $g\mapsto e^k\sigma_V(g)e^{-k}$, which preserves the semisimple conjugacy classes of $\GSpin(V)$.

Assume $V=W\oplus Fe$, where $e$ is anisotropic and fix an orthogonal basis $\{e_1,\dots, e_{n-1}\}$ of $W$. We set
\[
\GSpint(W)\simeq\GSpint(V)_e=\big\la g, e_{n-1}^{k-1}e\beta\st g\in \GSpin(W)\big\ra,
\]
so that
\[
\GSpint(W)=\GSpin(W)\rtimes\{1, e_{n-1}^{k-1}e\beta\}.
\]
The involution $\tau_W$ on $\GSpin(V)$ is defined by
\[
\tau_W(g)=(e_{n-1}^{k-1}e)\sigma_V(g) (e_{n-1}^{k-1}e)^{-1},
\]
and the element $e_{n-1}^{k-1}e\beta$ acts on $\GPin(V)$ via this involution. This involution plays the same role as the $\tau_W$ of the $\GPin$ case.

Our main theorem can be shown by showing the vanishing assertion
\[
\Scal'(\GSpin(V))^{\GSpint(W),\chi}=0
\]
just as the $\GPin$ case.

\quad

\noindent{\bf Involution $\tau_V$ on classical groups}: Let $G(V)=\GL(V), \UU(V), \OO(V)$ or $\SO(V)$. The involution $\tau_V$ on $G(V)$ is defined as follows:
\[
\tau_V(g)=
\begin{cases}
g^t&\text{for $\GL(V)$};\\
\beta g^{-1}\beta &\text{for $\UU(V)$};\\
g^{-1}&\text{for $\OO(V)$};\\
r_e^{k}g^{-1}r_e^{-k}&\text{for $\SO(V)$},
\end{cases}
\]
where $\beta:V\to V$ for $\UU(V)$ is Galois conjugation and $r_e\in\OO(V)$ is the reflection in the hyperplane orthogonal to $e$. This involution preserves the (semisimple) conjugacy classes of $G(V)$, and hence plays the same role as our $\sigma_V$ for the $\GPin(V)$ case.

The groups $\Gt(V)$ and $\Gt(W)$ and the involution $\tau_W$ are defined similarly to the $\GPin(V)$ case.

\begin{rem}
It should be pointed out here that one can show $\tau_V(g)$ and $g$ are conjugate in $G(V)$ not just for semisimple $g\in G(V)$ but for all $g\in G(V)$. (See \cite[I.2. Proposition, p.79]{MVW}.) This allows one to prove the assertion on contragredient without using Harish-Chandra's regularity theorem. Thus the existence of MVW-involution can be shown for $G(V)$ even when the characteristic of $F$ is not zero.
\end{rem}

%--------------------------------------------bibliography------------------------------------
\bibliographystyle{alpha}
%{amsalpha}
%\addcontentsline{toc}{section}{Bibliography}
\bibliography{MultOne}

\end{document}